\pdfoutput=1 

\RequirePackage{fix-cm}
\documentclass[smallextended]{svjour3}
\usepackage[utf8]{inputenc}
\usepackage{booktabs}
\usepackage{amssymb,amsmath,epsfig,graphics,psfrag,graphicx,color,subfigure}
\usepackage[title]{appendix}
\usepackage{multirow}
\definecolor{lightblue}{rgb}{0.22,0.45,0.70}
\definecolor{lightgreen}{rgb}{0.22,0.50,0.25}
\usepackage[right = 2.5cm, left=2.5cm, top = 2.5cm, bottom =2.5cm]{geometry}
\usepackage[colorlinks=true,breaklinks=true,linkcolor=lightblue,citecolor=lightblue]{hyperref}

\definecolor{darkred}{rgb}{0.82,0.15,0.20}
\definecolor{darkblue}{rgb}{0.82,0.15,0.12}

\renewenvironment{proof}{\noindent{\it Proof.}}{\hfill$\square$}

\numberwithin{equation}{section}
\numberwithin{figure}{section}
\numberwithin{table}{section}
\numberwithin{lemma}{section}
\numberwithin{corollary}{section}
\numberwithin{theorem}{section}
\numberwithin{remark}{section}

\newcommand\cero{\boldsymbol{0}}

\newcommand\bI{\mathbf{I}}

\newcommand\bV{\mathbf{V}}

\newcommand{\bZ}{\mathbf{Z}}
\newcommand\bPi{\boldsymbol{\Pi}}

\newcommand\beps{\boldsymbol{\varepsilon}}
\newcommand\ff{\boldsymbol{f}}

\newcommand\bg{\boldsymbol{g}}
\newcommand\nn{\boldsymbol{n}}
\newcommand\br{\boldsymbol{r}}

\newcommand\bu{\boldsymbol{u}}
\newcommand\bv{\boldsymbol{v}}
\newcommand\bw{\boldsymbol{w}}

\newcommand{\cA}{\mathcal{A}}
\newcommand{\cL}{\mathcal{L}}
\newcommand\cE{\mathcal{E}}

\newcommand\bbK{\mathbb{K}}
\newcommand{\bbV}{\mathbb{V}}
\newcommand\bbP{\mathbb{P}}

\newcommand\cT{\mathcal{T}}

\newcommand{\cN}{\mathcal{N}}

\newcommand{\norm}[1]{\left\|#1\right\|}

\newcommand{\nV}[1]{\vertiii{#1}}

\newcommand{\g}{\Gamma}

\newcommand\vdiv{\mathop{\mathrm{div}}\nolimits}
\newcommand\bt{\boldsymbol{t}}

\newcommand\bnabla{\boldsymbol{\nabla}}
\newcommand\bDelta{\boldsymbol{\Delta}}

\newcommand{\ds}{\,\mathrm{d}s}

\newcommand{\vertiii}[1]{{\left\vert\kern-0.25ex\left\vert\kern-0.25ex\left\vert #1 
		\right\vert\kern-0.25ex\right\vert\kern-0.25ex\right\vert}}

\allowdisplaybreaks

\begin{document}
	\titlerunning{Nitsche stabilized VEM for a Brinkman problem with mixed boundary conditions}
	\title{Nitsche stabilized Virtual element approximations for a Brinkman problem with mixed boundary conditions}
	\authorrunning{Mora, Vellojin, Verma}
	\author{David Mora \and Jesus Vellojin  \and Nitesh Verma}
	\institute{
		David Mora \at GIMNAP, Departamento de Matem\'atica, Universidad del B\'io-B\'io, Concepci\'on, Chile; and CI$^2$MA,
		Universidad de Concepci\' on, Chile.\\
		\email{dmora@ubiobio.cl}.
		\and
		Jesus Vellojin \at
		GIMNAP, Departamento de Matem\'atica, Universidad del B\'io-B\'io, Concepci\'on, Chile.\\
		\email{jvellojin@ubiobio.cl}.
		\and
		Nitesh Verma  (corresponding author) \at
		GIMNAP, Departamento de Matem\'atica, Universidad del B\'io-B\'io, Concepci\'on, Chile.\\
		\email{nverma@ubiobio.cl}
	}
	\date{}
	\maketitle
	
	\begin{abstract}
		In this paper, we formulate, analyse and implement the discrete formulation of the Brinkman problem with mixed boundary conditions, including slip boundary condition, using the Nitsche's technique for virtual element methods. {The divergence conforming virtual element spaces for the velocity function and piecewise polynomials for pressure are approached for the discrete scheme. We derive a robust stability analysis of the Nitsche stabilized discrete scheme for this model problem. We establish an optimal and vigorous a priori error estimates of the discrete scheme with constants independent of the viscosity.} Moreover, a set of numerical tests demonstrates the robustness with respect to the physical parameters and verifies the derived convergence results.
	\end{abstract}
	
	\keywords{Brinkman equation \and slip boundary condition \and virtual element methods \and Nitsche method \and a priori error analysis \and numerical experiments.}
	\subclass{}

	\section{Introduction and problem statement}
We are interested in the numerical approximation by the virtual element method
of the Brinkman system with mixed boundary conditions, that is, Dirichlet conditions
on one part on the boundary and slip conditions on the rest of the boundary. These boundary conditions, representing the inflow/outflow in domain, and flow through the boundary wall as well fluid slipping along the boundary wall, have been introduced in the several applications such as water treatment, reverse osmosis and so on.
The Brinkman equations can be seen as an extension of Darcy’s law to describe
the laminar flow behavior of a viscous fluid within a porous material
with possibly heterogeneous permeability, so that the flow is dominated
by Darcy regime on a part of the domain and by Stokes on the other parts of the domain.
In the last years, the numerical solution of this system has acquired great
interest due to high practical importance in different areas,
including several industrial and environmental applications, such as,
filtering porous layers, oil reservoirs, the study of foams, among others.

The virtual element method (VEM) introduced in \cite{daveiga-b13}, belong to the
so-called polytopal methods for solving PDEs by using general polygonal/polyhedral meshes.
These methods have received substantial attention in the last years, for instance,
hybrid high order method \cite{HHO1,HHObook,droniou2,droniou3},
discontinuous Galerkin method \cite{Cangiani14,paola2021,zhao1},
mimetic finite difference method \cite{Mimetic2014},
virtual element method \cite{huang1,MR4510898,MR4586821}
The VEM can be seen as an extension of Finite Elements
Method (FEM) to polygonal or polyhedral decompositions.
The VEM has been applied for different problems in fluid mechanics, see for instance,
\cite{ABMV2014,BMVjsc19,BLVm2an17,BLVsinum18,CGima17,CGSm3as17,GSm3as21,GMScalcolo18,Mora21imajna,vacca17}
and the references therein.

In \cite{GMScalcolo18}, virtual element method
for a pseudostress-velocity formulation for the nonlinear Brinkman flow has been introduced.
The stream virtual element method has been introduced and analyzed in \cite{Mora21imajna}.
In this case, the problem is written as a single
equation in terms of the stream function of the velocity field by using the incompressibility condition, and optimal error estimates, independent of the viscosity parameter, are obtained.
In \cite{BLVm2an17,BLVsinum18} the authors have introduced a novel divergence-free
virtual element method for solving the Stokes and Navier-Stokes problems.
This method has been used in \cite{vacca17} for solving the Brinkman problem
with homogeneous Dirichlet boundary conditions, and the error estimates
independent of the model parameters are written.
Moreover, we also mention recent works for the numerical discretization
of the Brinkman problem by finite element method \cite{anaya16,anaya17,meddahi1,fila1}.

The so-called Nitsche methods can be regarded
as a stabilization technique where some terms are added to the
variational formulation so that the boundary conditions can be incorporated in
a weak form. One of the main advantages of the Nitsche method is its versatility. It can be applied to a wide range of partial differential equations, including elliptic, parabolic, and hyperbolic equations.  For instance,
finite element discretizations with Nitsche method has evolved to handle general boundary conditions \cite{juntunen2009nitsche}, including interface problems \cite{hansbo2005nitsche}, unilateral and frictional contact \cite{chouly2013nitsche,chouly2018unbiased}, or membrane filtration processes \cite{carro_ijnmf24}, among others. On each reference, we observe that a properly formulated penalty terms in the Nitsche method ensure consistency and stability of the numerical solution, even for complex problems and irregular domains. This stability property is crucial for obtaining reliable and accurate results in practical simulations. Moreover, the Nitsche method can be relatively straightforward to implement compared to other approaches for enforcing boundary conditions. Its penalty-based formulation simplifies the incorporation of boundary conditions into the variational formulation of the problem. By the nature of the method, this can be seamlessly integrated with modern numerical techniques such as adaptive mesh refinement, parallel computing, higher-order finite elements, or virtual elements method. In the latter, the literature regarding Nitsche's method is scarce, with a recent contribution by \cite{bertoluzza2022weakly}. Here, the authors study the extension to virtual elements of  the Lagrange multiplier method, in its stabilized formulation as proposed by Barbosa and Hughes \cite{barbosa1991finite}, and the Nitsche method \cite{nitsche91}.  They  proved stability and optimal error estimates, under custom conditions on the stabilization parameters. The results are extended for two and three dimensional geometries with curved domains, where the numerical experiments assess the performance of the method, suggesting a viable alternative to the corresponding scheme in the finite element method.

In the present contribution we propose a Nitsche method
for the Brinkman system with mixed boundary conditions.
We consider Dirichlet condition on a part of the boundary and
slip conditions on the rest of the boundary. {This approach have been employed in several Navier-Stokes models to impose a slip boundary condition. We have for example the work from \cite{urquiza2014weak}, where they compare the Lagrange multiplier and Nitsche's method for the weakly imposition of the slip boundary condition on curved domains. Recently, in  \cite{gjerde2022nitsche} the authors study the weak imposition of the slip-boundary through Nitsche's method and projected normals between the computational and the continuous domain, and also explore the well-known  Babu\v{s}ka-Sapondzhyan Paradox. The type
of conditions behind the models presented in the aforementioned works are important in different areas for fluid flow problems}. 
In particular, these boundary conditions appears naturally in the analysis of numerical methods for desalinization processes, filtration, among others.
For that reason, in this work, we extend the results presented
in \cite{vacca17} for the new model problem. {Unlike the Dirichlet boundary conditions, the slip-boundary conditions are inadequate to impose  strongly for discrete solution. This drawback is due to the unavailability of the degrees of freedom for the normal component and normal derivative of the function on the boundary in the discrete space.}
It is well know that different strategies can be considered
for imposing mixed conditions, including slip conditions, such as, Lagrange multiplier method.
However, we here propose a symmetric discrete variational formulation
by adding Nitsche terms in order to incorporate the boundary conditions
considered in the model problem. Moreover, we discretize
the problem, by using the virtual element method presented in \cite{BLVm2an17}
for the Stokes problem. {We define a new Nitsche term to impose the slip boundary condition for the discrete scheme using the piecewise polynomial projection on each polygon $K$ (in the mesh) on the boundary.} We establish that the discrete problem is well posed
by proving a global inf-sup condition and we write stability by using
appropriate mesh-parameter dependent norms. Under rather mild assumptions on the polygonal meshes
and by using interpolation estimates, the convergence rate is proved to
be optimal in terms of the mesh size $h$. {We would like to emphasis that the constants in the derivation of error estimates are independent of the physical parameters in this model problem. The trace inequality for the piecewise polynomials are utilised in the analysis in order to prove the stability.}
In summary, the advantages of the proposed Nitsche VEM method
for the Brinkman problem are, on the one hand,
the possibility to use general polygonal meshes,
and on the other hand, the possibility of an easy imposition
of general boundary conditions, including non homogeneous
Dirichlet conditions, slip conditions, among others.

The rest of the paper is organized as follows: in Section~\ref{sec:model}
we introduce the variational formulation of the Brinkman equations
with mixed boundary conditions.
In Section~\ref{sec:VEapprox} we present the virtual element discretization
of arbitrary order $k\ge2$ with Nitsche's technique.
The existence, uniqueness and stability results of
the discrete formulation by using a global inf-sup condition is presented
in Section~\ref{solvab}. In Section~\ref{sec:abstract-error-analysis} we obtain error estimates for the
velocity and pressure fields.
Section~\ref{sec:numerical-section}
is devoted to analyze, through several numerical experiments, the performance
and robustness of the proposed Nitsche method.

In this article, we will employ standard notations for Sobolev
spaces, norms and seminorms. In addition, we will denote by $C$ a
generic constant independent of the mesh parameter $h$ and model parameters,
which may take different values in different occurrences.

\section{Governing equations}\label{sec:model}

Let $\Omega$ be an open, bounded subset of $\mathbb{R}^2$ having Lipschitz–continuous boundary,
such that $\partial \Omega = \bar{\g}_D \cup \bar{\g}_N$, and $\g_D \cap \g_N = \emptyset$. {The boundary subdomain $\Gamma_D$ represents a part of $\partial\Omega$ where a fixed value for the velocity (Dirichlet boundary) is given, while $\Gamma_N$ denotes a subdomain where we have a slip boundary condition. An example of such domain is depicted in  Figure \ref{fig:brinkmann-domain}}.
Thus, the system of interest can be written as the following problem.
Given the body force $\ff \in [L^2(\Omega)]^2$ and boundary condition $\bg \in [H^{1/2}(\g_D)]^2$, find the fluid velocity $\bu$ and fluid pressure $p$, such that

\begin{align}
		\bbK^{-1}\bu - \nu \vdiv (\beps (\bu)) + \nabla p & = \ff & \text{in $\Omega$},\label{eq:Brinkman1}\\
		 \nabla \cdot \bu &= 0 & \text{in $\Omega$},\label{eq:Brinkman2}\\
		\bu &= \bg& \text{on $\g_D$}, \label{bc:Sigma} \\
		(\beps(\bu) \nn) \cdot \bt =0,  \quad \text{and} \quad \bu \cdot \nn &= 0 & \text{on $\g_N$}, \label{bc:Gamma}
	\end{align}
where $\nu$ is the viscosity of
the fluid, $\beps (\bv):=\frac{1}{2}(\nabla \bv+(\nabla \bv)^t)$ is the symmetric derivative,
$\mathbb{K}$ is a bounded, symmetric, and positive definite
tensor describing the permeability properties of the Brinkman region,
and {$\nn:=(n_1,n_2)^{\texttt{t}}$ and $\bt$ are unit normal and tangent on the boundary $\g_N$, respectively, with $\bt:=(-n_2,n_1)^{\texttt{t}}$}.

\begin{figure}[!hbt]\centering
\includegraphics[scale=1]{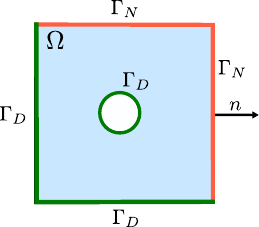}
\caption{Sample geometry of the considered Brinkman domain with mixed boundary conditions.}
\label{fig:brinkmann-domain}
\end{figure}


In order to obtain the weak form of the governing equations,
let us introduce the functional spaces for the velocity and pressure.
\begin{align*}
	\bV:= \{ \bv \in [H^1(\Omega)]^2: \bv \cdot \nn =0 \text{ on } \g_N, \bv = \cero \text{ on } \g_D\} \qquad \text{and} \qquad Q:= L^2_0(\Omega).
\end{align*}

We first state the assumptions on the physical parameters used throughout this paper.
The permeability tensor holds the positive definiteness
for all $\bv \in \bV$, that is there exist $\kappa_1, \kappa_2 >0$ such that
$$ 0 < \kappa_1 \| \bv \|_{0,\Omega}^2 \le (\mathbb{K}^{-1} \bv , \bv)_{0,\Omega} \le  \kappa_2 \| \bv \|_{0,\Omega}^2 < \infty,$$
and we also assume that $0 < \nu \le 1$.

Next, we test the corresponding equations in \eqref{eq:Brinkman1}-\eqref{eq:Brinkman2}
and using integration by parts and the fact that $\bv \in \bV$, we have
	\begin{align*}
		\int_{\Omega} \mathbb{K}^{-1} \bu \cdot \bv + \nu \int_{\Omega} \beps(\bu) :\beps(\bv) - \int_{\Omega} p \vdiv \bv -   \int_{\g_N} (\nu \beps(\bu) - p \mathbb{I}) \nn \cdot \bv \ds& = \int_{\Omega} \ff \cdot \bv\qquad\forall\bv\in\bV, &\\
		- \int_{\Omega} \vdiv \bu\,  q &= 0\qquad \qquad \quad \forall q\in Q. &
	\end{align*}
	{To impose the Dirichlet condition for the continuous solution $\bu$, we define the Sobolev space $$\bV_g:=\{ \bv \in\bV:\bv \cdot \nn =0 \text{ on } \g_N, \bv = \bg \text{ on } \g_D\}.$$}
\noindent The boundary term in the weak formulation for $\bu \in \bV_g$, $\bv \in \bV$
 and $p \in Q$ can be computed by splitting the terms
 in terms of tangential and normal components ($\bw= (\bw \cdot \nn) \nn + (\bw \cdot \bt) \bt$) as,
\begin{align*}
	\int_{\g_N} (\nu \beps(\bu) - p \mathbb{I}) \nn \cdot \bv \ds & = 
\int_{\g_N} \Big( (\nn^t(\nu \beps(\bu) \nn)) (\bv \cdot \nn) + (\bt^t(\nu \beps(\bu) \nn)) (\bv \cdot \bt) - p (\bv \cdot \nn) \Big) \ds \\
	& = \int_{\g_N}  (\bt^t(\nu \beps(\bu) \nn)) (\bv \cdot \bt)\ds=0,
\end{align*}
where we have also used the boundary condition $\eqref{bc:Gamma}$.

\noindent Thus, the weak formulation of the problem is stated as: find $\bu\in \bV_g$ and $p\in Q$
such that
\begin{subequations} \label{eq:weak}
	\begin{align} 
		m(\bu, \bv) + a(\bu, \bv) + b(\bv,p) & = F(\bv) & \quad \forall \bv \in \bV,  \label{eq:weak-u}\\
		b(\bu,q) &= 0 & \quad \forall q \in Q,  \label{eq:weak-p}
	\end{align}
\end{subequations}
where the corresponding bilinear forms and the linear functional are introduced as follows
\begin{gather*}
	m(\bu, \bv):= \int_{\Omega} \mathbb{K}^{-1} \bu \cdot \bv, \qquad  a(\bu, \bv) := \nu \int_{\Omega} \beps(\bu) :\beps(\bv), \\
	b(\bv,q):= - \int_{\Omega} q \vdiv \bv, \quad \text{and} \quad F(\bv) := \int_{\Omega} \ff \cdot \bv.
\end{gather*}

\noindent Now, we introduce the parameter dependent norm for all $\bv \in \bV$ as follows:
\begin{gather*}
	\vertiii{\bv}_{\bV}^2:= \|\mathbb{K}^{-1/2} \bv\|_{0,\Omega}^2 +  \nu \| \beps(\bv) \|_{0,\Omega}^2 + \| \vdiv \bv \|_{0,\Omega}^2.
\end{gather*}

\noindent Defining the kernel space as
$$\bZ:= \{ \bv \in \bV: b(\bv, q)=0 \quad \forall q \in Q \} = \{ \bv \in \bV: \vdiv \bv = 0 \},  $$
then we have the following properties of the bilinear forms:
\begin{gather*}
	m(\bu, \bv)  \le \nV{\bu}_{\bV}  \vertiii{\bv}_{\bV}, \qquad 	a(\bu, \bv)  \le \vertiii{\bu}_{\bV}  \vertiii{\bv}_{\bV} \quad \forall \bu,  \bv \in \bV, \\
	b(\bv, q)  \le \vertiii{\bv}_{\bV} \| q\|_Q \quad \forall \bv \in \bV, q \in Q, \qquad  F(\bv)  \le \| \ff \|_{\bV^*} \vertiii{\bv}_{\bV} \quad \forall \bv \in \bV, \\ 
	m(\bv, \bv) + a(\bv, \bv) \ge \nV{\bv}_{\bV}^2 \quad \forall \bv \in \bZ, \quad \text{and} \quad
	\sup_{(0 \neq)\bv \in \bV}\frac{b(\bv, q)}{\vertiii{\bv}_{\bV}}  \ge \beta \norm{q}_{Q} \qquad \forall q \in Q,
\end{gather*}
where $\norm{q}_{Q}:=\Vert q\Vert_{0,\Omega}$ and $\|\ff\|_{\bV^*}:= \sup_{\bv  \in \bV}\frac{F(\bv)}{\vertiii{\bv}_{\bV}}.$

Using the ellipticity of $m(\cdot, \cdot) + a(\cdot, \cdot)$ on $\bZ$,
and the inf-sup condition of $b(\cdot,\cdot)$, and continuity of the
bilinear forms of $m(\cdot, \cdot)$, $a(\cdot, \cdot)$, $b(\cdot, \cdot)$
and continuity of the linear functional $F(\cdot)$, we have the well posedness
of the weak formulation~\eqref{eq:weak} (refer  \cite{boffi13,brenner08}).

\begin{lemma}
	The continuous solution $(\bu, p) \in \bV_g \times Q$ of formulation
	\eqref{eq:weak} holds the following bound, for a constant $C$ (independent of $\nu,\, \bbK$),
	\begin{align} \label{est:cts}
		\vertiii{\bu}_{\bV} + \| p\|_Q \le C \|\ff\|_{\bV^*}.
	\end{align}
\end{lemma}

\section{Virtual element approximation} \label{sec:VEapprox}

In this section we construct a VEM for solving the Brinkman problem with mixed
boundary conditions using Nitsche's technique. We start denoting by $\{{\mathcal T}_h\}_h$
a sequence of partitions into polygons of the domain~$\Omega$.
The elements in $\cT_h$ are denoted as $K$, and edges by $e$.
Let $h_K$ denote the diameter of the element $K$ and $h:=\max_{K\in{\mathcal T}_h}h_K$
the maximum of the diameters of all the elements of the mesh.
By~$N^v_K$ we  denote the number of vertices in the polygon~$K$,
$N^e_K$~stands for the number of edges on~$\partial K$,
and $e$~is a generic edge of $\mathcal{T}_h$.
We denote by~$\boldsymbol{n}_K^e$ the unit normal pointing outwards~$K$
and by  $\bt^e_K$ the unit tangent vector along~$e$ on~$K$,
and $V_i$~represents the $i^{th}$ vertex of the polygon~$K$.

Next, we denote the sets of all boundary edges as $\cE_h(\Omega)$ and denote
the edges on the Dirichlet boundary as $\cE_h^{D}:= \g_D \cap \cE_h(\Omega)$
and edges on the other boundary as $\cE_h^{N} := \g_N \cap \cE_h(\Omega)$.

In addition, for the theoretical analysis we will make the following assumptions:
there exists $C_{{\mathcal T}}>0$ such that, for every $h$ and every $K\in {\mathcal T}_h$,
(a) the ratio between the shortest edge
and $h_K$~is larger than $C_{{\mathcal T}}$; and (b) $K\in{\mathcal T}_h$
is star-shaped with respect to every point within a  ball of radius~$C_{{\mathcal T}}h_K$.

	In what follows, we denote by $\mathbb{P}_k(K)$ the space of
	polynomials of degree up to $k$, defined locally on $K\in\cT_h$.
Moreover, we denote by $\mathcal{G}_{k}(K):=\nabla(\mathbb{P}_{k+1}(K))\subseteq[\mathbb{P}_k(K)]^2$
and $\mathcal{G}_{k}^{\perp}(K)\subseteq[\mathbb{P}_{k}(K)]^{2}$ the $L^2$-orthogonal complement to
$\mathcal{G}_{k}(K)$.


\subsection{Discrete spaces and degrees of freedom} \label{subsec:VEspaces}

The local virtual element spaces are defined for $k \ge 2$ as,
\begin{align*}
	\tilde{\bV}_h^k(K):= \{ \bv \in [H^1(K) \cap C^0(\partial K)]^2: &- \bDelta \bv - \nabla s \in \mathcal{G}_{k-2}(K)^{\perp}(\text{for some } s \in L^2(K)), \quad \vdiv \bv \in \mathbb{P}_{k-1}(K), \\
	& \quad \bv|_{e} \in \mathbb{P}_k(e) \, \forall e \in \partial K\}, \\
	\quad Q_h^k(K) := \mathbb{P}_{k-1}(K). 
\end{align*}

Denote the dimension of local space $\tilde{\bV}_h^k(K)$
as $N_{\bV}^K$, and dimension of $Q_h^k(K)$ as $N_Q^K$.

The local degrees of freedom $( \text{dof}_i^{\bV},\, 1 \le i \le N_{\bV}^K)$
for space $\tilde{\bV}_h^k(K)$ for a generic $\bv_h\in\tilde{\bV}_h^k(K)$ are given by
\begin{itemize}
	\item the values of $\bv_h$ at the vertices of the polygon $K$,
	\item the values of $\bv_h$ at $k-1$ points in the interior of edges $e \in \partial K$,
	\item the moments of $\bv_h$
	$$\int_K \bv_h \cdot \bg_{k-2}^{\perp} \quad \text{ for all } \bg_{k-2}^{\perp} \in \mathcal{G}_{k-2}(K)^{\perp},$$
	\item for $k >1$, the moments of $\vdiv \bv_h$ in $K$,
	$$\int_K \vdiv \bv_h\, q_{k-1} \quad \text{ for all } q_{k-1} \in \mathbb{P}_{k-1}(K),$$
\end{itemize}
And for the space $Q_h^k(K)$, the degrees of freedom are 
\begin{itemize}
	\item the moments of $q_h$
	$$\int_K q_h \,q_{k-1} \quad \text{ for all } q_{k-1} \in \mathbb{P}_{k-1}(K).$$
\end{itemize}
\noindent Denoting the bilinear forms on each element as $a^K(\cdot, \cdot):= a(\cdot, \cdot)|_K$ for any bilinear form, we introduce the energy projection operators $\bPi^{\bnabla, k}_K: [H^1(K)]^2 \to [\mathbb{P}_{k}(K)]^2, \bPi^{\beps, k}_K: [H^1(K)]^2 \to [\mathbb{P}_{k}(K)]^2,$ and $L^2$ projection operator $ \bPi^{\cero,k}_K: [L^2(K)]^2 \to [\mathbb{P}_{k}(K)]^2$ for all $K\in \cT_h$ to define the computable bilinear forms,
\begin{align}
	\label{Def:Pi_0}& \quad (\bPi^{\cero,k}_K \bv - \bv, \br_k)_{0,K} =0  \qquad \forall \bv \in [L^2(K)]^2, \br_k \in [\mathbb{P}_{k}(K)]^2,\\ 
\label{Def:Pi_grad}	\forall \bv \in [H^1(K)]^2, &\begin{cases}
		(\bnabla ( \bPi^{\bnabla, k}_K \bv - \bv), \bnabla \br_k) = 0 \,\quad\forall \br_k \in [\mathbb{P}_{k}(K)]^2, \\ \bPi^{\cero,k}_K (\bPi^{\bnabla, k}_K \bv - \bv) =\cero
	\end{cases} \\
	 \label{Def:Pi_eps}
	\forall \bv \in [H^1(K)]^2, &\begin{cases}
		a^K( \bPi^{\beps, k}_K \bv - \bv, \br_k) = 0 \,\quad\forall \br_k \in [\mathbb{P}_{k}(K)]^2, \\ (\bPi^{\beps, k}_K \bv - \bv, \br_k)_{0, K} =\cero  \quad \forall \br_k \in \ker(a^K(\cdot,\cdot)).
	\end{cases}
\end{align}
Refer to \cite{vacca17}, we utilize the modified virtual
element space for flux with the help of energy projection \eqref{Def:Pi_grad} as,
\begin{align*}
	{\bV}_h^k(K):= \{ \bv \in [H^1(K) \cap C^0(\partial K)]^2: &- \bDelta \bv - \nabla s \in \mathcal{G}_{k}(K)^{\perp} \quad \text{for some } s \in L^2(K), \\
	& \quad \vdiv \bv \in \mathbb{P}_{k-1}(K),  \quad \bv|_{e} \in \mathbb{P}_k(e) \, \forall e \in \partial K, \\
	& \quad (\bPi^{\bnabla,k}_K \bw - \bw, \bg_{k\backslash k-2}^{\perp})_{0,K}=0 \quad \forall \, \bg_{k\backslash k-2}^{\perp} \in \mathcal{G}_{k}(K)^{\perp} \backslash \mathcal{G}_{k-2}(K)^{\perp} \}.
\end{align*}
The degrees of freedom for the space ${\bV}_h^k(K)$ are same as $\tilde{\bV}_h^k(K)$.

\begin{remark}
	The space ${\bV}_h^k(K)$ will be useful for computing
	the $L^2$-projection \eqref{Def:Pi_0} onto $\mathbb{P}_{k}(K)$ to approximate
	the zero order term presented in bilinear form $m(\cdot,\cdot)$.
\end{remark}

\begin{remark}
	The imposition of Dirichlet boundary condition using the interpolant is not enough with complex domains, and the slip boundary conditions are not easy to impose in numerical experiments. Thus, we proceed here with the Nitsche's technique for mixed boundary imposition.
\end{remark}

 We introduce the virtual element spaces to be used
 in the discretization of the Brinkman problem as:
\begin{align*}
	&\bV_h^k:= \{\bv_h \in [H^1(\Omega)]^2: \bv_h|_K \in 	\bV_h^k(K) \, \, \forall K \in \cT_h \}, \qquad   Q_h^k:= \{q_h \in Q: q_h|_K \in 	Q_h^k(K) \, \, \forall K \in \cT_h \}.
\end{align*}

Thus, the discrete formulation states that we seek the discrete velocity $\bu_h \in \bV_h^k$ and discrete fluid pressure $p_h \in Q_h^k$ such that the following holds for all $\bv_h \in 	\bV_h^k$ and $q_h \in Q_h^k$
\begin{align*}
	& \sum_{K \in \cT_h} \Big( m^K(\bPi^{\cero,k}_K \bu_h, \bPi^{\cero,k}_K \bv_h) + \mathcal{S}^{0}_K((\bI - \bPi^{0, k}_K)\bu_h, (\bI - \bPi^{0, k}_K)\bv_h)  \\	
	& \qquad \qquad  + a^K(\bPi^{\beps, k}_K  \bu_h, \bPi^{\beps, k}_K \bv_h) + \mathcal{S}^{\beps}_K((\bI - \bPi^{\beps, k}_K)\bu_h, (\bI - \bPi^{\beps, k}_K)\bv_h)  \Big) \\	
	& \qquad + \sum_{e \in \cE_h^D} \Big(  \gamma_{D}\int_e h_e^{-1}  \bu_h \cdot \bv_h \ds  -  \int_e (\nu \beps( {\Pi^{\beps,k}_{K_e}} \bu_h) \nn) \cdot  \bv_h \ds -  \int_e (\nu \beps( {\Pi^{\beps,k}_{K_e}} \bv_h) \nn) \cdot  \bu_h \ds \Big)  \\
	&  \qquad  + \sum_{e \in \cE_h^N} \Big( \gamma_{N} \int_e h_e^{-1} ( \bu_h \cdot \nn) (\bv_h \cdot \nn) \ds  - \int_e (\nn^t(\nu \beps( {\Pi^{\beps,k}_{K_e}} \bu_h) \nn)) (\bv_h \cdot \nn) \ds  - \int_e (\nn^t(\nu \beps( {\Pi^{\beps,k}_{K_e}} \bv_h) \nn)) (\bu_h \cdot \nn) \ds \Big) \\
	& \qquad + \sum_{K \in \cT_h} b^K(\bv_h, p_h) + \sum_{e \in \cE_h^D}\int_e (p_h \nn) \cdot \bv_h \ds + \sum_{e \in \cE_h^N}\int_e p_h (\bv_h \cdot \nn) \ds \\
	& =  \sum_{K \in \cT_h} (\bPi^{\cero,k-1}_K \ff, \bv_h)_{0,K}, \\	
	& \sum_{K \in \cT_h} b^K(\bu_h, q_h) + \sum_{e \in \cE_h^D}\int_e (q_h \nn) \cdot \bu_h \ds + \sum_{e \in \cE_h^N}\int_e q_h (\bu_h \cdot \nn) \ds  =  \sum_{e \in \cE_h^D} \int_e  \bg \cdot (q_h \nn)\ds,
\end{align*}
where $\Pi^{\beps,k}_{K_e}$ is the projection onto piecewise polynomial
on element $K_e$ having boundary containing the edge $e$.

Introduce the discrete bilinear forms as, for all $\bu_h, \bv_h \in \bV_h^k,\, q_h \in Q_h^k,$
\begin{align*}
	m_{h}(\bu_h, \bv_h) &:= \sum_{K \in \cT_h} \big( m^K(\bPi^{\cero,k}_K \bu_h, \bPi^{\cero,k}_K \bv_h) + \mathcal{S}^{0}_K((\bI - \bPi^{0, k}_K)\bu_h, (\bI - \bPi^{0, k}_K)\bv_h) \big), \\
	a_{h}(\bu_h, \bv_h) &:= \sum_{K \in \cT_h}a_{h}^K(\bu_h, \bv_h) +  \cN_{h}^{S, D}(\bu_h, \bv_h) + \cN_{h}^{B, D}(\bu_h, \bv_h) + \cN_{h}^{B, D}(\bv_h, \bu_h) \\
	& \qquad  + \cN_{h}^{S, N}(\bu_h, \bv_h) + \cN_{h}^{B, N}(\bu_h, \bv_h) + \cN_{h}^{B, N}(\bv_h, \bu_h), \\
	b_{h}(\bv_h, q_h) &:= \sum_{K \in \cT_h} b^K(\bv_h, q_h) +  \cN_{h}^{b, D}(\bv_h, q_h) + \cN_{h}^{b, N}(\bv_h, q_h), \\
	F_h(\bv_h) &:= \langle \ff_h, \bv_h \rangle + \gamma_{D} \sum_{e \in \cE_h^D} \int_e h_e^{-1}  \bg \cdot \bv_h \ds - \sum_{e \in \cE_h^D} \int_e \bg \cdot (\nu \beps( {\Pi^{\beps,k}_{K_e}} \bv_h) \nn) \ds, \\
	\text{and } \quad G(q_h) &:=  \sum_{e \in \cE_h^D} \int_e  \bg \cdot (q_h \nn)\ds,
\end{align*}
$\text{where } \, a_{h}^K(\bu_h, \bv_h):= a^K(\bPi^{\beps, k}_K  \bu_h, \bPi^{\beps, k}_K \bv_h) + \mathcal{S}^{\beps}_K((\bI - \bPi^{\beps, k}_K)\bu_h, (\bI - \bPi^{\beps, k}_K)\bv_h),$ and  $\ff_h|_K:= \bPi^{\cero,k-1}_K \ff.$

On the one hand, we note that the stabilization terms $\mathcal{S}^{\beps}_K(\cdot,\cdot), \mathcal{S}^{0}_K(\cdot,\cdot)$ are defined, for all $\bv_h \in \bV_h^k(K),$ so that we have the stability with respective bilinear forms, for constants $\zeta_1, \zeta_2, \xi_1, \xi_2 >0$, independent of $h$ and any physical parameters,
\begin{equation}\label{ghjtnb}
	\zeta_1 \, a^K(\bv_h,\bv_h) \le \mathcal{S}^{\beps}_K(\bv_h, \bv_h) \le \zeta_2 \, a^K(\bv_h,\bv_h), \,\, \text{and} \,\,
	\xi_1 m^K(\bv_h,\bv_h) \le \mathcal{S}^0_K(\bv_h, \bv_h) \le \xi_2 \,m^K(\bv_h,\bv_h).
\end{equation}

\noindent On the other hand, the Nitsche's stabilization terms with parameter $\gamma_{D}, \gamma_{N}>0$, $\forall \bv_h \in \bV_h^k$, $q_h \in Q_h^k$ as
\begin{gather*}
	\cN_{h}^{S, D}(\bu_h, \bv_h) := \gamma_{D} \sum_{e \in \cE_h^D} \int_e h_e^{-1}  \bu_h \cdot \bv_h \ds, \qquad 
	\cN_{h}^{B, D}(\bu_h, \bv_h) := - \sum_{e \in \cE_h^D} \int_e (\nu \beps( {\Pi^{\beps,k}_{K_e}} \bu_h) \nn) \cdot  \bv_h \ds, \\
	\cN_{h}^{b, D}(\bv_h, q_h) := \sum_{e \in \cE_h^D}\int_e (q_h \nn) \cdot \bv_h \ds, \qquad
	\cN_{h}^{S, N}(\bu_h, \bv_h) := \gamma_{N} \sum_{e \in \cE_h^N}\int_e h_e^{-1} ( \bu_h \cdot \nn) (\bv_h \cdot \nn) \ds, \\
	\cN_{h}^{B, N}(\bu_h, \bv_h) := - \sum_{e \in \cE_h^N}\int_e (\nn^t(\nu \beps( {\Pi^{\beps,k}_{K_e}} \bu_h) \nn)) (\bv_h \cdot \nn) \ds,  \quad
	\text{and} \quad \cN_{h}^{b, N}(\bv_h,q_h) :=  \sum_{e \in \cE_h^N}\int_e q_h (\bv_h \cdot \nn) \ds.
\end{gather*}

Now, we are in a position to introduce the virtual element formulation
with Nitsche stabilization for the Brinkman problem with mixed boundary conditions,
as follows: seek $\bu_h \in \bV_h^k,$ and $p_h \in Q_h^k$ such that
\begin{subequations} \label{eq:weak-h}
	\begin{align}
	m_h (\bu_h, \bv_h) + a_{h}(\bu_h, \bv_h) + b_h(\bv_h, p_h)  & =  F_h(\bv_h) & \quad \forall \bv_h \in \bV_h^k, \label{eq:weak-h-u}	\\
	b_h(\bu_h, q_h)   & = G(q_h) & \quad \forall  q_h \in Q_h^k.	\label{eq:weak-h-p}
	\end{align}
\end{subequations}
Denoting 
\begin{align*}
	\cA_h ((\bu_h, p_h), (\bv_h, q_h)) &:= m_h (\bu_h, \bv_h) + a_{h}(\bu_h, \bv_h) + b_h(\bv_h, p_h) - b_h(\bu_h, q_h), \\
	\cL_h((\bv_h, q_h)) &¨:= F_h(\bv_h) - G(q_h),
\end{align*}
then we can rewrite the discrete formulation \eqref{eq:weak-h-u}-\eqref{eq:weak-h-p}, in vector form as:
find $(\bu_h, p_h) \in \bbV_h:= \bV_h^k \times Q_h^k$ such that
\begin{align}  \label{eq:weak-h-vec}
	\cA_h ((\bu_h, p_h), (\bv_h, q_h)) =  \cL_h((\bv_h, q_h))  \qquad \forall (\bv_h, q_h) \in \bbV_h.
\end{align}

\begin{remark}
	The Nitsche stabilization terms in our formulation~\eqref{eq:weak-h},
	such as, $\cN_{h}^{S, D}(\cdot, \cdot)$ and $\cN_{h}^{S, N}(\cdot, \cdot)$
	are used to impose the Dirichlet and slip boundary conditions, respectively,
	for the discrete solution. The other terms in the formulation, such as,
	$\cN_{h}^{B, D}(\bu_h, \bv_h), \, \cN_{h}^{B, N}(\bu_h, \bv_h), \, \cN_{h}^{b, D}(\bv_h, p_h), \cN_{h}^{b, N}(\bv_h, p_h)$ appears naturally while the terms $\cN_{h}^{B, D}(\bv_h, \bu_h),$ $ \cN_{h}^{B, N}(\bv_h, \bu_h),$ $\cN_{h}^{b, D}(\bu_h, q_h), \cN_{h}^{b, N}(\bu_h, q_h)$ are added to retain the symmetry in the discrete formulation.
\end{remark}

\section{Solvability of the VEM with Nitsche stabilized terms}\label{solvab}

In this section, we are going to prove the well-posedness of our discrete formulation.
First, we state below the preliminary results used further in analysis.
{
We recall the classical trace and inverse inequalities for piecewise polynomials.
\begin{lemma}[Discrete trace inequality] \label{lem:dis-tr}
	There exists a constant $C_t$ such that for all $w_h \in \bbP_k(K)$ and $K \in 
	\cT_h$ such that
	\begin{align*}
		\| w_h \|_{0, \partial K} \le C_t  (h_K^{-1/2}\| w_h \|_{0, K} + h_K^{1/2}\| \nabla w_h \|_{0, K}).
	\end{align*}
\end{lemma}
\begin{lemma}[Discrete inverse inequality] \label{lem:dis-inv}
	There exists a constant $C_I$ such that for all $w_h \in \bbP_k(K)$ and $K \in 
	\cT_h$ such that
	\begin{align*}
		\| \nabla w_h \|_{0,K} \le C_I  h_K^{-1}\| w_h \|_{0, K}.
	\end{align*}
\end{lemma}
}

\begin{lemma} \label{lem:discrete_trace}
	The following inequalities hold, with $C_{\text{tr}}$ and $C_T$ independent of $h$,
	\begin{align*}
		\| q_h \|_{h,D} := \Big( \sum_{e \in \cE_h^{D}} h_e \| q_h \|_{0, e}^2 \Big)^{1/2} & \le C_{\text{tr}}  \| q_h \|_{0, \Omega} \qquad \forall q_h \in Q_h^k, \\
		\Big( \sum_{e \in \cE_h^{D}} h_e\| \beps({\Pi^{\beps,k}_{K_e}} \bv_h) \nn \|_{0,e}^2 \Big)^{1/2} & \le C_T  \| \beps( \bv_h) \|_{0, \Omega} \qquad \forall \bv_h \in \bV_h^k.
	\end{align*}
and same holds on the boundary $\g$.
\end{lemma}
{\begin{proof}
	Use of Lemma \ref{lem:dis-tr} and Lemma \ref{lem:dis-inv} and for $C_{\text{tr}}:=C_t C_I$, there holds
	\begin{align*}
		\| q_h \|_{0, \partial K} \le C_{\text{tr}} h_K^{-1/2} \| q_h\|_{0,K} \quad \forall q_h \in Q_h^k.
	\end{align*}
	Using the above inequality, we achieve
	\begin{align*}
		\| q_h\|_{h,D} & = \left( \sum_{e \in \cE_h^D} h_e \| q_h \|_{0,e}^2 \right)^{1/2 }  \le \left( \sum_{K \in \cT_h} h_K \| q_h \|_{0,\partial K}^2 \right)^{1/2}  \le C_{\text{tr}} \left( \sum_{K \in \cT_h} \| q_h \|_{0,K}^2 \right)^{1/2} = C_{\text{tr}} \| q_h \|_{0,\Omega}.
	\end{align*}
	Application of the discrete trace and inverse inequalities \cite[Section $5.2$]{massing_cmame18} for the piecewise polynomials in vector form leads to
	\begin{align*}
		\Big( \sum_{e \in \cE_h^{D}} h_e\| \beps({\Pi^{\beps,k}_{K_e}} \bv_h) \nn \|_{0,e}^2 \Big)^{1/2} & \le \left( \sum_{K \in \cT_h} h_K  \| \beps({\Pi^{\beps,k}_{K}} \bv_h) \nn \|_{0,\partial K}^2  \right)^{1/2} 
		 \le \tilde{C}_{\text{tr}} \left( \sum_{K \in \cT_h}  \| \beps (\Pi^{\beps,k}_{K} \bv_h)\|_{0,K}^2 \right)^{1/2}.
	\end{align*}
	 The continuity of $\Pi^{\beps,k}_{K}$ projection by definition \eqref{Def:Pi_eps}, that is $\| \beps (\Pi^{\beps,k}_{K}  \bv_h)\|_{0,K}  \le C_{\Pi} \| \beps (\bv_h)\|_{0,K}$ for all $K \in \cT_h$, and thus, we conclude the inequality with $C_T:=\tilde{C}_{\text{tr}} C_{\Pi}$.
\end{proof}}

\noindent Next, we introduce the following discrete norms on boundary parts.
\begin{align*}
	\| \bv_h \|_{1/2,h,D}^2 := \sum_{e \in \cE_h^{D}} h_e^{-1} \| \bv_h \|_{0,e}^2, \qquad \text{ and } \qquad \| \bv_h \|_{1/2,h,N}^2 := \sum_{e \in \cE_h^{N}} h_e^{-1} \| \bv_h \cdot \nn \|_{0,e}^2.
\end{align*}

\noindent Now, we have that the discrete forms satisfy the following properties.
\begin{lemma} \label{cor:boundary}
	The boundary terms are bounded, for all $\bu_h,\bv_h \in \bV_h^k$, $q_h \in Q_h^k$, as follows:
	\begin{align*}
		\cN_{h}^{S,D}(\bu_h, \bv_h) & \le \gamma_{D} \| \bu_h \|_{1/2,h,D} \| \bv_h \|_{1/2,h,D}, \\
		\cN_{h}^{B, D}(\bu_h, \bv_h) &\le C_T \nu \| \beps(\bu_h) \|_{0, \Omega} \| \bv_h \|_{1/2,h,D}, \\
		\cN_{h}^{b, D}(\bv_h, q_h) & \le C_{\text{tr}} \| q_h \|_{0, \Omega} \| \bv_h \|_{1/2,h,D}, \\
		\cN_{h}^{S, N}(\bu_h, \bv_h) & \le \gamma_{N} \| \bu_h \|_{1/2,h,N} \| \bv_h \|_{1/2,h,N}, \\
		\cN_{h}^{B, N}(\bu_h, \bv_h) & \le  C_T \,\nu \| \beps(\bu_h)\|_{0,\Omega} \|\bv_h \|_{1/2,h, N}, \\
         \cN_{h}^{b, N}(\bv_h,q_h) & \le C_{\text{tr}} \| q_h \|_{0, \Omega} \| \bv_h \|_{1/2,h,N}.
	\end{align*}
\end{lemma}
\begin{proof}
	We start with the first bound. The Cauchy-Schwarz inequality yields
	\begin{align*}
		\cN_{h}^{S,D}(\bu_h, \bv_h) \le \gamma_{D}  \sum_{e \in \cE_h^D} \Big( \int_e h_e^{-1}\bu_h^2 \Big)^{1/2} \Big( \int_e h_e^{-1} \bv_h^2 \Big)^{1/2} \le  \gamma_{D} \| \bu_h \|_{1/2,h,D} \| \bv_h \|_{1/2,h,D}.
	\end{align*}
	Next, by using the Cauchy-Schwarz inequality and the use of Lemma~\ref{lem:discrete_trace} leads to
	\begin{align*}
		\cN_{h}^{B, D}(\bu_h, \bv_h) & \le \sum_{e \in \cE_h^B}\nu h_e^{1/2} \| \beps({\Pi^{\beps,k}_{K_e}} \bu_h) \nn \|_{0,e} h_e^{-1/2} \| \bv_h \|_{0,e} \\
		& \le  C_T \nu\| \beps(\bu_h) \|_{0, \Omega} \big( \sum_{e \in \cE_h^B} h_e^{-1} \| \bv_h \|_{0,e}^2 \big)^{1/2}.
	\end{align*}
Proceeding, in similar manner and using that $\|\nn \|_{\infty,e} \le 1$, we get the bound
\begin{align*}
	\cN_{h}^{b, D}(\bv_h, q_h) & \le \sum_{e \in \cE_h^D} h_e^{1/2} \| q_h \|_{0, e} \|\nn \|_{\infty,e} h_e^{-1/2}\|\bv_h\|_{0, e} \le C_{\text{tr}} \| q_h \|_{0, \Omega} \big( \sum_{e \in \cE_h^D} h_e^{-1} \| \bv_h \|_{0,e}^2 \big)^{1/2}.
\end{align*}
On boundary $\g_N$, we have 
\begin{align*}
	\cN_{h}^{S, N}(\bu_h, \bv_h) \le \gamma_{N}  \sum_{e \in \cE_h^N} \Big( \int_e h_e^{-1} ( \bu_h \cdot \nn)^2 \Big)^{1/2} \Big( \int_e h_e^{-1} ( \bv_h \cdot \nn)^2 \Big)^{1/2} \le  \gamma_{N} \| \bu_h \|_{1/2,h,N} \| \bv_h \|_{1/2,h,N}.
\end{align*}
Next, as before, by using the bound of unit normal $\|\nn \|_{\infty,e} \le 1$
and use of Lemma~\ref{lem:discrete_trace}, we have
\begin{align*}
	\cN_{h}^{B, N}(\bu_h, \bv_h) \le \sum_{e \in \cE_h^N} \nu \|\nn \|_{\infty,e}  h_e^{1/2}\| \beps( {\Pi^{\beps,k}_{K_e}} \bu_h) \nn\|_{0,e} h_e^{-1/2} \|\bv_h \cdot \nn\|_{0,e} \le \nu C_T \| \beps(\bu_h)\|_{0,\Omega} \|\bv_h \|_{1/2,h, N}.
\end{align*}
Using again the previous arguments gives
\begin{align*}
	\cN_{h}^{b, N}(\bv_h,q_h) \le \sum_{e \in \cE_h^N} h_e^{1/2} \| q_h \|_{0, e} h_e^{-1/2}\|\bv_h \cdot \nn\|_{0, e} \le C_{\text{tr}} \| q_h \|_{0, \Omega} \big( \sum_{e \in \cE_h^N} h_e^{-1} \| \bv_h \cdot \nn\|_{0,e}^2 \big)^{1/2}.
\end{align*}
Thus, the proof is complete.
\end{proof}

\noindent Define the mesh-dependent norm for all $\bv_h \in \bV_h^k$ as,
$$\vertiii{\bv_h}_h^2:= \vertiii{\bv_h}_{\bV}^2  + \| \bv_h \|_{1/2,h,D}^2 + \| \bv_h \|_{1/2,h,N}^2,$$
and the parameter-dependent discrete norm on $\bbV_h^k$ as
\begin{align*}
	\vertiii{(\bv_h,q_h)}_h^2 & := \vertiii{\bv_h}_h^2+\| q_h \|_{0, \Omega}^2.
\end{align*}

\subsection{Continuity of the forms $\cA_h$ and $\cL_h$}

In this subsection, we show the boundedness for the presented
discrete forms in the discrete formulation.
We start by bounding every term in the definitions of $\cA_h$ and $\cL_h$.

\begin{itemize}
	\item $m_{h}(\cdot, \cdot)$ is continuous. In fact, we have that
	\begin{align*}
	m_{h}(\bu_h, \bv_h) & \le \sum_{K \in \cT_h} \big( m^K(\bPi^{\cero,k}_K \bu_h, \bPi^{\cero,k}_K \bu_h)^{1/2} m^K(\bPi^{\cero,k}_K \bv_h, \bPi^{\cero,k}_K \bv_h)^{1/2} \\
	& \qquad \quad  + \xi_2 \, m^K((\bI - \bPi^{0, k}_K)\bu_h, (\bI - \bPi^{0, k}_K)\bu_h)^{1/2} m^K((\bI - \bPi^{0, k}_K)\bv_h, (\bI - \bPi^{0, k}_K)\bv_h)^{1/2} \big) \\
	& \le \max \{ 1, \xi_2 \} \,\sum_{K \in \cT_h}  m^K(\bu_h, \bu_h) \, m^K(\bv_h, \bv_h) = \xi^* \| \bbK^{-1/2} \bu_h\|_{0, \Omega} \| \bbK^{-1/2} \bv_h\|_{0, \Omega},
\end{align*}
where we have used \eqref{ghjtnb} and with $\xi^* := \max \{ 1, \xi_2 \}.$
\item $a_{h}(\cdot, \cdot)$ is continuous:
\begin{align*}
	{a}_{h}^K(\bu_h, \bv_h) & \le \big( a^K(\bPi^{\beps,k}_K \bu_h, \bPi^{\beps,k}_K \bu_h)^{1/2} a^K(\bPi^{\beps, k}_K \bv_h, \bPi^{\cero,k}_K \bv_h)^{1/2} \\
	& \qquad \quad  + \zeta_2 \, a^K((\bI - \bPi^{\beps, k}_K)\bu_h, (\bI - \bPi^{\beps, k}_K)\bu_h)^{1/2} \, a^K((\bI - \bPi^{\beps, k}_K)\bv_h, (\bI - \bPi^{\beps, k}_K)\bv_h)^{1/2} \big) \\
	& \le \max \{ 1, \zeta_2 \} \,  a^K(\bu_h, \bu_h) \, a^K(\bv_h, \bv_h) = \max \{ 1, \zeta_2 \} \, \nu \| \beps (\bu_h)\|_{0, K} \| \beps (\bv_h) \|_{0, K},
\end{align*}
where we have used \eqref{ghjtnb}.
Thus, we obtain
\begin{align*}
	a_{h}(\bu_h, \bv_h) &:= \sum_{K \in \cT_h}{a}_{h}(\bu_h, \bv_h) +  \cN_{h}^{S, D}(\bu_h, \bv_h) + \cN_{h}^{B, D}(\bu_h, \bv_h) + \cN_{h}^{B, D}(\bv_h, \bu_h) \\
	& \qquad  + \cN_{h}^{S, N}(\bu_h, \bv_h) + \cN_{h}^{B, N}(\bu_h, \bv_h) + \cN_{h}^{B, N}(\bv_h, \bu_h) \\
	& \le \max \{ 1, \zeta_2 \} \, \nu \| \beps (\bu_h)\|_{0, \Omega} \| \beps (\bv_h) \|_{0, \Omega}+ \gamma_{D} \| \bu_h \|_{1/2,h,D} \| \bv_h \|_{1/2,h,D} + C_T \nu \| \beps(\bu_h) \|_{0, \Omega} \| \bv_h \|_{1/2,h,D}  \\
	& \qquad + C_T \nu \| \beps(\bv_h) \|_{0, \Omega} \| \bu_h \|_{1/2,h,D} + \gamma_{N} \| \bu_h \|_{1/2,h,N} \| \bv_h \|_{1/2,h,N} + C_T \,\nu \| \beps(\bu_h)\|_{0,\Omega} \|\bv_h \|_{1/2,h, N}\\
	& \qquad + C_T \,\nu \| \beps(\bv_h)\|_{0,\Omega} \|\bu_h \|_{1/2,h, N}\\
	& \le \zeta^*  \big( \nu \| \beps (\bu_h)\|_{0, \Omega}^2 + \| \bu_h \|_{1/2, h,D}^2 + \| \bu_h \|_{1/2, h,N}^2 \big)^{1/2}  \big( \nu \| \beps (\bv_h)\|_{0, \Omega}^2 + \| \bv_h \|_{1/2, h,D}^2 + \| \bv_h \|_{1/2, h,N}^2  \big)^{1/2},
\end{align*}
where $\zeta^* := (2 C_T \nu^{1/2} + \gamma_{D}+ \gamma_{N}+ \max \{ 1, \zeta_2 \} ).$
\item $b_{h}(\cdot, \cdot)$ is continuous:
\begin{align*}
	b_{h}(\bv_h, q_h) & \le   \| \vdiv \bv_h\|_{0, \Omega} \| q_h\|_{0, \Omega}  +  C_{\text{tr}} \| q_h \|_{0, \Omega} ( \| \bv_h \|_{1/2, h,D} + \| \bv_h \|_{1/2, h,N}) \\
	& \le (1+C_{\text{tr}}) \Big( \| \vdiv \bv_h\|_{0, \Omega}^2  +   \| \bv_h \|_{1/2, h,D}^2 + \| \bv_h \|_{1/2, h,N}^2 \Big)^{1/2} \| q_h\|_{0, \Omega} .
\end{align*}
\item $F_{h}(\cdot)$ is continuous: 
\begin{align*}
	F_h(\bv_h) & \le 
	 \| \ff_h \|_{0, \Omega}  \|\bv_h \|_{0, \Omega}  + \gamma_{D} \sum_{e \in \cE_h^D} \|h_e^{-1/2}  \bg\|_{0,e} \|h_e^{-1/2} \bv_h \|_{0,e} + \sum_{e \in \cE_h^D} \|h_e^{-1/2}  \bg\|_{0,e}  \nu \|h_e^{1/2} \beps( {\Pi^{\beps,k}_{K_e}} \bv_h) \nn\|_{0,e} \\
	& \le 
	\| \bbK ^{1/2} \|_{\infty, \Omega} \| \ff \|_{0, \Omega}  \|\bbK^{-1/2}\bv_h \|_{0, \Omega} + \gamma_{D} \sum_{e \in \cE_h^D} \|h_e^{-1/2}  \bg\|_{0,e} \|h_e^{-1/2} \bv_h \|_{0,e} \\
	& \qquad \quad + ( \sum_{e \in \cE_h^D} \|h_e^{-1/2}  \bg\|_{0,e}^2)^{1/2}  \nu C_T \|\beps( \bv_h) \|_{0,\Omega} \\
	& \le C \Big( \| \ff \|_{0, \Omega}^2   + (\gamma_{D} +  \nu \, C_T )  \|  \bg\|_{1/2,h, D} ^2 \Big)^{1/2}  \vertiii{\bv_h }_{h},
\end{align*}
where the discrete norm $\|  \bg\|_{1/2,h, D}:= \Big( \sum_{e \in \cE_h^{D}} h_e^{-1} \| \bg \|_{0, e}^2 \Big)^{1/2}$ (see \cite{urquiza2014weak}).
\item $G(\cdot)$ is continuous:
\begin{align*}
G(q_h) &\le  \sum_{e \in \cE_h^D} h_e^{-1/2} \| \bg\|_{0,e} h_e^{1/2} \| q_h \|_{0,e} \le C_{tr} \|  \bg\|_{1/2,h, D}\| q_h \|_{0,\Omega} .
\end{align*}

\end{itemize}

Thus, as a consequence of the above bounds,
we have that the discrete bilinear form $\cA_h$ and the linear
form $\cL_h$ are bounded with constants independent
of the mesh-size and the physical parameters.

\subsection{Global inf-sup condition}


First, in the following lemma, we show that the bilinear form
$b_h(\cdot, \cdot)$ satisfies the inf-sup condition on $\bV_h^k \times Q_h^k$.
\begin{lemma} \label{lem:discrete_inf}
	The bilinear form $b_{h}(\cdot, \cdot)$ satisfies inf-sup condition on ${\bV}_h^k \times Q_h^k$, that is, there exist constants $\beta >0$ such that
	\begin{align*}
		\sup_{(0 \neq)\bv_h \in \bV_h^k} \frac{b_{h}(\bv_h, q_h)}{\vertiii{ \bv_h}_{h}} \ge \beta \| q_h \|_{0, \Omega} \quad \forall q_h \in Q_h^k.
	\end{align*}
\end{lemma}
\begin{proof}
	From \cite{vacca17}, we have that
	\begin{align*}
		\sup_{(0 \neq){\bv}_h \in \bV_h^k \cap \bV} \frac{b({\bv}_h, q_h)}{\vertiii{\bv_h}_{\bV}} \ge \beta \| q_h \|_{0, \Omega} \qquad \forall q_h \in Q_h^k.
	\end{align*}
	For ${\bv}_h \in \bV_h^k  \cap \bV$, $\vertiii{\bv_h}_h= \vertiii{\bv_h}_{\bV}$ and $b_{h}({\bv}_h, q_h)= b({\bv}_h, q_h) \,\, \forall q_h \in Q_h^k$ and thus yield
	\begin{align*}
		\sup_{(0 \neq){\bv}_h \in \bV_h^k\cap \bV} \frac{b_{h}({\bv}_h, q_h)}{\vertiii{ {\bv}_h}_{h}} = \sup_{(0 \neq){\bv}_h \in \bV_h^k \cap \bV} \frac{b({\bv}_h, q_h)}{\vertiii{{\bv}_h}_{\bV}}.
	\end{align*}
	Trivially, we have $\bV_h^k \cap \bV \subseteq \bV_h^k$ and achieve
	\begin{align*}
		\sup_{(0 \neq)\bv_h \in \bV_h^k} \frac{b_{h}(\bv_h, q_h)}{\vertiii{ \bv_h}_{h}} \ge \sup_{(0 \neq){\bv}_h \in \bV_h^k\cap \bV} \frac{b_{h}({\bv}_h, q_h)}{\vertiii{ \bv_h}_{h}}
		\ge \beta \| q_h \|_{0, \Omega} \qquad \forall q_h \in Q_h^k.
	\end{align*}
\end{proof}

{\begin{lemma} \label{lem:coerc}
		For $\gamma_{D},\gamma_{N}> \frac{2 C_T^2 }{\min \{1, \zeta_1\}}$ and for all $\bv_h \in \bV_h^k$, there holds,
		\begin{align*}
			m_{h}(\bv_h, \bv_h)+ a_{h}(\bv_h, \bv_h) \ge (\xi_* + \zeta_*) \big(\|\mathbb{K}^{-1/2} \bv_h\|_{0,\Omega}^2 +  \nu \| \beps(\bv_h) \|_{0,\Omega}^2 + \| \bv_h \|_{1/2,h,D}^2 + \| \bv_h \|_{1/2,h,N}^2 \big),
		\end{align*}
	for $\xi_*, \zeta_*$ independent of $h$ and all the physical parameters.
	\end{lemma}
	\begin{proof}
		Use of the lower bound of $\mathcal{S}^0_K(\cdot,\cdot)$ from \eqref{ghjtnb}, and $\xi_*:=\min \{1, \xi_1\}$ yield 
		\begin{align*}
			m_{h}(\bv_h, \bv_h) &:= \sum_{K \in \cT_h}m_{h}^K(\bv_h, \bv_h) \ge \xi_*  \| \bbK^{-1} \bv_h \|_{0, \Omega}^2.
		\end{align*}
		From definition of the bilinear forms, lower bound of $\mathcal{S}^{\beps}_K(\cdot,\cdot)$ in \eqref{ghjtnb} and use of Lemma \ref{cor:boundary} gives
		\begin{align*}
			a_{h}(\bv_h, \bv_h) &:= \sum_{K \in \cT_h}a_{h}^K(\bv_h, \bv_h) +  \cN_{h}^{S, D}(\bv_h, \bv_h) + 2\,\cN_{h}^{B, D}(\bv_h, \bv_h) + \cN_{h}^{S, N}(\bv_h, \bv_h) + 2 \, \cN_{h}^{B, N}(\bv_h, \bv_h) \\
			& \ge \min \{1, \zeta_1\} \nu \| \beps(\bv_h) \|_{0,\Omega}^2 + \gamma_{D}  \| \bv_h \|_{h, D}^2  + \gamma_{N}  \| \bv_h \|_{1/2,h,N}^2 - 2   C_T \nu \| \beps(\bv_h) \|_{0, \Omega} \| \bv_h \|_{h,D} \\
			& \qquad \qquad -2 C_T \nu \| \beps(\bv_h)\|_{0,\Omega}\|\bv_h\|_{1/2,h,N}.
		\end{align*}
	Taking $C_\zeta:= \min \{1, \zeta_1\}$ and use of Young's inequality with $\epsilon=C_\zeta/2>0$, there holds 
	$$2 C_T \nu \| \beps(\bv_h) \|_{0, \Omega} \| \bv_h \|_{1/2,h,D} \le \frac{C_\zeta}{2} \nu^2 \| \beps(\bv_h) \|_{0,\Omega}^2 +  \frac{2 C_T^2 }{C_\zeta} \| \bv_h \|_{1/2,h, D}^2, $$
	 and using $\nu \le 1$, we achieve
		\begin{align*}
			a_{h}(\bv_h, \bv_h) & \ge  \frac{C_\zeta}{2}  \nu \| \beps(\bv_h) \|_{0,\Omega}^2 + \big(\gamma_{D} - \frac{2 C_T^2 }{C_\zeta} \big)  \| \bv_h \|_{1/2,h, D}^2  + \big(\gamma_{N} - \frac{2 C_T^2 }{C_\zeta} \big)  \| \bv_h \|_{1/2,h, N}^2.
		\end{align*}
		Choosing $\gamma_{D},\gamma_{N}> \frac{2 C_T^2 }{C_\zeta}$ and choosing $\zeta_*:= \min \{ \frac{C_\zeta}{2}, \gamma_{D} - \frac{2 C_T^2 }{C_\zeta}, \gamma_{N} - \frac{2 C_T^2 }{C_\zeta} \}$, we get
		\begin{align*}
			a_{h}(\bv_h, \bv_h) & \ge \zeta_* ( \nu \| \beps(\bv_h) \|_{0,\Omega}^2 +  \| \bv_h \|_{1/2,h, D}^2  + \| \bv_h \|_{1/2,h, N}^2).
		\end{align*}
	Thus, we can conclude this lemma.
	\end{proof}
}

In the next theorem, we prove the global inf-sup condition
for the discrete bilinear form $\cA_h(\cdot, \cdot)$.

\begin{theorem} \label{thm:discrete-inf-sup}
	For $(\bu_h, p_h)  \in \bbV_h^k$, there exists $(\bv_h, q_h) \in \bbV_h^k$ with $\vertiii{(\bv_h, q_h) }_{h} \le C \vertiii{(\bu_h, p_h)}_h$ such that
	\begin{align*}
		\cA_h ((\bu_h, p_h), (\bv_h, q_h))  \ge C_* \vertiii{(\bu_h, p_h)}_h^2.
	\end{align*}
	where $C_*$ is a constant independent of the mesh-size and the physical parameters.
\end{theorem}
\begin{proof}
	By definition of $\cA_h(\cdot,\cdot)$ and Lemma \ref{lem:coerc}, we have 
	\begin{align} \label{est:GI-1}
		& \cA_h ((\bu_h, p_h), (\bu_h, p_h))  = m_{h}(\bu_h, \bu_h)+ a_{h}(\bu_h, \bu_h) \nonumber \\
		& \qquad  \ge (\xi_* + \zeta_*) \big(\|\mathbb{K}^{-1/2} \bu_h\|_{0,\Omega}^2 +  \nu \| \beps(\bu_h) \|_{0,\Omega}^2 + \| \bu_h \|_{1/2,h,D}^2 + \| \bu_h \|_{1/2,h,N}^2 \big).
	\end{align}
The discrete inf-sup condition in Lemma \ref{lem:discrete_inf} gives the existence of $\bw_h^* \in \bV_h^k$ such that $$ -b_h(\bw_h^*, p_h) \ge \beta_* \| p_h \|_{0, \Omega}^2 \quad \text{and} \quad \vertiii{ \bw_h^* }_h \le \| p_h \|_{0, \Omega}.$$
Thus, we get
\begin{align*}
	\cA_h ((\bu_h, p_h), (\bw_h^*, 0)) & = m_{h}(\bu_h, \bw_h^*)+ a_{h}(\bu_h, \bw_h^*) + b_h(\bw_h^*, p_h) \\
	& \le  (\xi^*+ \zeta^*) \big( \| \bbK^{-1} \bu_h \|_{0, \Omega}^2 +  \nu \| \beps (\bu_h)\|_{0, \Omega}^2 + \| \bu_h \|_{1/2, h,D}^2 + \| \bu_h \|_{1/2, h,N}^2 \big)^{1/2} \vertiii{ \bw_h^* }_{h}\\
	& \quad - \beta_* \| p_h \|_{0, \Omega}^2 \\
	& \le  (\xi^*+ \zeta^*) \big( \| \bbK^{-1} \bu_h \|_{0, \Omega}^2 +  \nu \| \beps (\bu_h)\|_{0, \Omega}^2 + \| \bu_h \|_{1/2, h,D}^2 + \| \bu_h \|_{1/2, h,N}^2 \big)^{1/2} \| p_h \|_{0, \Omega} \\
	& \quad - \beta_* \| p_h \|_{0, \Omega}^2.
\end{align*}
By using Young's inequality with $\epsilon_1 >0$ gives
\begin{align} \label{est:GI-2}
	\cA_h ((\bu_h, p_h), (\bw_h^*, 0)) & \le  \frac{\epsilon_1}{2}(\xi^*+ \zeta^*)  \big( \| \bbK^{-1} \bu_h \|_{0, \Omega}^2 +  \nu \| \beps (\bu_h)\|_{0, \Omega}^2 + \| \bu_h \|_{1/2, h,D}^2 + \| \bu_h \|_{1/2, h,N}^2 \big)  \nonumber \\
	& \qquad - \big(\beta_* -  \frac{1}{2\epsilon_1}(\xi^*+ \zeta^*) \big) \| p_h \|_{0, \Omega}^2.
\end{align}
Moreover, we know that for $c:=\frac{1}{\vert\Omega\vert}\int_{\Omega}\vdiv \bu_h$, we obtain
\begin{align*}
b_h(\bu_h, c) & = - \sum_{K \in \cT_h} (\vdiv \bu_h, c)_{0,K} + \sum_{e \in \cE_h^{D}} (\bu_h, c \nn)_{0,e} + \sum_{e \in \cE_h^{N}} (\bu_h \cdot \nn, c)_{0,e} \\
&  = c \Big( - \int_{\Omega} \vdiv \bu_h + \sum_{e \in \cE_h^{D} \cup  \cE_h^{N}} \int_e (\bu_h \cdot \nn) \ds \Big) = 0
\end{align*}
and $(\vdiv \bu_h-c) \in Q_h^k$, then, we  consider
\begin{align*} 
	\cA_h ((\bu_h, p_h), (\cero, \vdiv \bu_h-c)) & = - b_h(\bu_h, \vdiv \bu_h - c) \\
	& = \| \vdiv \bu_h \|_{0,\Omega}^2 - \sum_{e \in \cE_h^{D}} (\vdiv \bu_h, \bu_h)_{0,e} - \sum_{e \in \cE_h^{N}} (\vdiv \bu_h, \bu_h \cdot \nn)_{0,e} \\
	& \ge  \| \vdiv \bu_h \|_{0,\Omega}^2  - \big( \sum_{e \in \cE_h^{D} \cup \cE_h^{N}  } h_e \| \vdiv \bu_h \|_{0,e}^2 \big)^{1/2}  \Big( \| \bu_h\|_{1/2, h, D}^2 + \| \bu_h \cdot \nn \|_{1/2, h, N}^2 \Big)^{1/2}.
\end{align*}
Use of trace inequality in Lemma \ref{lem:discrete_trace} and Young's inequality with $\epsilon_2>0$ gives
\begin{align}\label{est:GI-3}
	\cA_h ((\bu_h, p_h), (\cero, \vdiv \bu_h-c)) & \ge  \| \vdiv \bu_h \|_{0,\Omega}^2 - C_{\text{tr}} \| \vdiv \bu_h \|_{0,\Omega}  \Big( \| \bu_h\|_{1/2, h, D}^2 + \| \bu_h \cdot \nn \|_{1/2, h, N}^2 \Big)^{1/2} \nonumber \\
	& \ge \Big( 1- \frac{\epsilon_2}{2} C_{\text{tr}} \Big) \| \vdiv \bu_h \|_{0,\Omega}^2 - \frac{C_{\text{tr}}}{2\, \epsilon_2} \Big( \| \bu_h\|_{1/2, h, D}^2 + \| \bu_h \cdot \nn \|_{1/2, h, N}^2 \Big).
\end{align}
By considering $(\bv_h, q_h) := \delta (\bu_h, p_h) - (\bw_h^*, 0) + (\cero, \vdiv \bu_h-c)$, for constant $\delta> 0$, and using the bounds \eqref{est:GI-1}-\eqref{est:GI-3}, we have
\begin{align*}
	\cA_h ((\bu_h, p_h), (\bv_h, q_h)) & \ge \delta (\xi_* + \zeta_*) \big(\|\mathbb{K}^{-1/2} \bu_h\|_{0,\Omega}^2 +  \nu \| \beps(\bu_h) \|_{0,\Omega}^2 + \| \bu_h \|_{1/2,h,D}^2 + \| \bu_h \|_{1/2,h,N}^2 \big) \\
	& \qquad -\frac{\epsilon_1}{2}(\xi^*+ \zeta^*)  \big( \| \bbK^{-1} \bu_h \|_{0, \Omega}^2 +  \nu \| \beps (\bu_h)\|_{0, \Omega}^2 + \| \bu_h \|_{1/2, h,D}^2 + \| \bu_h \|_{1/2, h,N}^2 \big)  \\
	& \qquad + \big(\beta_* -  \frac{1}{2\epsilon_1}(\xi^*+ \zeta^*) \big) \| p_h \|_{0, \Omega}^2 + \Big( 1- \frac{\epsilon_2}{2} C_{\text{tr}} \Big) \| \vdiv \bu_h \|_{0,\Omega}^2 \\
	& \qquad - \frac{C_{\text{tr}}}{2\, \epsilon_2} \Big( \| \bu_h\|_{1/2, h, D}^2 + \| \bu_h \cdot \nn \|_{1/2, h, N}^2 \Big) \\
	& \ge \Big( (\delta-\frac{\epsilon_1}{2})(\xi_* + \zeta_*) - \frac{C_{\text{tr}}}{2\, \epsilon_2}\Big) \big(\|\mathbb{K}^{-1/2} \bu_h\|_{0,\Omega}^2 +  \nu \| \beps(\bu_h) \|_{0,\Omega}^2 + \| \bu_h \|_{1/2,h,D}^2 + \| \bu_h \|_{1/2,h,N}^2 \big)\\
	& \qquad + \big(\beta_* -  \frac{1}{2\epsilon_1}(\xi^*+ \zeta^*) \big) \| p_h \|_{0, \Omega}^2 + \Big( 1- \frac{\epsilon_2}{2} C_{\text{tr}} \Big) \| \vdiv \bu_h \|_{0,\Omega}^2.
\end{align*}
 Taking $ \delta:= \frac{C_{\text{tr}}^2 + \beta_*}{\xi_* + \zeta_*}$,
 $ \epsilon_1 := \frac{\xi_* + \zeta_*}{\beta_*}$,
 and $\epsilon_2 :=\frac{1}{C_{\text{tr}}}$, then we obtain
 \begin{align*}
 	\cA_h ((\bu_h, p_h), (\bv_h, q_h)) & \ge \frac{C_{\text{tr}}^2 + \beta_*}{2} \big(\|\mathbb{K}^{-1/2} \bu_h\|_{0,\Omega}^2 +  \nu \| \beps(\bu_h) \|_{0,\Omega}^2 + \| \bu_h \|_{1/2,h,D}^2 + \| \bu_h \|_{1/2,h,N}^2 \big) \\
 	& \qquad + \frac{1}{2}\| \vdiv \bu_h \|_{0,\Omega}^2 + \frac{\beta_*}{2} \| p_h \|_{0, \Omega}^2 \\
 	& \ge C_* \vertiii{(\bu_h, p_h)}_h^2,
 \end{align*}
with $C_*:=\frac{1}{2} \Big( C_{\text{tr}}^2 + 2 \beta_* + 1 \Big)$.

Moreover, we see that $(\bv_h, q_h)= (\delta\bu_h- \bw_h^*, \delta p_h + \vdiv \bu_h-c)$ satisfies
\begin{align*}
	\vertiii{(\bv_h, q_h)}_h^2 & \le  \|\mathbb{K}^{-1/2} (\delta \bu_h- \bw_h^*)\|_{0,\Omega}^2 +  \nu \| \beps(\delta \bu_h- \bw_h^*) \|_{0,\Omega}^2 + \| \vdiv (\delta \bu_h- \bw_h^*) \|_{0,\Omega}^2  \\
	& \qquad + \sum_{e \in \cE_h^{D}} h_e^{-1} \| \delta \bu_h- \bw_h^* \|_{0,e}^2  +  \sum_{e \in \cE_h^{N}} h_e^{-1} \| (\delta \bu_h- \bw_h^*) \cdot \nn_e \|_{0,e}^2 + \|{\delta} p_h + \vdiv \bu_h-c \|_{0, \Omega}^2 \\
	& {\le  \left( \frac{C_{\text{tr}}^2 + \beta_*}{\xi_* + \zeta_*} + 2 \right) ( \vertiii{\bu_h}_h^2 + \| p_h \|_{0, \Omega}^2) }\le C \vertiii{(\bu_h, p_h)}_h^2,
\end{align*}
with $C:=\Big( \frac{C_{\text{tr}}^2 + \beta_*}{\xi_* + \zeta_*} + { 2} \Big).$
\end{proof}

Now, we are now in a position to establish the unique solvability, and the stability properties of
the discrete problem~\eqref{eq:weak-h}. The proof is a direct consequence of the above result.

\begin{theorem}
	{For given $\ff \in [L^2(\Omega)]^2$ and $\bg \in [H^{1/2}(\Gamma_D)]^2$}, the discrete problem~\eqref{eq:weak-h-vec} is well-posed and solution $(\bu_h, p_h) \in \bV_h^k \times Q_h^k$ satisfies
	\begin{align*}
	\vertiii{(\bu_h, p_h)}_h \le C ( \| \ff \|_{0, \Omega}^2   + (\gamma_{D} +  \nu \, C_T + C_{\text{tr}})  \|  \bg\|_{1/2,h, D} ^2)^{1/2}.
	\end{align*}
	where $C$ is a constant independent of viscosity and mesh size $h$.
\end{theorem}

\section{Abstract error analysis}\label{sec:abstract-error-analysis}

{We assume the regularity of the given data $\ff \in [H^{\max \{ 0,r-1\}}(\Omega)]^{2}$ and $\bg \in [H^{r + \frac{1}{2}}(\g_D)]^{2}$ with $r>0$ then we have the regularity estimate for the continuous solutions as
\begin{align} \label{est:reg}
	\| \bu \|_{1+r,\Omega} + \| p \|_{r,\Omega} \le C_{\text{reg}} (\| \ff \|_{\max \{0,r-1 \},\Omega} +\| \bg \|_{r+\frac{1}{2}, \g_D} ).
\end{align}}

\begin{lemma}[Polynomial approximation] \label{lem:poly}
	For $\bu \in [H^{1+r}(\Omega)]^2$ with ${r >0}$, there exists a polynomial approximation
	$\bu_{\pi}|_K \in [\mathbb{P}_k(K)]^{2}$ for each polygon $K$ and satisfies
	\begin{align*}
			\sum_{K \in \cT_h} \Big( \|  \bu -  \bu_{\pi} \|_{0, K} + h |  \bu -  \bu_{\pi}|_{1, K} \Big)  & \lesssim h^{\min\{ k, r\}+1} | \bu|_{1+r, \Omega}.
		\end{align*}
\end{lemma}

\begin{lemma}[Interpolation approximation] \label{lem:interpolant}
	For $\bu \in [H^{1+r}(\Omega)]^2$ with ${r >0}$, there exists
	a polynomial approximation $\bu_I \in\bV_h^k$ and satisfies
	\begin{align*}
	\|  \bu -  \bu_I \|_{0, \Omega} + h |  \bu -  \bu_I|_{1, \Omega} &\lesssim h^{\min\{ k, r\}+1} | \bu|_{1+r, \Omega}.
	\end{align*}
\end{lemma}



The following theorem provides the rate of convergence of our virtual
element scheme with Nitsche stabilization, for the Brinkman problem
with mixed boundary conditions, presented in \eqref{eq:weak-h}.

\begin{theorem}
	Let $(\bu, p) \in \bV \times Q$ and $(\bu_h, p_h) \in \bV_h^k \times Q_h^k$
	be the solutions of the continuous and discrete problems.
	Assuming {$\ff \in [H^{\max \{ r,k\}-1}(\Omega)]^{2}$ and $\bg \in [H^{r + \frac{1}{2}}(\g_D)]^{2}$ for $r>0$, there holds}
	\begin{align*}
	\vertiii{(\bu, p) - (\bu_h, p_h)}_h  \le C h^{\min \{ r,k \}} {(\| \ff \|_{\max \{ r,k\}-1,\Omega} +\| \bg \|_{r+\frac{1}{2}, \g_D} )},
	\end{align*}
	where $C$ is a generic constant independent of $\nu$ and $h$.
\end{theorem}
\begin{proof}
	Let $\boldsymbol{\delta}_h := (\bu_h, p_h) - (\bv_h, q_h) \in \bbV_h^k= \bV_h^k \times Q_h^k$, then,
	as a consequence of the global inf-sup condition in Theorem \ref{thm:discrete-inf-sup}, we have
	\begin{align*}
		C \vertiii{ \boldsymbol{\delta}_h }_{h} & \le \sup_{((\cero, 0) \neq) (\bw_h, r_h) \in \bbV_h^k} \frac{\cA_h(\boldsymbol{\delta}_h \pm (\bu, p), (\bw_h, r_h))}{\vertiii{(\bw_h, r_h)}_h} \\
	& = \sup_{((\cero, 0) \neq) (\bw_h, r_h) \in \bbV_h^k} \left( \frac{\cL_h(\bw_h, r_h) -\cA_h((\bu, p), (\bw_h, r_h))}{\vertiii{(\bw_h, r_h)}_h} + \frac{\cA_h((\bu, p) - (\bv_h, q_h), (\bw_h, r_h))}{\vertiii{(\bw_h, r_h)}_h} \right).
\end{align*}
We start the error analysis with estimating the consistency error for any $(\bw_h, r_h) \in \bV_h^k \times Q_h^k$ 
$$E_c(\bu, p) := \cL_h((\bw_h, r_h)) - \cA_h((\bu, p), (\bw_h, r_h)).$$

\noindent Use of the definition of discrete formulation \eqref{eq:weak-h-vec} leads to
\begin{align*}
E_c(\bu, p) & = F_h(\bw_h)   - G(r_h)  - m_h(\bu, \bw_h) - a_h(\bu, \bw_h) - b_h(\bw_h, p) + b_h(\bu, r_h) \\
& = \langle \ff_h, \bw_h \rangle -\sum_{K \in \cT_h} \big( m_h^K(\bu, \bw_h)+ a_h^K( \bu, \bw_h)  \big) - \sum_{K \in \cT_h} b^K(\bw_h, p)  - \sum_{e \in \cE_h^D}\int_e (p \nn) \cdot \bw_h \ds \\	
& \qquad - \sum_{e \in \cE_h^N}\int_e p (\bw_h \cdot \nn) \ds  + \sum_{e \in \cE_h^D}   \int_e (\nu \beps( {\Pi^{\beps,k}_{K_e}} \bu) \nn) \cdot  \bw_h \ds  + \sum_{e \in \cE_h^N} \int_e (\nn^t(\nu \beps( {\Pi^{\beps,k}_{K_e}} \bu) \nn)) (\bw_h \cdot \nn) \ds .
\end{align*}
The boundary conditions for continuous solution $\bu$ \eqref{bc:Sigma}-\eqref{bc:Gamma} and the continuous formulation \eqref{eq:weak} implies
\begin{align*}
	E_c(\bu, p) & = \Big(\langle \ff_h, \bw_h \rangle - (\ff,\bw_h) \Big) + \Big( m(\bu, \bw_h)+ a( \bu, \bw_h) +b(\bw_h, p) - \sum_{e \in \cE_h^D} \int_e (\nu \beps(\bu) - p \bI)\nn  \cdot \bw_h  \ds \\
	& \qquad - \sum_{e \in \cE_h^N} \int_e (\nn^t(\nu \beps(\bu) - p \mathbb{I})\nn) (\bw_h \cdot \nn)  \ds \Big) -\sum_{K \in \cT_h} \big( m_h^K(\bu, \bw_h)+ a_h^K( \bu, \bw_h)  \big) 
	\\ 
	& \qquad - \sum_{K \in \cT_h} b^K(\bw_h, p)  - \sum_{e \in \cE_h^D}\int_e (p \nn) \cdot \bw_h \ds  - \sum_{e \in \cE_h^N}\int_e p (\bw_h \cdot \nn) \ds  \\	
	& \qquad + \sum_{e \in \cE_h^D}   \int_e (\nu \beps( {\Pi^{\beps,k}_{K_e}} \bu) \nn) \cdot  \bw_h \ds  + \sum_{e \in \cE_h^N} \int_e (\nn^t(\nu \beps( {\Pi^{\beps,k}_{K_e}} \bu) \nn)) (\bw_h \cdot \nn) \ds \\
	& = \Big(\langle \ff_h, \bw_h \rangle - (\ff,\bw_h) \Big) + \sum_{K \in \cT_h} \big( m^K(\bu, \bw_h) - m_h^K(\bu, \bw_h) \big) + \sum_{K \in \cT_h} \big( a^K( \bu, \bw_h)  - a_h^K(\bu, \bw_h) \big)  \\
	& \qquad  - \sum_{e \in \cE_h^D}   \int_e (\nu \beps( (\bI - \Pi^{\beps,k}_{K_e}) \bu) \nn) \cdot  \bw_h \ds  - \sum_{e \in \cE_h^N} \int_e (\nn^t(\nu \beps( (\bI - \Pi^{\beps,k}_{K_e}) \bu) \nn)) (\bw_h \cdot \nn) \ds \\
	& =: \sum_{j=1}^5 T_j.
\end{align*}
The Cauchy-Schwarz inequality with $C_{\bbK}$ is  a constant depending on the bound $\| \bbK\|_{\infty, \Omega}$ gives
\begin{align*}
	T_1 & \le  C_{\bbK} \Big(\sum_{K \in \cT_h} \| (I- \Pi^{0,k-1}_K) \ff\|_{0,K}^2 \Big)^{1/2} \| \bbK^{-1/2} \bw_h\|_{0,\Omega} \\
	& \le C_{\bbK} h^{k} | \ff|_{k-1,\Omega}  \| \bbK^{-1/2} \bw_h\|_{0,\Omega}.
\end{align*}
Referring to \cite{Mora21imajna} for the estimates of the term $T_2$ as
\begin{align*}
T_2 & =\sum_{K \in \cT_h} \big( m^K(\bu, \bw_h) - m_h^K(\bu, \bw_h) \big)  \\
& =\sum_{K \in \cT_h} \big( (\bbK^{-1} \bu, \bw_h)_{0,K} - (\bbK^{-1}  \Pi^{0,k}_K \bu, \Pi^{0,k}_K \bw_h) - S^0_K((\bI - \Pi^{0,k}_K) \bu, (\bI - \Pi^{0,k}_K) \bw_h) \big) \\
& = \sum_{K \in \cT_h} \big( ((\bI - \Pi^{0,k}_K) (\bbK^{-1} \bu), (\bI  - \Pi^{0,k}_K) \bw_h)_{0,K} - (  (\bI - \Pi^{0,k}_K) \bu, (\bI - \Pi^{0,k}_K) (\bbK^{-1} \bw_h))_{0,K} \\
& \qquad \qquad  - (  (\bI - \Pi^{0,k}_K) \bu,  \bbK^{-1}(\bI - \Pi^{0,k}_K)  \bw_h)_{0,K}  - S^0_K((\bI - \Pi^{0,k}_K) \bu, (\bI - \Pi^{0,k}_K) \bw_h) \big) \\
& \le C_{\bbK} \sum_{K \in \cT_h} \Big( \|(\bI - \Pi^{0,k}_K) (\bbK^{-1} \bu) \|_{0,K} \| (\bI  - \Pi^{0,k}_K) \bw_h \|_{0,K} + \|(\bI - \Pi^{0,k}_K) \bu \|_{0,K} \| (\bI  - \Pi^{0,k}_K) (\bbK^{-1} \bw_h) \|_{0,K} \\
& \qquad \qquad \quad + \| \bbK^{-1}\|_{\infty, K} \|(\bI - \Pi^{0,k}_K) \bu \|_{0,K} \| (\bI  - \Pi^{0,k}_K) \bw_h \|_{0,K} \Big) \\
& \le C_{\bbK} h^{\min \{ r,k\}} |\bu|_{1+r, \Omega} \| \bbK^{-1/2} \bw_h \|_{0, \Omega}.
\end{align*}
Use of polynomial approximation Lemma \ref{lem:poly} and polynomial consistency leads to
\begin{align*}
	T_3 &= \sum_{K \in \cT_h} \big( a^K( \bu - \bu_{\pi}, \bw_h)  - a_h^K(\bu - \bu_{\pi}, \bw_h) \big) \\
	& \le 2  \big( \sum_{K \in \cT_h} \nu\| \beps(\bu - \bu_{\pi})\|_{0,K}^2 \big)^{1/2}  \big( \sum_{K \in \cT_h} \nu \|\beps( \bw_h)\|_{0,K}^2 \big)^{1/2} \\
	& \le C h^{\min \{r,k\}} | \bu|_{1+r,\Omega} \,   \nu \|\beps( \bw_h)\|_{0,\Omega}.
\end{align*}
Combination of Cauchy-Schwarz inequality and Lemma \ref{lem:discrete_trace} gives
\begin{align*}
	T_4 & \le \sum_{e \in \cE_h^D}  h_e^{-1/2} \|\nu \beps( (\bI - \Pi^{\beps,k}_{K_e}) \bu) \nn \|_{0,e} h_e^{1/2}\| \bw_h\|_{0,e} \\
	&  \le \nu C_{\text{tr}}  \Big( \sum_{K \in \cT_h} \|\beps( (\bI - \Pi^{\beps,k}_{K}) \bu) \|_{0,K}^2 \Big)^{1/2} \| \bw_h\|_{1/2,h, D} \\
	&  \le \nu C_{\text{tr}}  h^{\min \{r,k\}} | \bu|_{1+r,\Omega} \| \bw_h\|_{1/2,h, D}.
\end{align*}
In similar manner asthebound for term $T_4$, there holds
\begin{align*}
T_5 & \le \sum_{e \in \cE_h^N}  h_e^{-1/2} \|\nu \beps( (\bI - \Pi^{\beps,k}_{K_e}) \bu) \nn \|_{0,e} h_e^{1/2}\| \bw_h \cdot \nn\|_{0,e} \\
& \le \nu C_{\text{tr}}  \Big( \sum_{K \in \cT_h} \|\beps( (\bI - \Pi^{\beps,k}_{K_e}) \bu) \|_{0,K_e}^2 \Big)^{1/2} \| \bw_h\|_{1/2,h, N}\\
&  \le \nu C_{\text{tr}}  h^{\min \{r,k \}} | \bu|_{1+r,\Omega} \| \bw_h\|_{1/2,h, N}.
\end{align*}
The use of the estimates of $T_1$-- $T_5$ and the continuity of discrete form arrive to 
\begin{align*}
C \vertiii{ \boldsymbol{\delta}_h }_{h} & \le   \vertiii{(\bu, p) - (\bv_h, q_h)}_h + C_{\bbK} h^{\min \{ k,r \}} \Big(  {| \ff|_{\max \{ r,k\}-1,\Omega}}+  | \bu|_{1+r,\Omega}^2  \Big)^{1/2} \vertiii{ \bw_h}_h.
\end{align*}
The triangle inequality along with the previous bound for $\vertiii{ \boldsymbol{\delta}_h }_{h}$ yield
\begin{align*}
	\vertiii{(\bu, p) - (\bu_h, p_h)}_h & \le \vertiii{(\bu, p) - (\bv_h, q_h)}_h + \vertiii{(\bv_h, q_h) - (\bu_h, p_h)}_h \\
	& \le 2 \vertiii{(\bu, p) - (\bv_h, q_h)}_h +C \, h^{\min \{ k,r \}} \Big(  {| \ff|^2_{\max \{ r,k\}-1,\Omega}} +  | \bu|_{1+r,\Omega}^2  \Big)^{1/2}  .
\end{align*}

\noindent Taking $\bv_h = \bu_I$, $q_h|_K = \Pi^{0,k-1}_K p$ followed with the use of Lemma \ref{lem:interpolant} and regularity estimate \eqref{est:reg} then we have the estimates
\begin{align*}
\vertiii{(\bu, p) - (\bu_h, p_h)}_h & \le C h^{\min \{ k,r \}} \Big(  {| \ff|^2_{\max \{ r,k\}-1,\Omega}} +  | \bu|_{1+r,\Omega}^2  + |p|_{r, \Omega}^2 \Big)^{1/2} \\
& { \le C h^{\min \{ r,k \}} (\| \ff \|_{\max \{ r,k\}-1,\Omega} +\| \bg \|_{r+\frac{1}{2}, \g_D} )},
\end{align*} 
where $C$ is independent of $\nu$ and $h$. The proof is complete.
\end{proof}

\begin{remark}
	The global inf-sup condition for the continuous Taylor-Hood virtual element spaces cannot be followed in the same manner due to the fact that $\vdiv \bV_h^k$ is a piecewise discontinuous polynomial space of degree $k-1$ while $Q_h^k$ is continuous polynomial space of degree $k-1$, and hence the error analysis for these spaces do not follow from the previous theorem and requires a different treatment though the coercivity on the kernel of $b_h(\cdot, \cdot)$.
\end{remark}




		\section{Numerical results}\label{sec:numerical-section}
		This section is devoted to explore different numerical experiments to verify the performance and robustness of the proposed method. The numerical approach involves the utilization of the Dune-Fem library \cite{dedner2010generic}, specifically the Dune-Vem module  \cite{dedner2022framework} for generating the computational spaces $\bV_h^k$ and $Q_h^k$.  We resort to the \texttt{DivfreeSpace} + $\mathbb{P}_0$ on polygonal meshes to construct these spaces. Then, similar to \cite{BLVm2an17},  a more accurate pressure is recovered by an element-wise post-processing procedure. For completeness, we represent the pressure surface plots using the $\mathbb{P}_0$ discrete approximation.
				
		In order to tackle the resulting sparse linear system $\boldsymbol{Ax}=\boldsymbol{F}$, arising from the discretization process, we employ the sparse linear solver \texttt{spsolve} from the Scipy library \cite{2020SciPy-NMeth} with UMFPACK .
		A critical parameter in our simulations is the Nitsche parameter, which is chosen to satisfy $\gamma_{D}=\gamma_{N}=\gamma$, with $\gamma\geq 100(k+1)^2$, where $k$ represents the order of the numerical scheme employed (see also \cite{bertoluzza2022weakly}).  
		
		We denote by $N$ the number of elements in the mesh.  To study the rate of convergence, we use the relations $\mathcal{O}(h) \approx \mathcal{O}(N^{-1/2})$.  By $e(\cdot)$ we  denote the error associated with the quantity $\cdot$ in its natural norm, and will denote by $h_i$ the mesh~size corresponding to a refinement level $i$. The experimental convergence order is computed as
		\[r(\cdot) = \frac{\log(e_i(\cdot))- \log(e_{i+1}(\cdot))}{\log(h_i)- \log(h_{i+1})}\approx  \frac{\log(e_i(\cdot))- \log(e_{i+1}(\cdot))}{-2(\log(N_i)- \log(N_{i+1}))}, \]
		where $N_i$ and $N_{i+1}$ are two consecutive measures of the number of elements for the refinements $h_i$ and $h_{i+1}$, respectively.
		
		When evaluating the performance of our numerical scheme, we explore different geometric configurations, thereby assessing its robustness and versatility. Triangle and quad meshes are generated using the Dune mesher or the GMSH software \cite{geuzaine2009gmsh} through the Pygmsh interface  \cite{pygmshschlomer2021}. Voronoi  and other polygonal meshes are constructed using Polymesher \cite{polymeshertalischi2012} and Matlab, respectively.
		

		For the experiments we consider that the domain is partitioned considering the following types of meshes:
		\begin{itemize}
			\item $\mathcal{T}_{1,h}:=$ A mesh using triangles
			\item $\mathcal{T}_{2,h}:=$ A mesh with quadrilaterals
			\item  $\mathcal{T}_{3,h}:=$ A mesh with a Voronoi tessellation
			\item  $\mathcal{T}_{4,h}:=$ A mesh with non-convex polygons
		\end{itemize}
		Other mesh types were considered and the results were similar to those presented below.
		
		\subsection{Convergence on a square domain}\label{subsec:squaredomain}
		We start by considering the square domain $\Omega:=(0,1)^2$ together with smooth solutions for the velocity and pressure. The right-hand side and boundary conditions are chosen such that the exact solutions are given by
		$$
		\bu(x,y):= \left(\frac{\partial \varphi}{\partial y},-\frac{\partial \varphi}{\partial x}\right), \qquad p(x,y):=\sin(x-y),
		$$
		where $\varphi(x,y):=-256x^2(x - 1)^2y(y - 1)(2y - 1)$. The velocity is solenoidal and is characterized for having high tangential components across $y=0$ and $y=1$. 
		
		The error history, together with the values for the $L^2$-norm of $\nabla\cdot\bu_h$ are presented in Tables \ref{table-square2D-presrec}-\ref{table-square2D-presrec-k3}. Here, we observe that the rates behave optimal trough the different meshes, and the velocity divergence is converging to zero. Also, we note that the scheme is not divergence-free because of the Nitsche method employed.	For comparison, the error curves for $k=2,3$ are presented in Figure \ref{fig:square2D-errorcurves}. We observe that the energy error behaves like $\mathcal{O}(h^k)$, as expected.
				
		Several computed solutions on different meshes are portrayed in Figures \ref{fig:uh-square2D}-\ref{fig:ph-square2D}. On Figure \ref{fig:uh-square2D}, the velocity components and vector field are represented in different meshes. On the other hand, Figure \ref{fig:ph-square2D} describe the behavior of the discrete pressure throughout several meshes, where we observe a good match between them.

		\begin{table}[t!]
			\setlength{\tabcolsep}{4.5pt}
			\centering 
			\caption{Example \ref{subsec:squaredomain}. Error history on different meshes with $k=2$ for $\bu_h$ and $p_h$ on the unit square domain $\Omega=(0,1)^2$.}
			{\small\begin{tabular}{|c|rcccccc|} 
					\hline\hline
					$\mathcal{T}_{i,h}$&N   &   $h$  & $\mathrm{e}(\bu)$  &   $r(\bu)$   &   $\mathrm{e}(p)$  &   $r(p)$  &  $\Vert\nabla\cdot\bu_h\Vert_{0,\Omega}$  \\
					\hline 
					\hline
					\multirow{5}{0.01\linewidth}{$1$}
					&128 & 0.177 & 1.94e+00 & $\star$ & 7.98e-01 & $\star$ & 1.47e-04  \\
					&512 & 0.088 & 6.16e-01 & 1.66 & 2.11e-01 & 1.92 & 2.69e-05  \\
					&2048 & 0.044 & 1.79e-01 & 1.78 & 5.40e-02 & 1.97 & 5.01e-06  \\
					&8192 & 0.022 & 4.76e-02 & 1.91 & 1.36e-02 & 1.99 & 8.71e-07  \\
					&32768 & 0.011 & 1.22e-02 & 1.97 & 3.41e-03 & 2.00 & 1.79e-07  \\
					\hline
					\hline
					\multirow{5}{0.01\linewidth}{$2$}
					&64 & 0.125 & 2.08e+00 & $\star$ & 1.95e+00 & $\star$ & 1.14e-03  \\
					&256 & 0.063 & 5.30e-01 & 1.97 & 5.61e-01 & 1.79 & 1.92e-04  \\
					&1024 & 0.031 & 1.33e-01 & 1.99 & 1.50e-01 & 1.91 & 3.16e-05  \\
					&4096 & 0.016 & 3.33e-02 & 2.00 & 3.84e-02 & 1.96 & 5.51e-06  \\
					&16384 & 0.008 & 8.32e-03 & 2.00 & 9.69e-03 & 1.99 & 1.07e-06  \\
					\hline
					\hline
					\multirow{5}{0.01\linewidth}{$3$}
					 &64 & 0.162 & 2.01e+00 & $\star$ & 1.28e+00 & $\star$ & 7.39e-03  \\
					&256 & 0.089 & 4.85e-01 & 2.05 & 3.26e-01 & 1.97 & 8.47e-04  \\
					&1024 & 0.046 & 1.28e-01 & 1.92 & 8.81e-02 & 1.89 & 1.24e-04  \\
					&4096 & 0.021 & 3.03e-02 & 2.08 & 1.91e-02 & 2.21 & 2.37e-05  \\
					&16384 & 0.011 & 7.54e-03 & 2.01 & 4.51e-03 & 2.08 & 3.86e-06  \\
					\hline
					\hline
					\multirow{5}{0.01\linewidth}{$4$}
					&64 & 0.156 & 2.36e+00 & $\star$ & 1.61e+00 & $\star$ & 6.24e-03  \\
					&256 & 0.078 & 6.23e-01 & 1.93 & 4.54e-01 & 1.83 & 8.49e-04  \\
					&1024 & 0.039 & 1.58e-01 & 1.98 & 1.18e-01 & 1.94 & 1.09e-04  \\
					&4096 & 0.020 & 3.97e-02 & 1.99 & 2.95e-02 & 2.00 & 1.35e-05  \\
					&16384 & 0.010 & 9.94e-03 & 2.00 & 7.27e-03 & 2.02 & 1.59e-06  \\
					\hline
					\hline
			\end{tabular}}
			\smallskip			
			\label{table-square2D-presrec}
		\end{table}
		
		\begin{table}[t!]
			\setlength{\tabcolsep}{4.5pt}
			\centering 
			\caption{Example \ref{subsec:squaredomain}. Error history on different meshes with $k=3$ for $\bu_h$ and $p_h$ on the unit square domain $\Omega=(0,1)^2$.}
			{\small\begin{tabular}{|c|rcccccc|} 
					\hline\hline
					$\mathcal{T}_{i,h}$&N   &   $h$  & $\mathrm{e}(\bu)$  &   $r(\bu)$   &   $\mathrm{e}(p)$  &   $r(p)$  &  $\Vert\nabla\cdot\bu_h\Vert_{0,\Omega}$  \\
					\hline 
					\hline
					\multirow{5}{0.01\linewidth}{$1$}
					&128 & 0.177 & 2.00e-01 & $\star$ & 7.45e-02 & $\star$ & 3.03e-05  \\
					&512 & 0.088 & 2.42e-02 & 3.05 & 1.05e-02 & 2.83 & 3.89e-06  \\
					&2048 & 0.044 & 2.96e-03 & 3.03 & 1.38e-03 & 2.93 & 1.14e-06  \\
					&8192 & 0.022 & 3.65e-04 & 3.02 & 1.76e-04 & 2.96 & 3.96e-07  \\
					&32768 & 0.011 & 4.72e-05 & 2.95 & 2.23e-05 & 2.98 & 1.39e-07  \\
					\hline
					\hline
					\multirow{5}{0.01\linewidth}{$2$}
					&64 & 0.125 & 3.52e-01 & $\star$ & 2.15e-01 & $\star$ & 1.38e-04  \\
					&256 & 0.063 & 4.08e-02 & 3.11 & 3.57e-02 & 2.59 & 1.56e-05  \\
					&1024 & 0.031 & 4.67e-03 & 3.13 & 5.04e-03 & 2.82 & 4.69e-06  \\
					&4096 & 0.016 & 5.47e-04 & 3.10 & 6.66e-04 & 2.92 & 1.66e-06  \\
					&16384 & 0.008 & 6.56e-05 & 3.06 & 8.54e-05 & 2.96 & 5.88e-07  \\
					\hline
					\hline
					\multirow{5}{0.01\linewidth}{$3$}
					&64 & 0.162 & 2.67e-01 & $\star$ & 1.08e-01 & $\star$ & 1.36e-04  \\
					&256 & 0.089 & 3.39e-02 & 2.98 & 1.45e-02 & 2.89 & 8.00e-06  \\
					&1024 & 0.046 & 4.63e-03 & 2.87 & 1.63e-03 & 3.16 & 1.71e-06  \\
					&4096 & 0.021 & 5.23e-04 & 3.15 & 1.56e-04 & 3.38 & 6.52e-07  \\
					&16384 & 0.011 & 6.62e-05 & 2.98 & 1.77e-05 & 3.14 & 2.31e-07  \\
					\hline
					\hline
					\multirow{5}{0.01\linewidth}{$4$}
					&64 & 0.156 & 2.99e-01 & $\star$ & 1.54e-01 & $\star$ & 1.06e-04  \\
					&256 & 0.078 & 3.72e-02 & 3.01 & 2.15e-02 & 2.84 & 7.86e-06  \\
					&1024 & 0.039 & 4.58e-03 & 3.02 & 2.74e-03 & 2.97 & 1.26e-06  \\
					&4096 & 0.020 & 5.65e-04 & 3.02 & 3.41e-04 & 3.00 & 4.12e-07  \\
					&16384 & 0.010 & 7.01e-05 & 3.01 & 4.25e-05 & 3.01 & 1.46e-07  \\
					\hline
					\hline
			\end{tabular}}
			\smallskip			
			\label{table-square2D-presrec-k3}
		\end{table}
		\begin{figure}[!hbt]
			\centering
			\begin{minipage}{0.49\linewidth}\centering
			\includegraphics[scale=0.5,,trim= 0cm 0cm 0cm 0cm,clip]{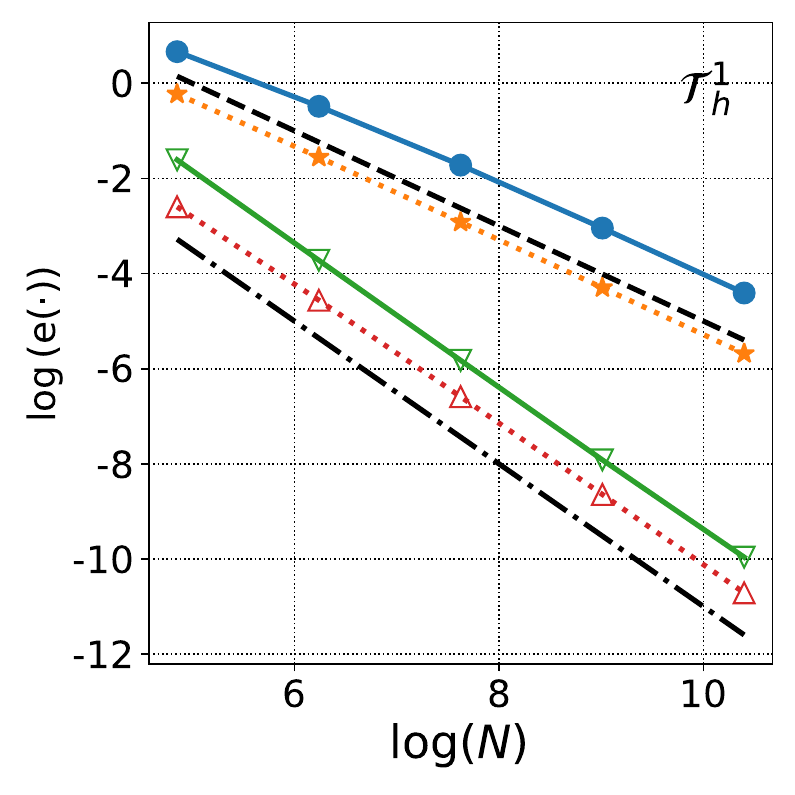}
		\end{minipage}
		\begin{minipage}{0.49\linewidth}\centering
			\includegraphics[scale=0.5,,trim= 0cm 0cm 0cm 0cm,clip]{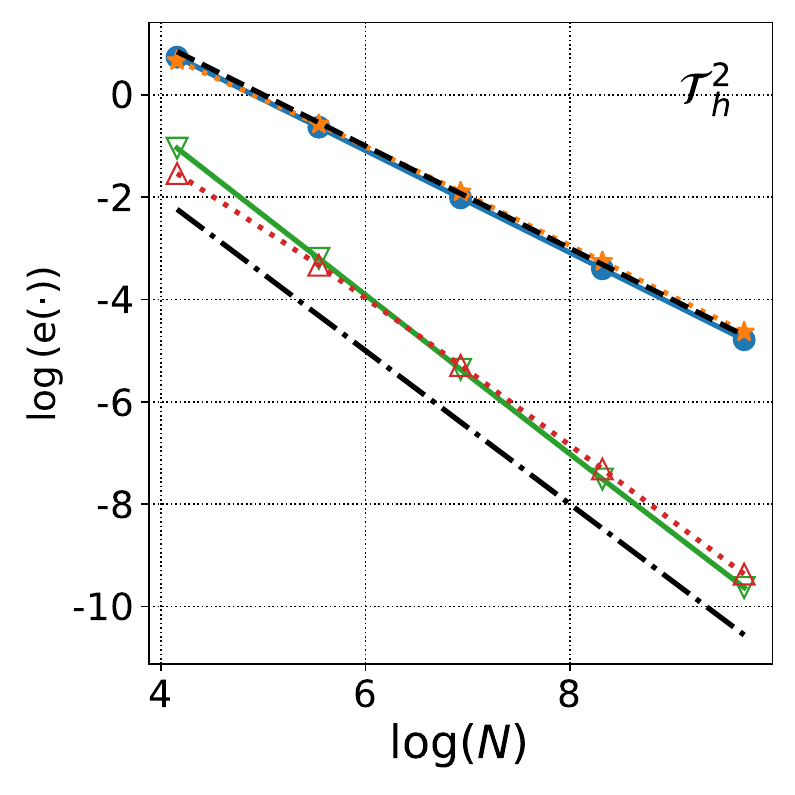}
		\end{minipage}\\
		\begin{minipage}{0.49\linewidth}\centering
			\includegraphics[scale=0.5,,trim= 0cm 0cm 0cm 0cm,clip]{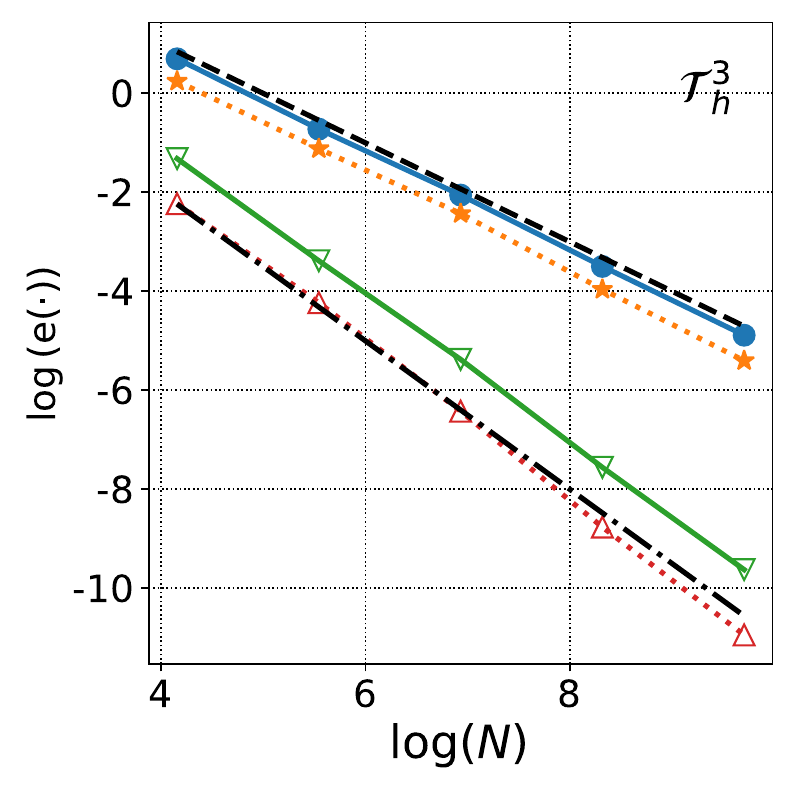}
		\end{minipage}
		\begin{minipage}{0.49\linewidth}\centering
			\includegraphics[scale=0.5,,trim= 0cm 0cm 0cm 0cm,clip]{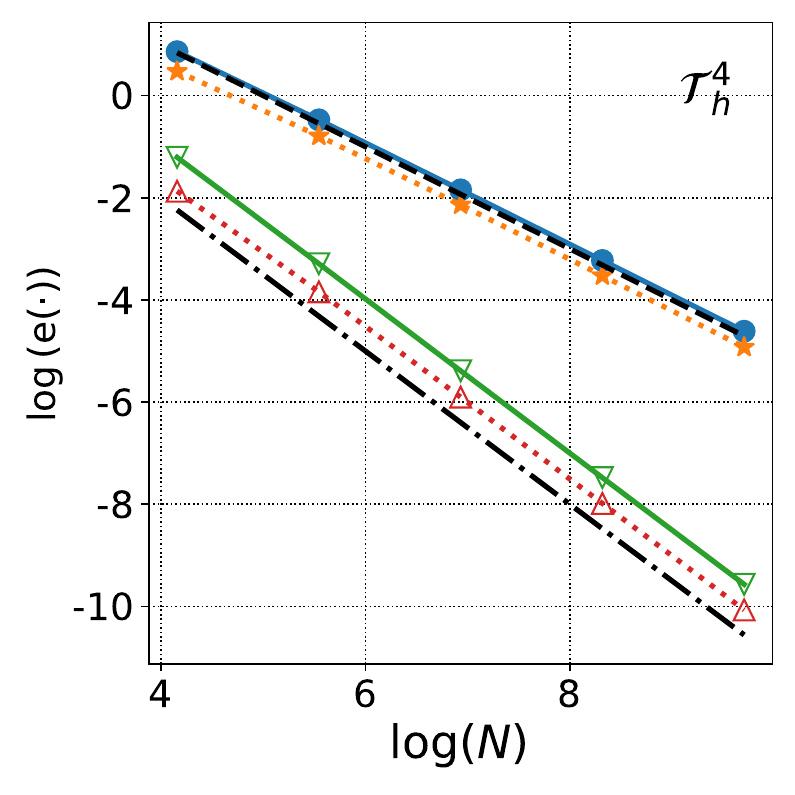}
		\end{minipage}\\
		\begin{minipage}{\linewidth}\centering
			\includegraphics[scale=0.32,trim= 0.8cm 2.2cm 0cm 2.2cm,clip]{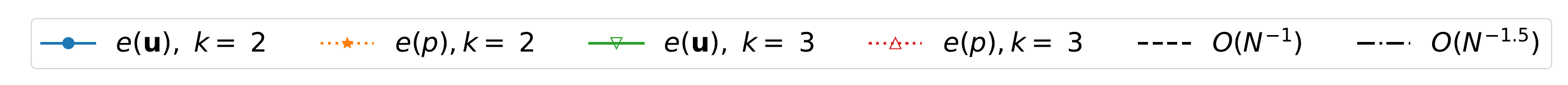}
		\end{minipage}
		\caption{Test \ref{subsec:squaredomain}.  Error curves of the virtual element scheme for the Brinkman equations using different meshes. Here, we set the parameters $\mathbb{K}=\mathbb{I}$, and $\nu=1$. }
		\label{fig:square2D-errorcurves}
		\end{figure}

		\begin{figure}[hbt!]
			\centering
			\begin{minipage}{0.32\linewidth}\centering
				{\footnotesize $\bu_{1,h}$}\\
				\includegraphics[scale=0.17, trim= 16cm 5cm 16cm 2.85cm,clip]{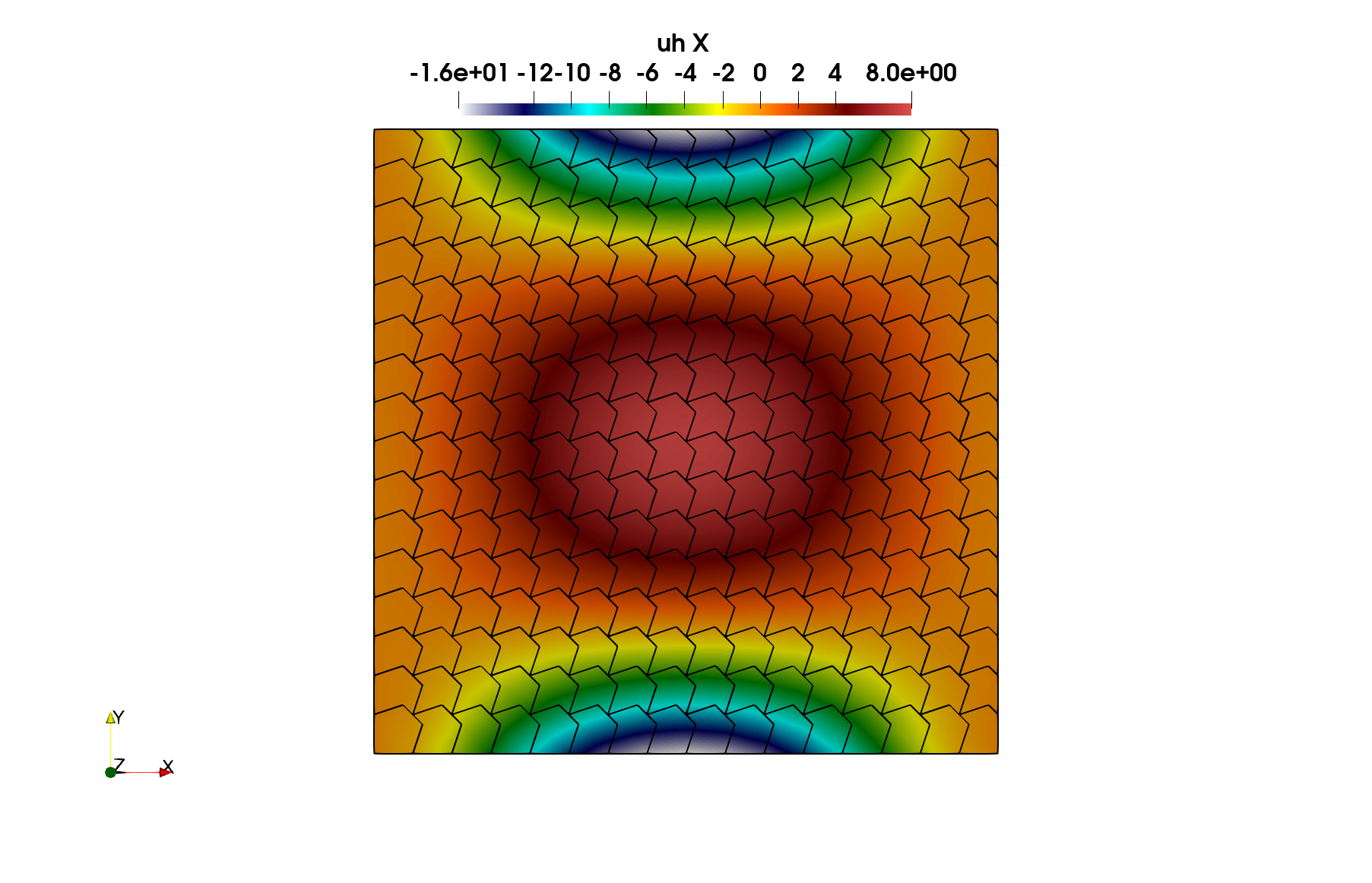}\\
			\end{minipage}
			\begin{minipage}{0.32\linewidth}\centering
				{\footnotesize $\bu_{2,h}$}\\
				\includegraphics[scale=0.17, trim= 16cm 5cm 16cm 2.85cm,clip]{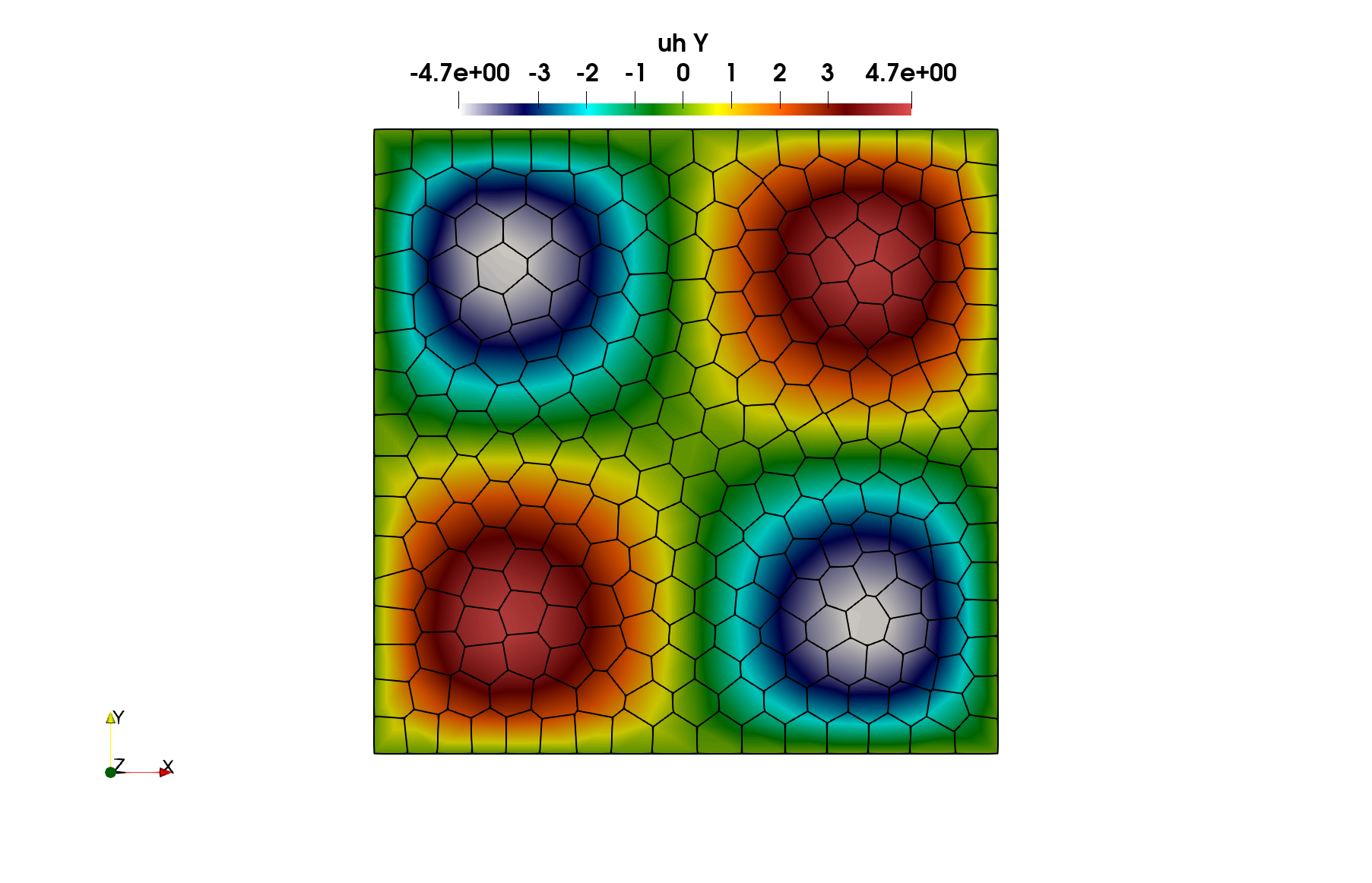}\\
			\end{minipage}
			\begin{minipage}{0.32\linewidth}\centering
				{\footnotesize $|\bu_h|$}\\
				\includegraphics[scale=0.17, trim= 16cm 5cm 16cm 2.85cm,clip]{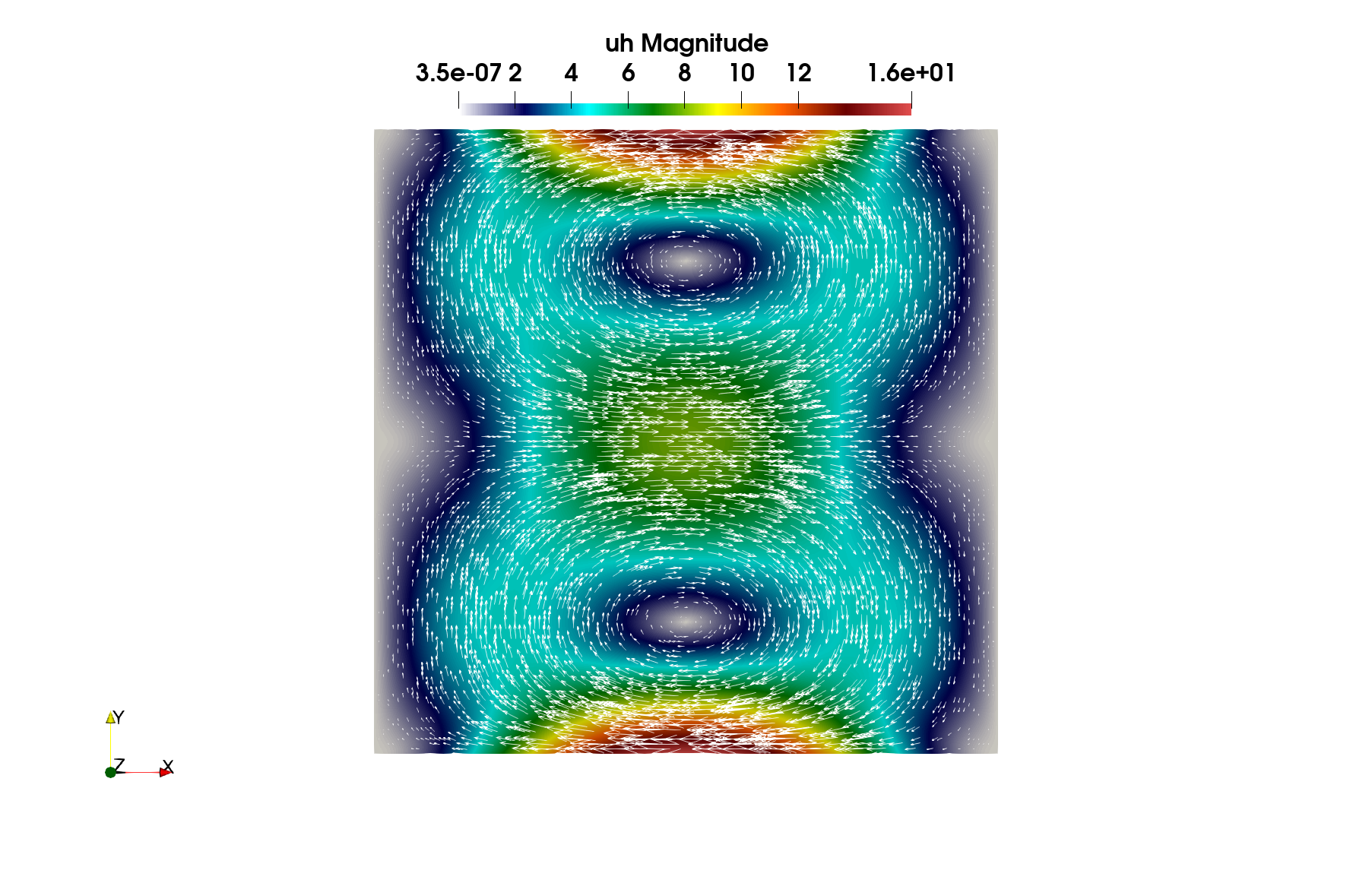}\\
			\end{minipage}
			\caption{Test \ref{subsec:squaredomain}. Scalar components of the computed velocity $\bu_{i,h}, i=1,2$ in different meshes together with the corresponding field $\bu_h$. }
			\label{fig:uh-square2D}
		\end{figure}

		\begin{figure}[hbt!]
			\centering
			\begin{minipage}{0.32\linewidth}\centering
				{\footnotesize $p_h$}\\
				\includegraphics[scale=0.17, trim= 16cm 5cm 16cm 2.85cm,clip]{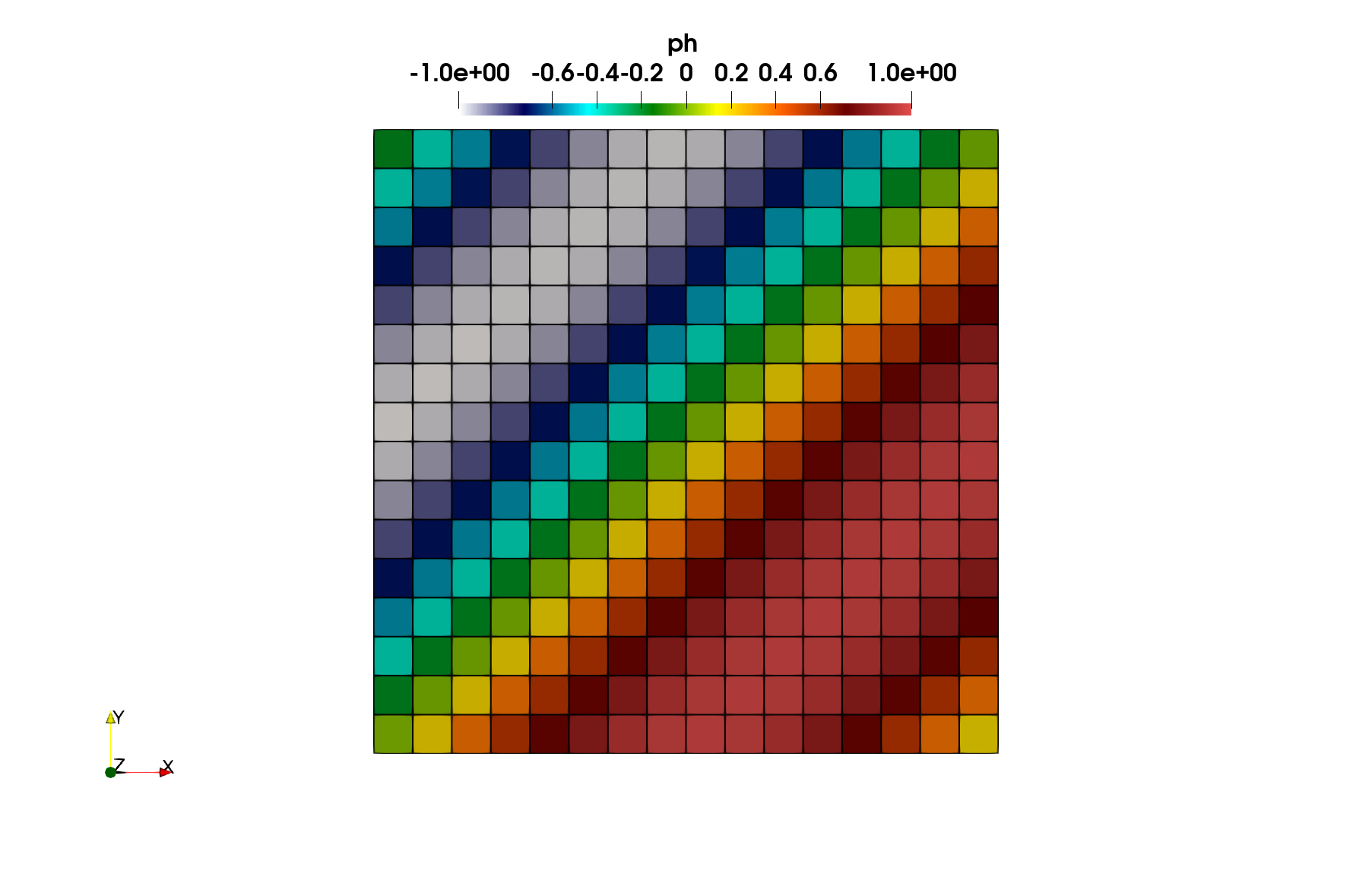}\\
			\end{minipage}
			\begin{minipage}{0.32\linewidth}\centering
				{\footnotesize $p_h$}\\
				\includegraphics[scale=0.17, trim= 16cm 5cm 16cm 2.85cm,clip]{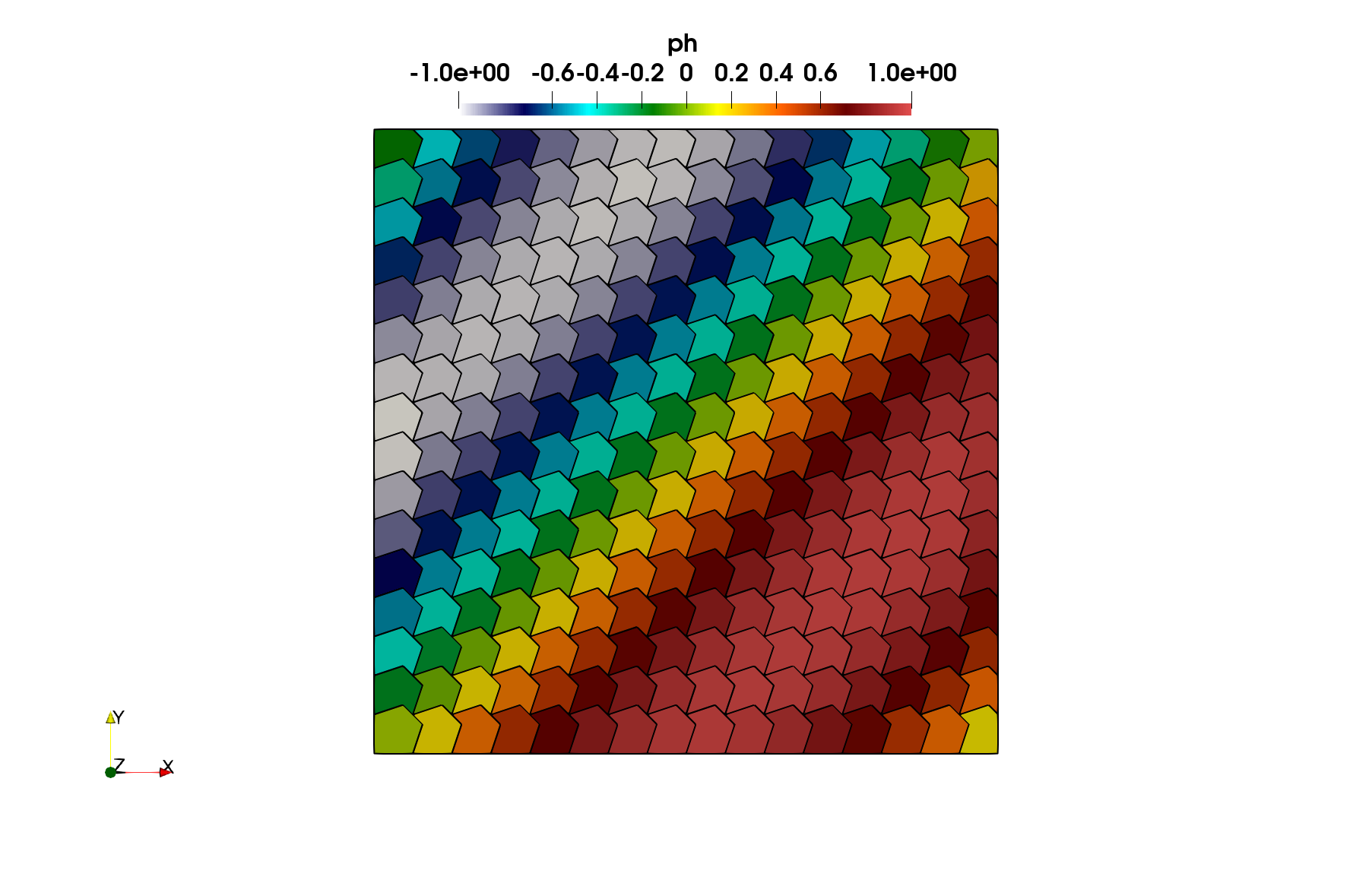}\\
			\end{minipage}
			\begin{minipage}{0.32\linewidth}\centering
				{\footnotesize $p_h$}\\
				\includegraphics[scale=0.17, trim= 16cm 5cm 16cm 2.85cm,clip]{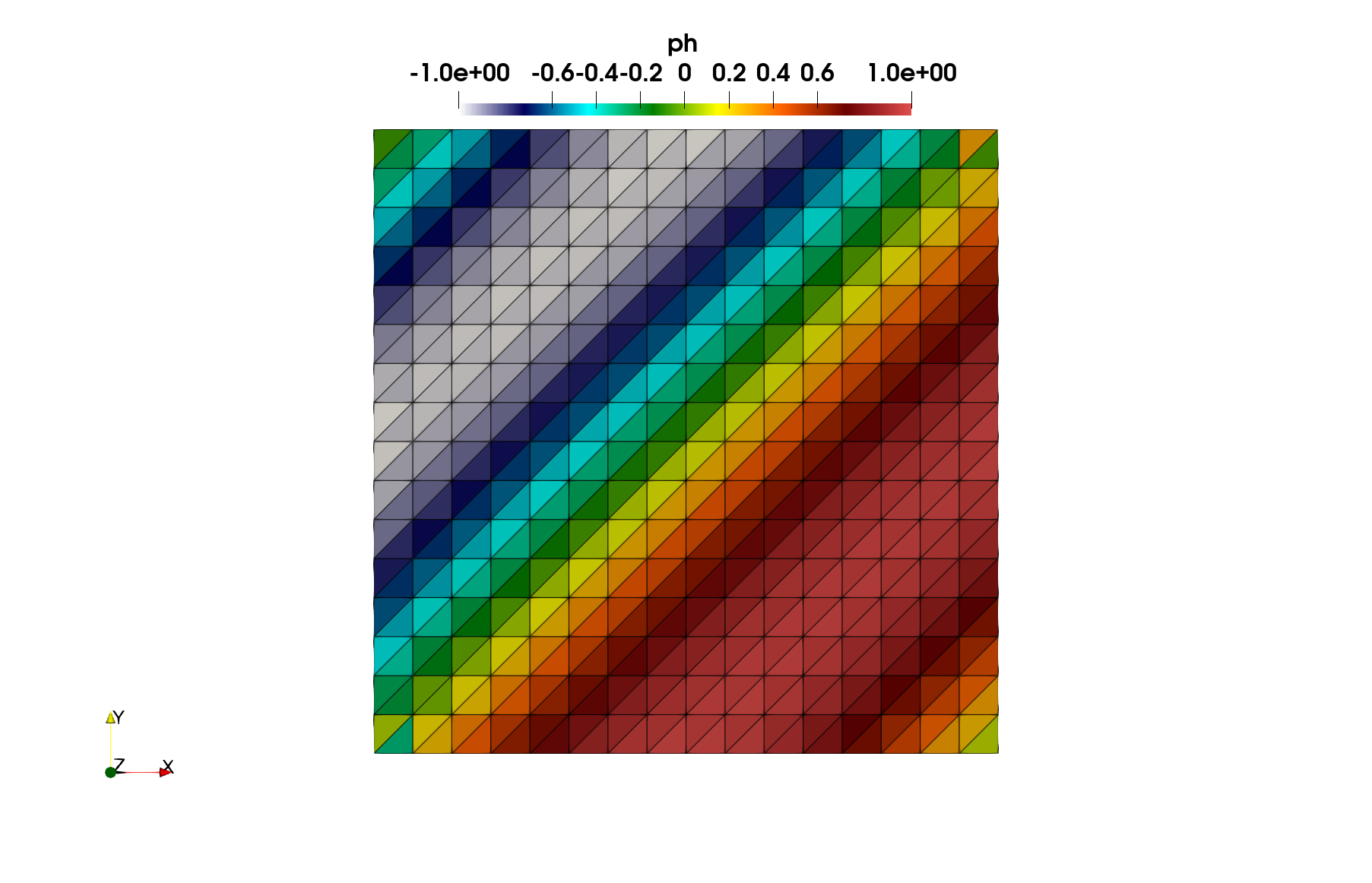}\\
			\end{minipage}
			\caption{Test \ref{subsec:squaredomain}. Comparison between the computed pressure $p_h$ on different polygonal meshes.  }
			\label{fig:ph-square2D}
		\end{figure}
		
		\subsection{Robustness with respect to $\nu$}\label{subsec:robustness-nu}
		In this section we aim to test the robustness of the scheme with respect to the viscosity. More precisely, we observe the error behavior when $\nu$ becomes small. For simplicity, the exact solutions are the ones given in section \ref{subsec:squaredomain}, and we take $\mathcal{T}=\mathcal{T}_{3,h}$ (Voronoi mesh).
		
		We consider a given viscosity $\nu=10^{-3j},\, j=1,2,3,4,$ and $\mathbb{K}=\mathbb{I}$. The computed errors and experimental rates of convergence are given in Tables \ref{table-square2D-robustness-nu-k2}-\ref{table-square2D-robustness-nu-k3}. We observe small variations in the computed errors between taking $\nu=10^{-3}$ and the rest, that do not affect the convergence rate. This can also be said about the present results and the ones given in Section \ref{subsec:squaredomain} for $\mathcal{T}_{3,h}$. The pattern is repeated for the scheme orders $k=2,3$, confirming the theoretical robustness with respect to the viscosity predicted in Section \ref{sec:abstract-error-analysis}.
		\begin{table}[t!]
			\setlength{\tabcolsep}{4.5pt}
			\centering 
			\caption{Example \ref{subsec:robustness-nu}. Error history  with respect to different values of $\nu$ on the unit square domain for $\bu_h$ and $p_h$. Here, we consider $k=2$ as the scheme order.}
			{\small\begin{tabular}{|c|rcccccc|} 
					\hline\hline
					$\nu$&N   &   $h$  & $\mathrm{e}(\bu)$  &   $r(\bu)$   &   $\mathrm{e}(p)$  &   $r(p)$  &  $\Vert\nabla\cdot\bu_h\Vert_{0,\Omega}$  \\
					\hline 
					\hline
					\multirow{5}{0.04\linewidth}{$10^{-3}$}
					&64 & 0.162 & 2.19e+00 & $\star$ & 9.48e-02 & $\star$ & 7.04e-03  \\
					&256 & 0.089 & 4.95e-01 & 2.14 & 3.02e-02 & 1.65 & 8.14e-04  \\
					&1024 & 0.046 & 1.28e-01 & 1.96 & 6.72e-03 & 2.17 & 1.10e-04  \\
					&4096 & 0.021 & 3.03e-02 & 2.08 & 1.62e-03 & 2.05 & 1.92e-05  \\
					&16384 & 0.011 & 7.54e-03 & 2.00 & 3.68e-04 & 2.14 & 1.10e-06  \\
					\hline
					\hline
					\multirow{5}{0.04\linewidth}{$10^{-6}$}
					&64 & 0.162 & 2.37e+00 & $\star$ & 9.40e-02 & $\star$ & 7.04e-03  \\
					&256 & 0.089 & 5.61e-01 & 2.08 & 3.00e-02 & 1.65 & 8.14e-04  \\
					&1024 & 0.046 & 1.49e-01 & 1.91 & 6.66e-03 & 2.17 & 1.10e-04  \\
					&4096 & 0.021 & 3.37e-02 & 2.14 & 1.61e-03 & 2.05 & 1.92e-05  \\
					&16384 & 0.011 & 8.27e-03 & 2.03 & 3.65e-04 & 2.14 & 1.10e-06  \\
					\hline
					\hline
					\multirow{5}{0.04\linewidth}{$10^{-9}$}
					&64 & 0.162 & 2.37e+00 & $\star$ & 9.40e-02 & $\star$ & 7.04e-03  \\
					&256 & 0.089 & 5.61e-01 & 2.08 & 3.00e-02 & 1.65 & 8.14e-04  \\
					&1024 & 0.046 & 1.49e-01 & 1.91 & 6.66e-03 & 2.17 & 1.10e-04  \\
					&4096 & 0.021 & 3.39e-02 & 2.14 & 1.61e-03 & 2.05 & 1.92e-05  \\
					&16384 & 0.011 & 8.44e-03 & 2.01 & 3.65e-04 & 2.14 & 1.10e-06  \\
					\hline
					\hline
					\multirow{5}{0.04\linewidth}{$10^{-12}$}
					&64 & 0.162 & 2.37e+00 & $\star$ & 9.40e-02 & $\star$ & 7.04e-03  \\
					&256 & 0.089 & 5.61e-01 & 2.08 & 3.00e-02 & 1.65 & 8.14e-04  \\
					&1024 & 0.046 & 1.49e-01 & 1.91 & 6.66e-03 & 2.17 & 1.10e-04  \\
					&4096 & 0.021 & 3.39e-02 & 2.14 & 1.61e-03 & 2.05 & 1.92e-05  \\
					&16384 & 0.011 & 8.44e-03 & 2.01 & 3.65e-04 & 2.14 & 1.10e-06  \\
					\hline
					\hline
			\end{tabular}}
			\smallskip			
			\label{table-square2D-robustness-nu-k2}
		\end{table}
		
			\begin{table}[t!]
			\setlength{\tabcolsep}{4.5pt}
			\centering 
			\caption{Example \ref{subsec:robustness-nu}. Error history  with respect to different values of $\nu$ on the unit square domain for $\bu_h$ and $p_h$. Here, we consider $k=3$ as the scheme order.}
			{\small\begin{tabular}{|c|rcccccc|} 
					\hline\hline
					$\nu$&N   &   $h$  & $\mathrm{e}(\bu)$  &   $r(\bu)$   &   $\mathrm{e}(p)$  &   $r(p)$  &  $\Vert\nabla\cdot\bu_h\Vert_{0,\Omega}$  \\
					\hline 
					\hline
					\multirow{5}{0.04\linewidth}{$10^{-3}$}
					&64 & 0.162 & 2.74e-01 & $\star$& 1.15e-02 & $\star$& 1.02e-04  \\
					&256 & 0.089 & 2.96e-02 & 3.21 & 1.44e-03 & 3.00 & 7.42e-06  \\
					&1024 & 0.046 & 4.06e-03 & 2.87 & 1.59e-04 & 3.18 & 1.75e-06  \\
					&4096 & 0.021 & 5.01e-04 & 3.02 & 1.52e-05 & 3.38 & 6.56e-07  \\
					&16384 & 0.011 & 6.54e-05 & 2.94 & 1.66e-06 & 3.20 & 2.32e-07  \\
					\hline
					\hline
					\multirow{5}{0.04\linewidth}{$10^{-6}$}
					&64 & 0.162 & 3.47e-01 & $\star$& 1.14e-02 & $\star$& 1.02e-04  \\
					&256 & 0.089 & 4.45e-02 & 2.96 & 1.43e-03 & 3.00 & 7.42e-06  \\
					&1024 & 0.046 & 5.96e-03 & 2.90 & 1.58e-04 & 3.18 & 1.75e-06  \\
					&4096 & 0.021 & 6.80e-04 & 3.13 & 1.51e-05 & 3.38 & 6.56e-07  \\
					&16384 & 0.011 & 8.09e-05 & 3.07 & 1.65e-06 & 3.20 & 2.32e-07  \\
					\hline
					\hline
					\multirow{5}{0.04\linewidth}{$10^{-9}$}
					&64 & 0.162 & 3.47e-01 & $\star$& 1.14e-02 & $\star$& 1.02e-04  \\
					&256 & 0.089 & 4.46e-02 & 2.96 & 1.43e-03 & 3.00 & 7.42e-06  \\
					&1024 & 0.046 & 6.00e-03 & 2.89 & 1.58e-04 & 3.18 & 1.75e-06  \\
					&4096 & 0.021 & 6.93e-04 & 3.11 & 1.51e-05 & 3.38 & 6.56e-07  \\
					&16384 & 0.011 & 8.68e-05 & 3.00 & 1.65e-06 & 3.20 & 2.32e-07  \\
					\hline
					\hline
					\multirow{5}{0.04\linewidth}{$10^{-12}$}
					&64 & 0.162 & 3.47e-01 & $\star$& 1.14e-02 & $\star$& 1.02e-04  \\
					&256 & 0.089 & 4.46e-02 & 2.96 & 1.43e-03 & 3.00 & 7.42e-06  \\
					&1024 & 0.046 & 6.00e-03 & 2.89 & 1.58e-04 & 3.18 & 1.75e-06  \\
					&4096 & 0.021 & 6.93e-04 & 3.11 & 1.51e-05 & 3.38 & 6.56e-07  \\
					&16384 & 0.011 & 8.68e-05 & 3.00 & 1.65e-06 & 3.20 & 2.32e-07  \\
					\hline
					\hline
			\end{tabular}}
			\smallskip			
			\label{table-square2D-robustness-nu-k3}
		\end{table}
		
		\subsection{Applications with several boundary conditions}
		This experiment aims to test the behavior of our method in applications where several geometric conditions are assumed.  Let $\boldsymbol{g}_1$ and $\boldsymbol{g}_2$ be two vector fields such that
		$$
		\bu\cdot\boldsymbol{n}=\boldsymbol{g}_1\cdot\nn, \qquad (\beps(\bu))\boldsymbol{n}\cdot \bt = \boldsymbol{g}_2\cdot\bt, \quad \text{ on } \Gamma_N.
		$$		
		Then, the discrete right-hand sides forms with additional Nitsche terms  are given as follow (cf. Section \ref{subsec:VEspaces}).
		$$
		\begin{aligned}
		\widetilde{F}_h(\bv_h) &:= F_h(\bv_h) +  \gamma_{N} \sum_{e \in \cE_h^N}\int_e h_e^{-1} (\boldsymbol{g}_1\cdot\nn) (\bv_h \cdot \nn) \ds - \sum_{e \in \cE_h^N}\int_e (\nn^t(\nu \beps( {\Pi^{\beps,k}_{K_e}} \bv_h) \nn)) (\boldsymbol{g}_1\cdot\nn) \ds \\
		&\hspace{2cm}+  \sum_{e \in \cE_h^N}\int_e (\boldsymbol{g}_2\cdot\bt) (\bv_h \cdot \bt) \ds
	\end{aligned}
	$$
	$$
		\widetilde{G}(q_h) := G(q_h) +  \sum_{e \in \cE_h^N} \int_e  \boldsymbol{g}_1 \cdot (q_h \nn)\ds.
	$$
		We divide the test in three cases, that we detail below.
		
		\subsubsection{Flow past cylinder}\label{subsec:flow_past_cylinder}
		We first focuse on the simulation of a cross-flow around a cylinder between two paralel plates. This phenomena can be characterized in a 2D model, for which we take $R=0.05$ as the radius of the cylinder, and the length and height of the channel are given by $L=0.82$ and $H=0.41$, respectively. The resulting domain is of the form
		$$
		\Omega:= (0,L)\times (0,H) \backslash D_R(x_c,y_c),
		$$
		where $D_R$ is the disk of radius $R$ and centered at $(x_c,y_c)=(0.2,0.2)$. We divide the boundaries as $\partial\Omega:=\Gamma_{\text{in}}\cup\Gamma_{\text{wall}}\cup\Gamma_{\text{circ}}\cup\Gamma_{\text{out}}$, corresponding to the inlet, walls, inner circle and outlet boundary, respectively. Each of these subdomains have the following data. 
		$$
		\begin{aligned}
			&\bu = \boldsymbol{0}, \quad \text{ on } \Gamma_w\cup\Gamma_{\text{circ}},\\
			&\bu\cdot\nn = \boldsymbol{g}_1\cdot\nn,\quad(\nu\beps(\bu))\nn\cdot\bt = \boldsymbol{g}_1\cdot\bt,  \quad \text{ with }\boldsymbol{g}_1(x,y)=
			\left(\frac{6u_{\max}y(H-y)}{H^2},0\right)^{\texttt{t}}, \quad\text{ on } \Gamma_{\text{in}},\\
			&(\nu\beps(\bu) - p\mathbb{I})\nn = \boldsymbol{0}, \quad \text{on } \Gamma_{\text{out}},
		\end{aligned}
		$$
		with $u_{\max}=2$. The viscosity is taken to be $\nu=1$.  Note that $\boldsymbol{g}_1$ corresponds to a Poiseuille flow and the condition over $\Gamma_{\text{out}}$ is a free flow boundary condition, which is beyond the scope of the proposed theory.
		
		We have depicted the computed velocity data in Figure \ref{fig:flow-cylinder-velocity}. Here, we observe that the boundary conditions are applied correctly, having no-slip around the inner circle and the walls. Also, by comparing the plots of $\boldsymbol{u}_{1,h}$ and $\boldsymbol{u}_{2,h}$ in the outflow we observe that is similar to that of Stokes and Navier-stokes on a channel that ends at $\Gamma_{\text{out}}$. To observe the behavior at the inlet, we have also plotted the velocity profiles in Figure \ref{fig:flow_cylinder_velocities-inlet}. Here, we observe that the parabolic and zero-tangential stress profiles are properly imposed. We finish this part of the experiment showing the surface plot of the pressure drop on two different meshes in Figure \ref{fig:flow-cylinder-pressure}.
		\begin{figure}[!hbt]\centering
		\begin{minipage}{0.49\linewidth}\centering
			{\footnotesize $\bu_{1,h}$}\\
			\includegraphics[scale=0.13, trim= 8cm 16cm 6cm 13.5cm, clip]{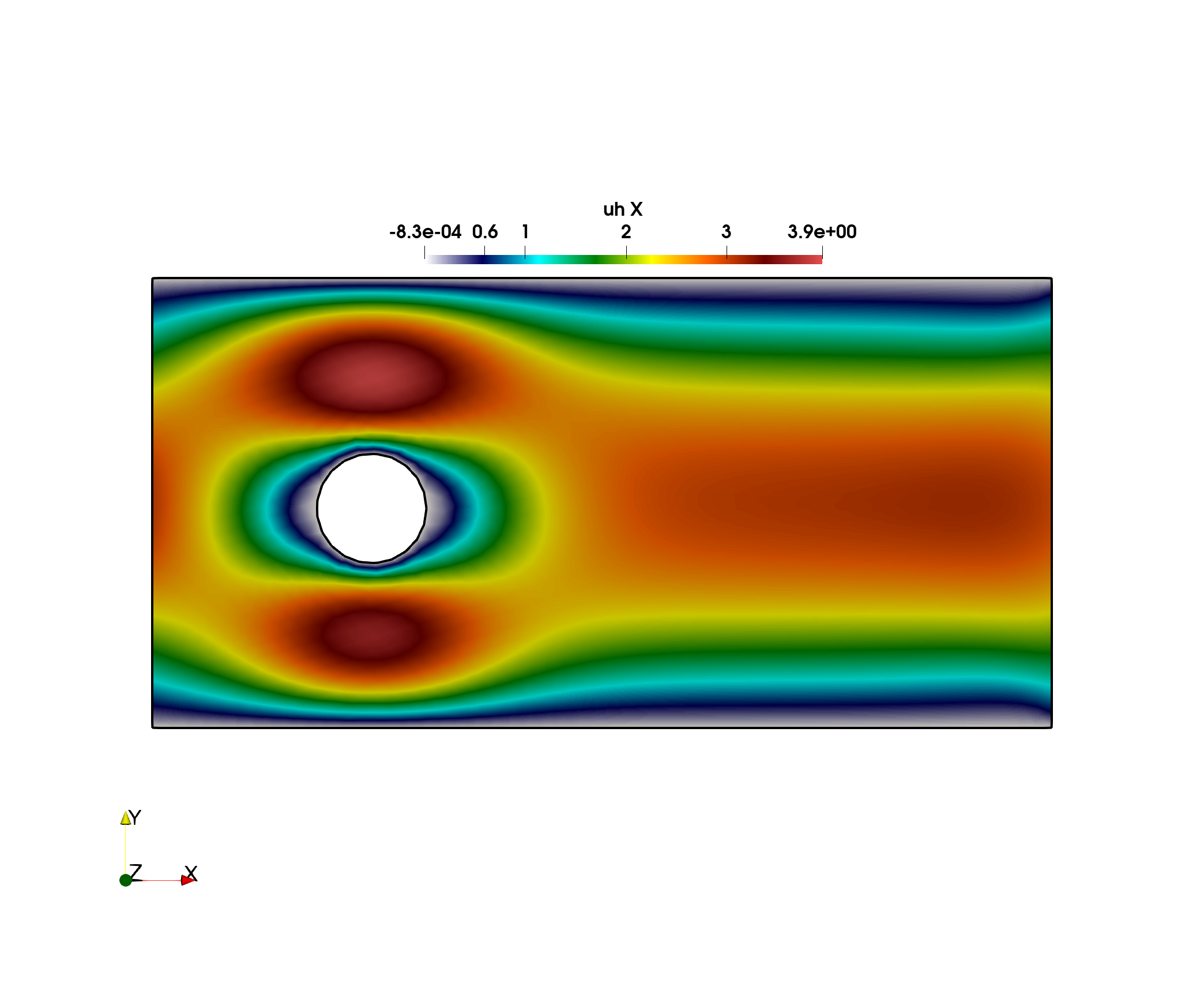}
		\end{minipage}
		\begin{minipage}{0.49\linewidth}\centering
			{\footnotesize $\bu_{2,h}$}\\
			\includegraphics[scale=0.13, trim= 8cm 16cm 6cm 13.5cm, clip]{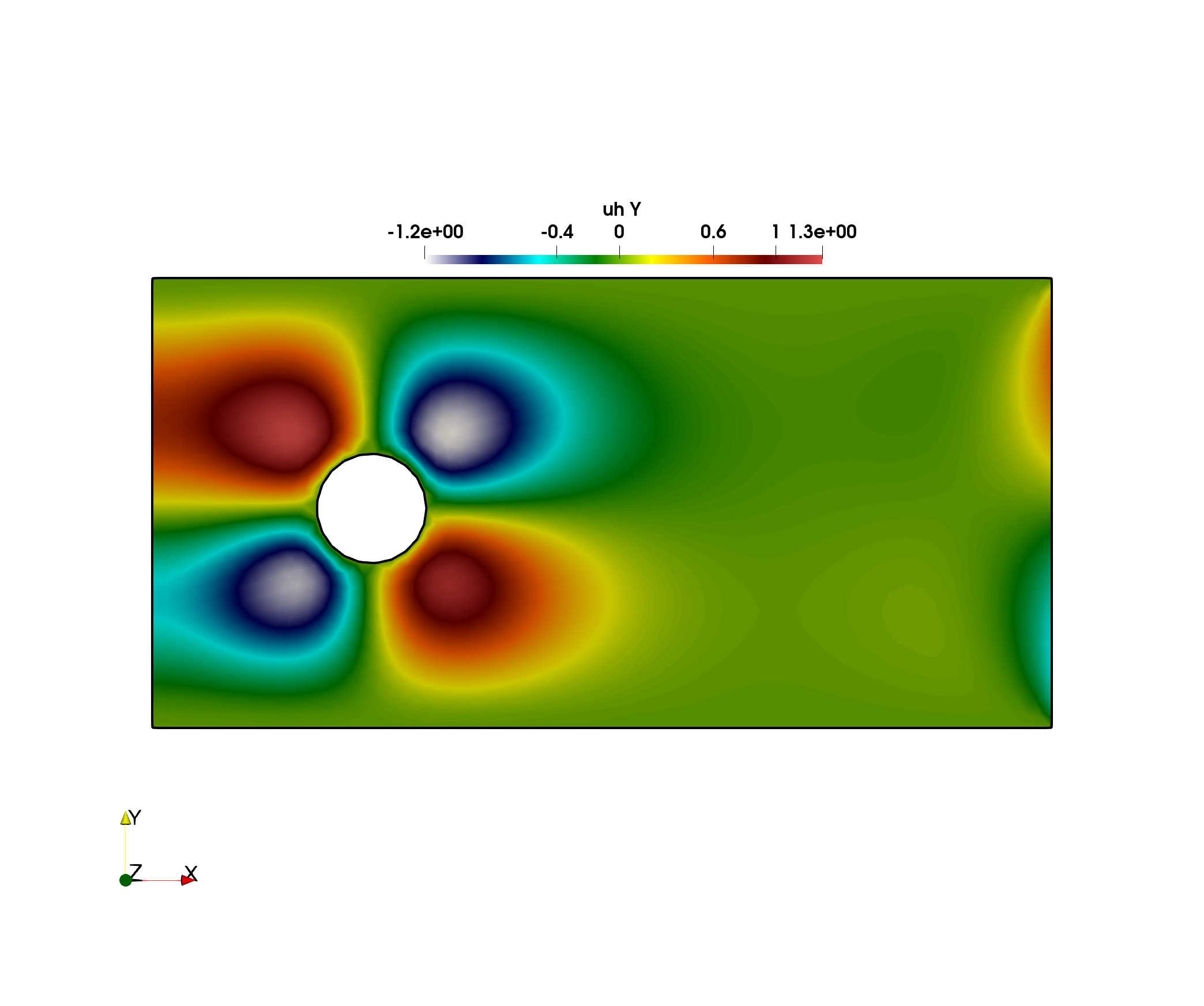}
		\end{minipage}\\
		\begin{minipage}{0.49\linewidth}\centering
			{\footnotesize $|\bu_{h}|$}\\
			\includegraphics[scale=0.13, trim= 8cm 16cm 6cm 13.5cm, clip]{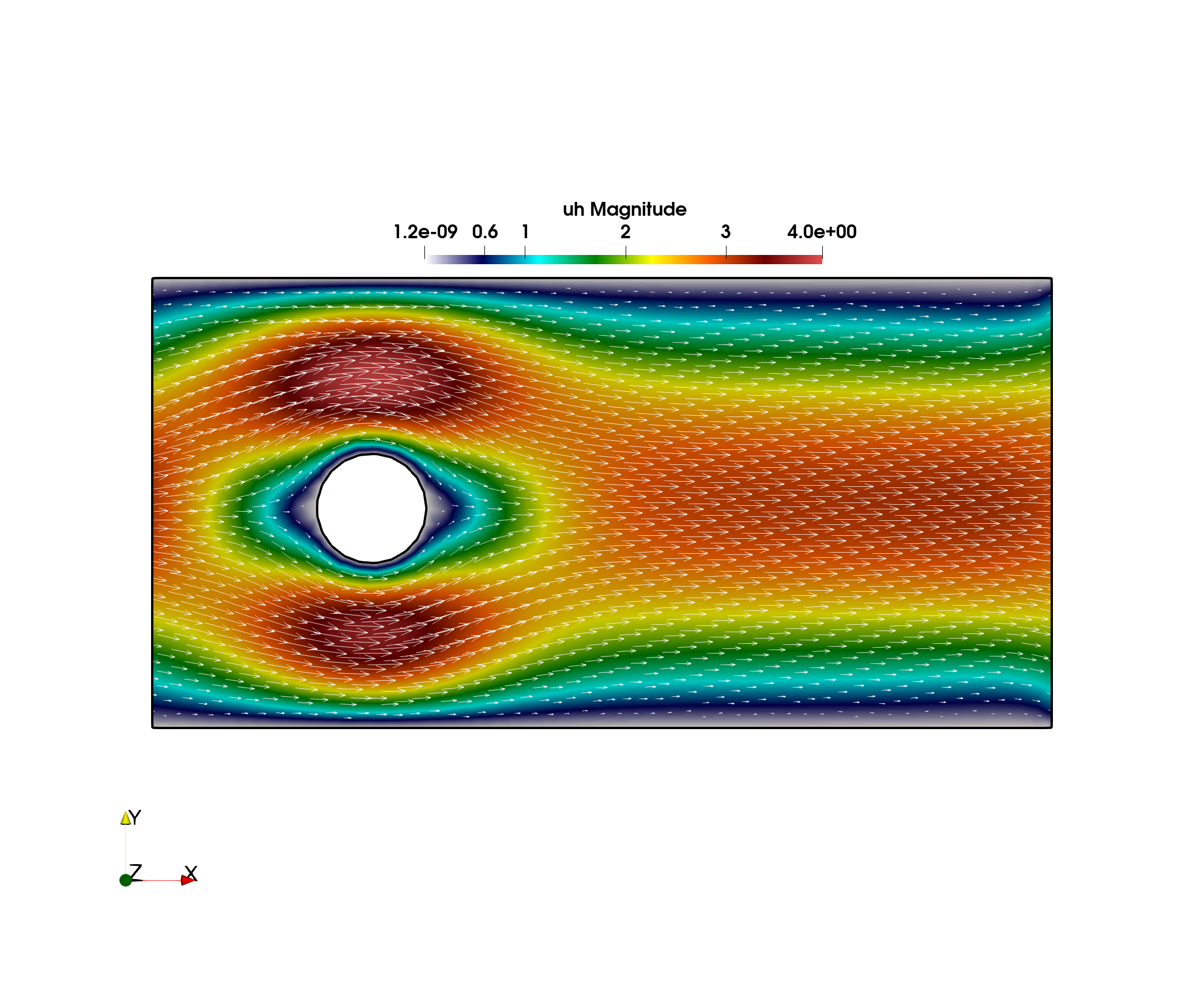}
		\end{minipage}
		\begin{minipage}{0.49\linewidth}\centering
			{\footnotesize $|\bu_{h}|$}\\
			\includegraphics[scale=0.13, trim= 8cm 16cm 6cm 13.5cm, clip]{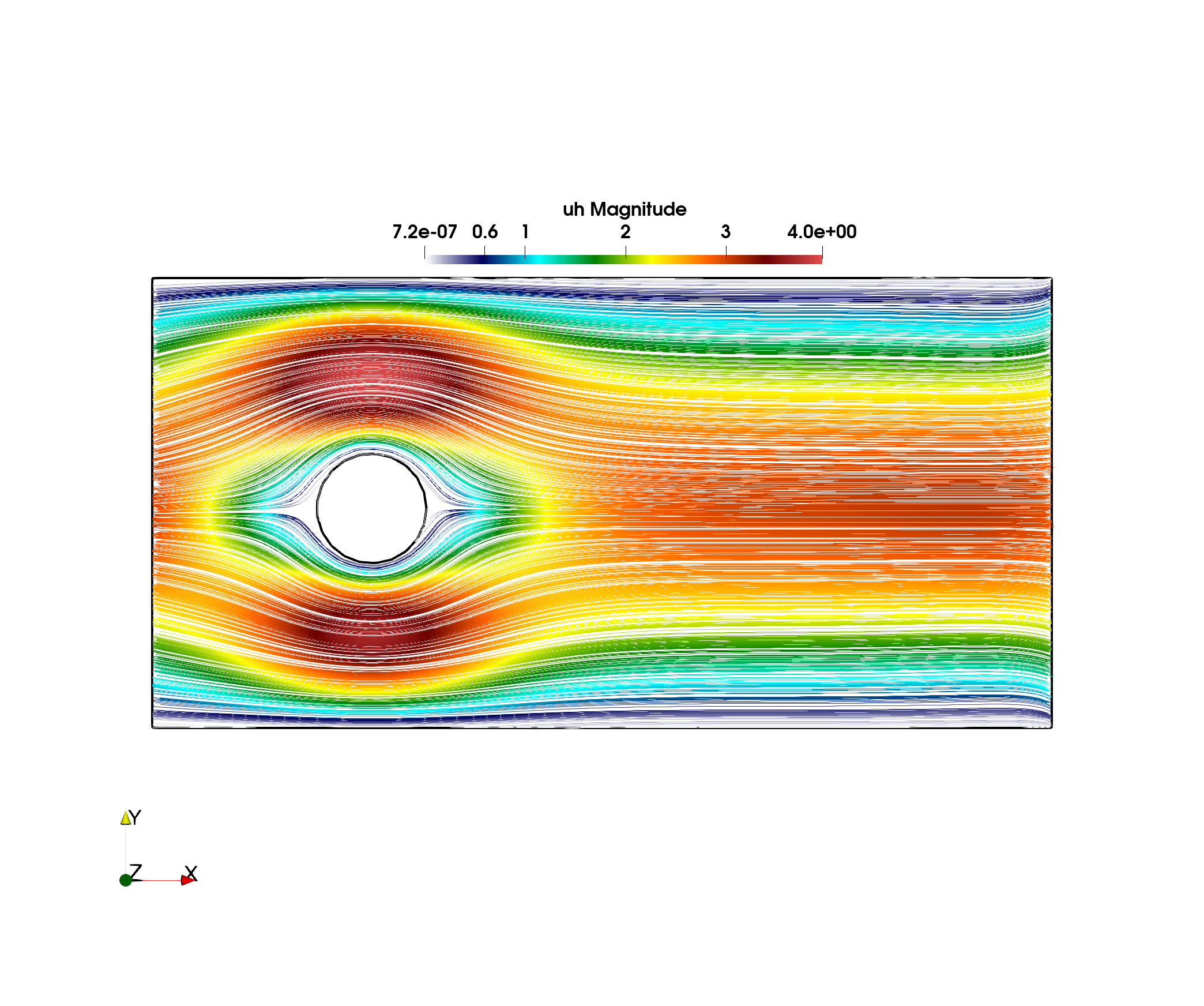}
		\end{minipage}
		\caption{Test \ref{subsec:flow_past_cylinder}. Computed  scalar components of the velocity field (top) along with the corresponding velocity field and streamlines (bottom) in the cross-flow model.}
		\label{fig:flow-cylinder-velocity}
		\end{figure}
		
		\begin{figure}[!hbt]\centering
			\begin{minipage}{0.49\linewidth}\centering
			\includegraphics[scale=0.4]{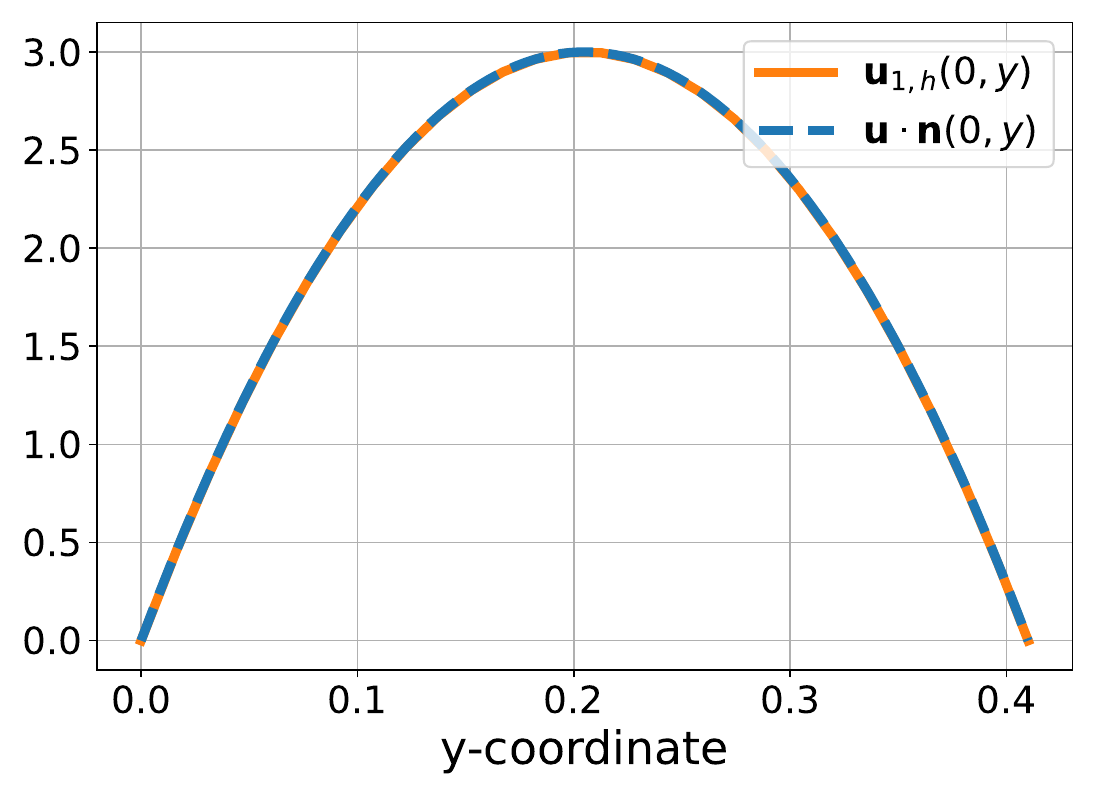}
			\end{minipage}
			\begin{minipage}{0.49\linewidth}\centering
				\includegraphics[scale=0.4]{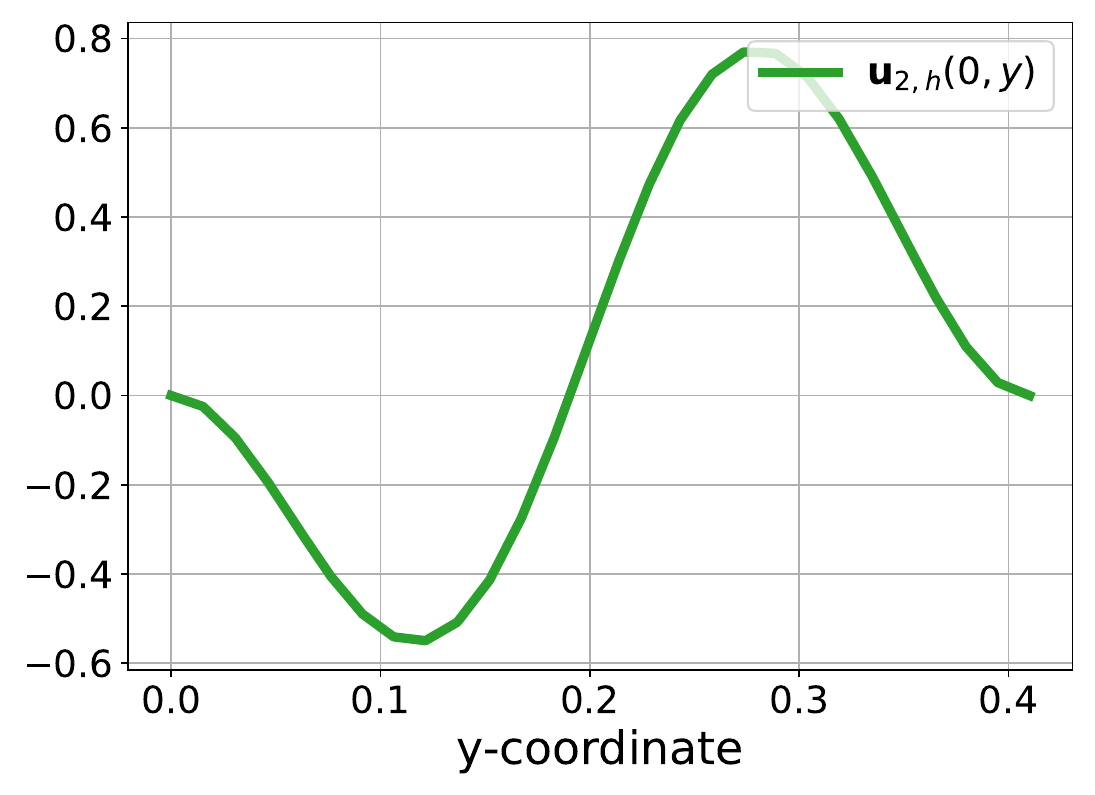}
			\end{minipage}
			\caption{Test \ref{subsec:flow_past_cylinder}. Plots of the resulting computed inlet velocity components weakly imposed at $\Gamma_{\text{in}}$. The behavior of $\bu_{2,h}$ results from imposing $(\nu\beps(\bu))\nn\cdot\bt = \boldsymbol{g}_1\cdot\bt$.}
			\label{fig:flow_cylinder_velocities-inlet}
		\end{figure}
		
		\begin{figure}[!hbt]\centering
			\begin{minipage}{0.49\linewidth}\centering
				{\footnotesize $p_h,\; \mathcal{T}_{1,h}$}\\
				\includegraphics[scale=0.13, trim= 8cm 16cm 6cm 13.7cm, clip]{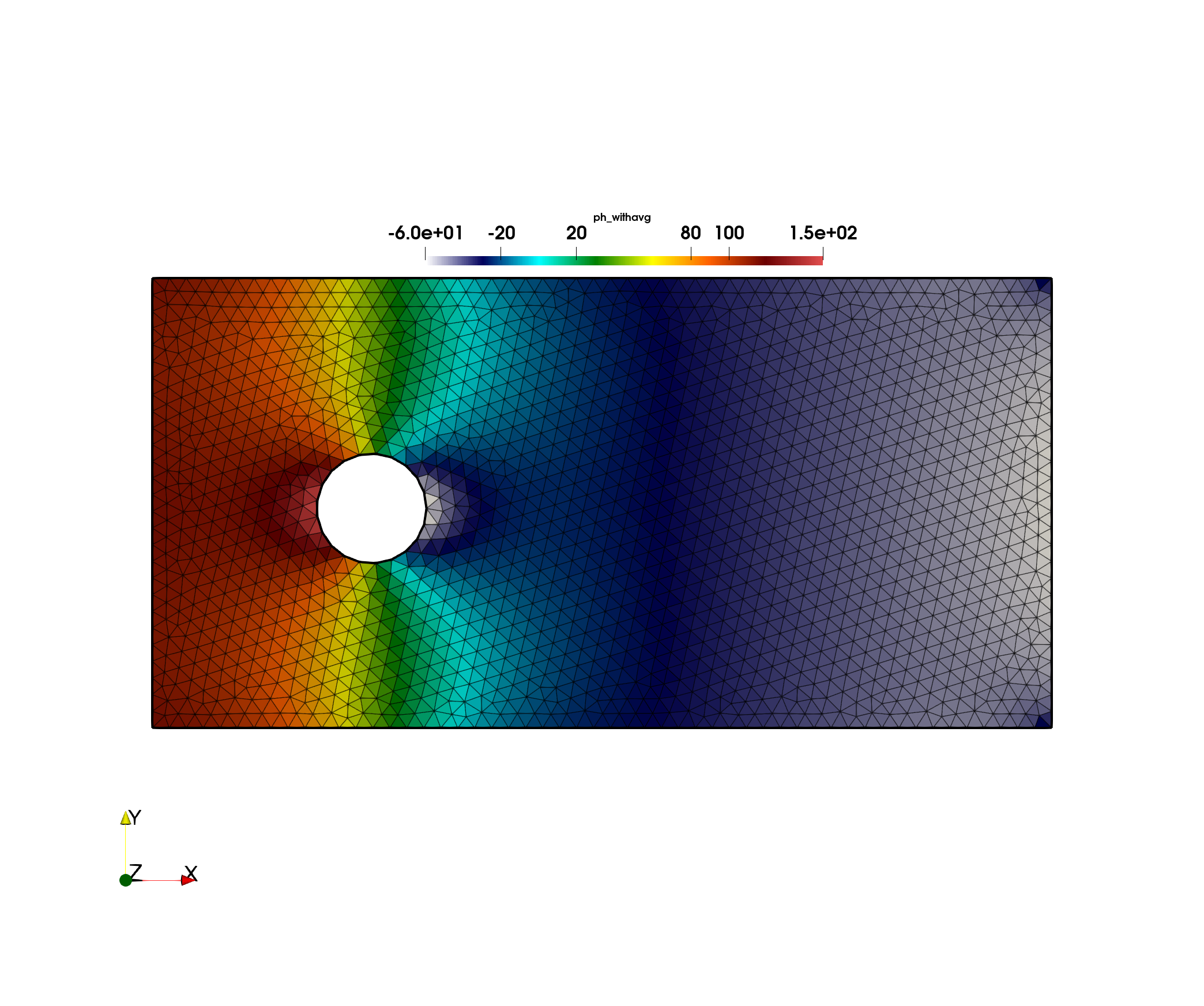}
			\end{minipage}
			\begin{minipage}{0.49\linewidth}\centering
				{\footnotesize $p_h,\; \mathcal{T}_{3,h}$}\\
				\includegraphics[scale=0.13, trim= 8cm 16cm 6cm 13.7cm, clip]{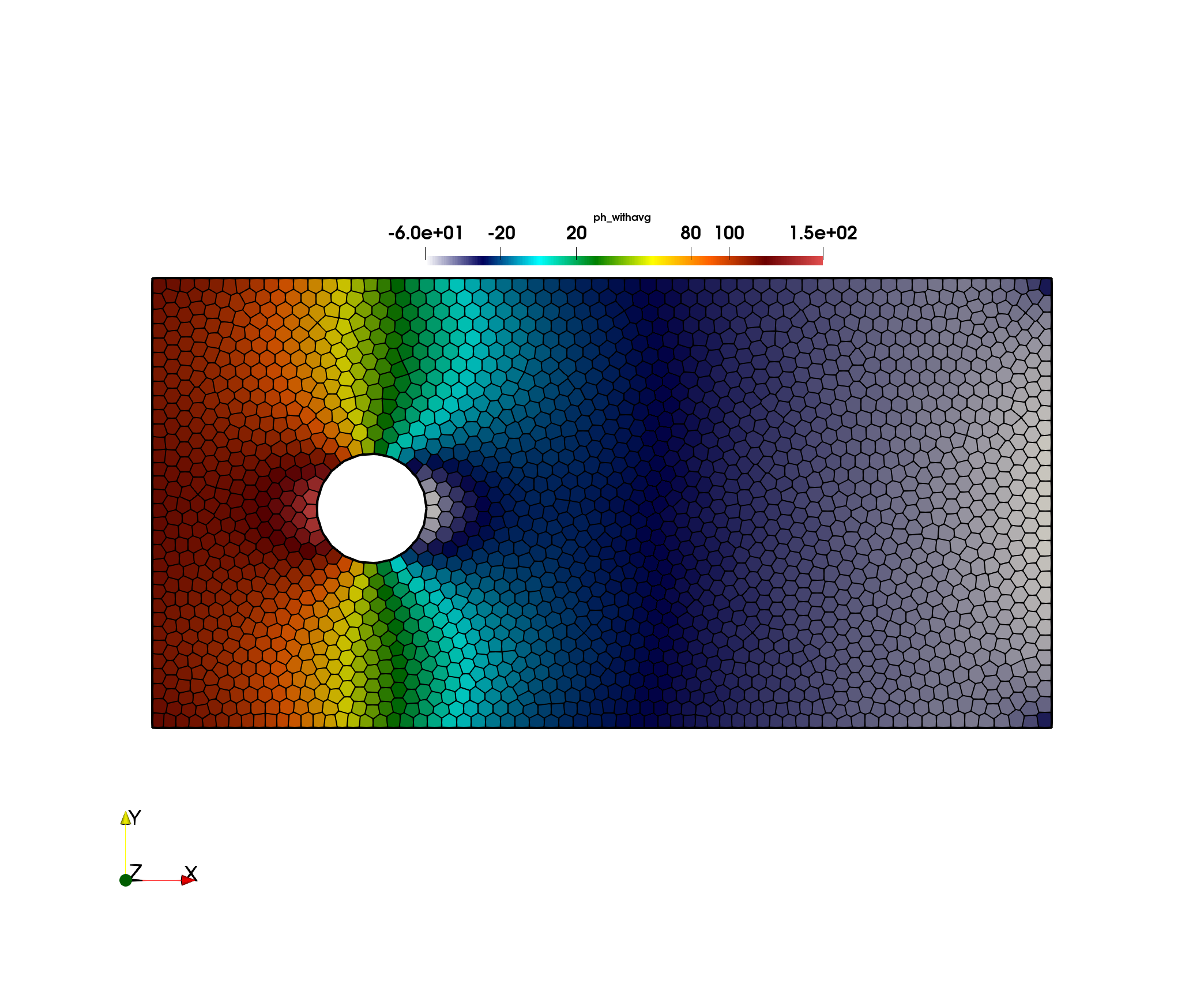}
			\end{minipage}
			\caption{Test \ref{subsec:flow_past_cylinder}. Comparison between computed pressures in the cross-flow model using $\mathcal{T}_{1,h}$ and $\mathcal{T}_{3,h}$. }
			\label{fig:flow-cylinder-pressure}
		\end{figure}

		\subsubsection{Flow into a backwards facing step}\label{subsec:backward_facing_step}
		A classical benchmark test for Stokes, Navier-Stokes or Brinkman equations is the backward facing step. Let us consider $H:=1$ and a channel given by
		$$
		\Omega:=(0,9H)\times (0,2H) \backslash \{(0,2H)\times (0,H)\},
		$$
		where we have a backward facing step in $(2H,H)$. We consider as the inlet region the left boundary, where we weakly impose 
		$$\bu\cdot\nn = \boldsymbol{g}_1\cdot\nn,\quad(\nu\beps(\bu))\nn\cdot\bt = \boldsymbol{g}_1\cdot\bt,  \quad \text{with }\quad\boldsymbol{g}_1(x,y)=
		\left(u_{\max} - 0.5(y-1.5H)^ 2,0\right)^{\texttt{t}}.$$
		As before, the inflow consist of a Poiseuille type flow, with $u_{\max}=0.125$. In the outlet (right boundary), we impose a zero stress boundary condition $(\nu\boldsymbol{\epsilon}(\bu) - p\mathbb{I})\boldsymbol{n}=\boldsymbol{0}$, and no-slip boundary conditions on the rest of $\partial\Omega$. The viscosity is taken to be $\nu=1$. The experiment was carried in two different Voronoi meshes: one with 4096 elements, and the other with  16384. The behavior for both meshes where similar, and we present the results with the coarsest mesh. The physical parameters were taken as $\mathbb{K}= \mathbb{I}$ and $\nu=1$. The approximated velocity field and pressure are presented in Figure \ref{fig:backward-facing-step}. Here, we have some well-known phenomena that include singularities at the reentrant corner and recirculation zones. The negative pressure values is valid due to the incompressible formulation of the problem, which allows to write the variable $p$ in terms of the gauge pressure. To observe the weak imposition of the boundary conditions at the inlet, we depict the velocity components in Figure~\ref{fig:backward-facing-step-inlet}, where we have a parabolic behavior of the first velocity component together with the tangential stress provided by the Poiseuille flow.
		
		\begin{figure}[!hbt]
			\centering
			\begin{minipage}{1\linewidth}\centering
				{\footnotesize \hspace{1.5cm}$\bu_{2,h}$}\\
				\includegraphics[scale=0.18, trim= 6cm 23cm 6cm 21.5cm, clip]{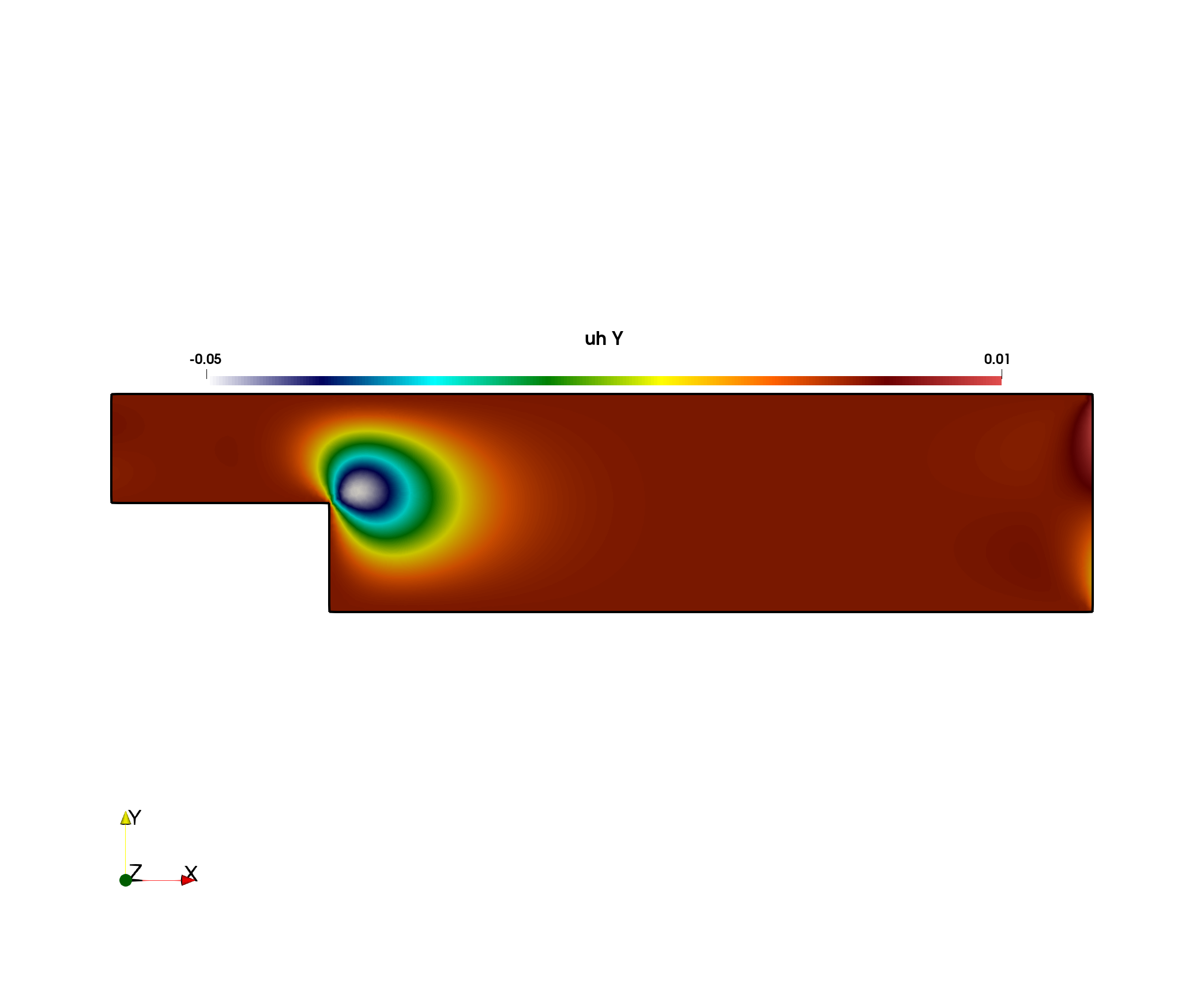}
			\end{minipage}\\
			\begin{minipage}{1\linewidth}\centering
				{\footnotesize \hspace{1.5cm}$|\bu_h|$}\\
				\includegraphics[scale=0.18, trim= 6cm 23cm 6cm 21.5cm, clip]{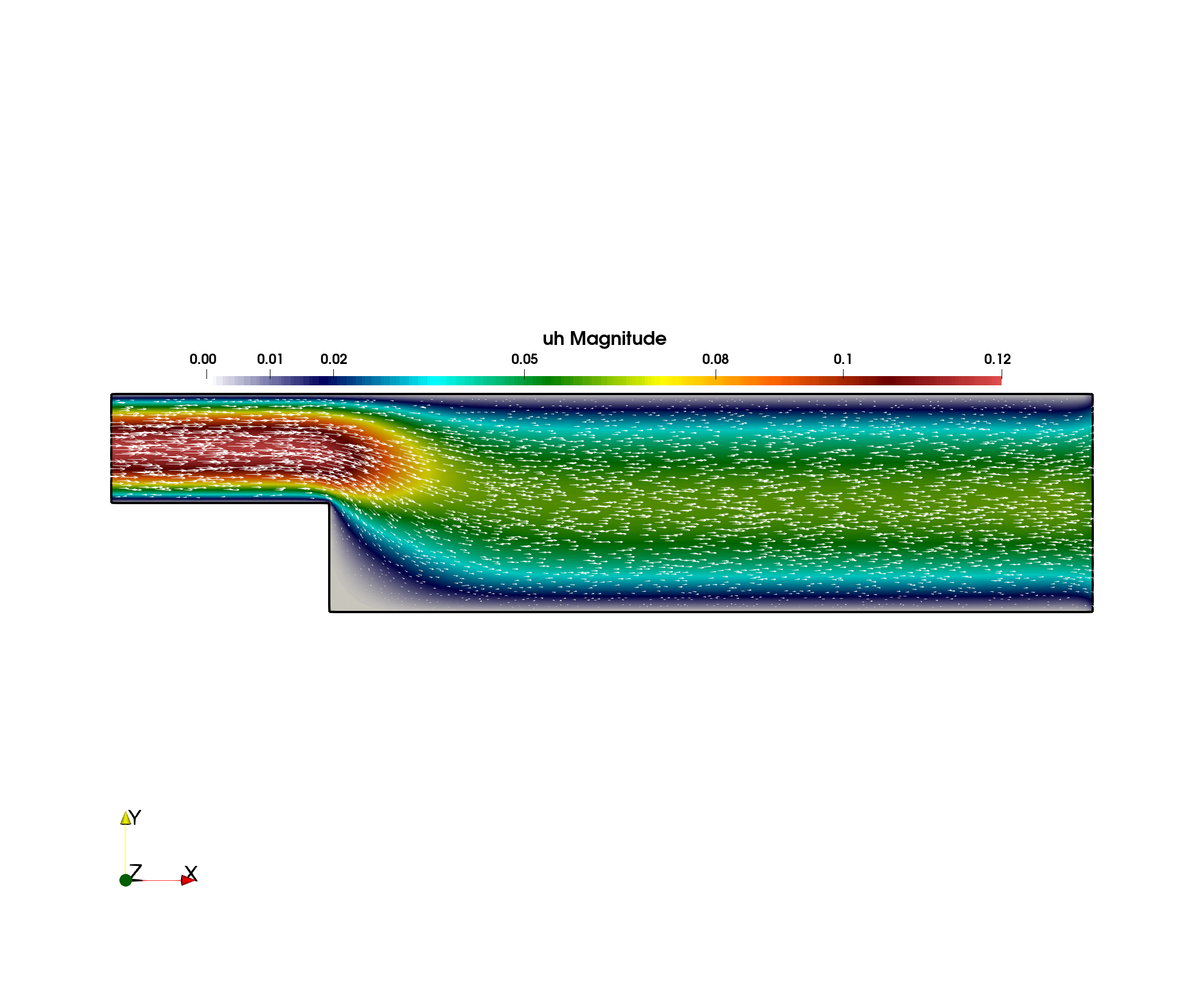}
			\end{minipage}\\
			\begin{minipage}{1\linewidth}\centering
				{\footnotesize \hspace{1.5cm}$|\bu_{h}|$}\\
				\includegraphics[scale=0.18, trim= 6cm 23cm 6cm 21.5cm, clip]{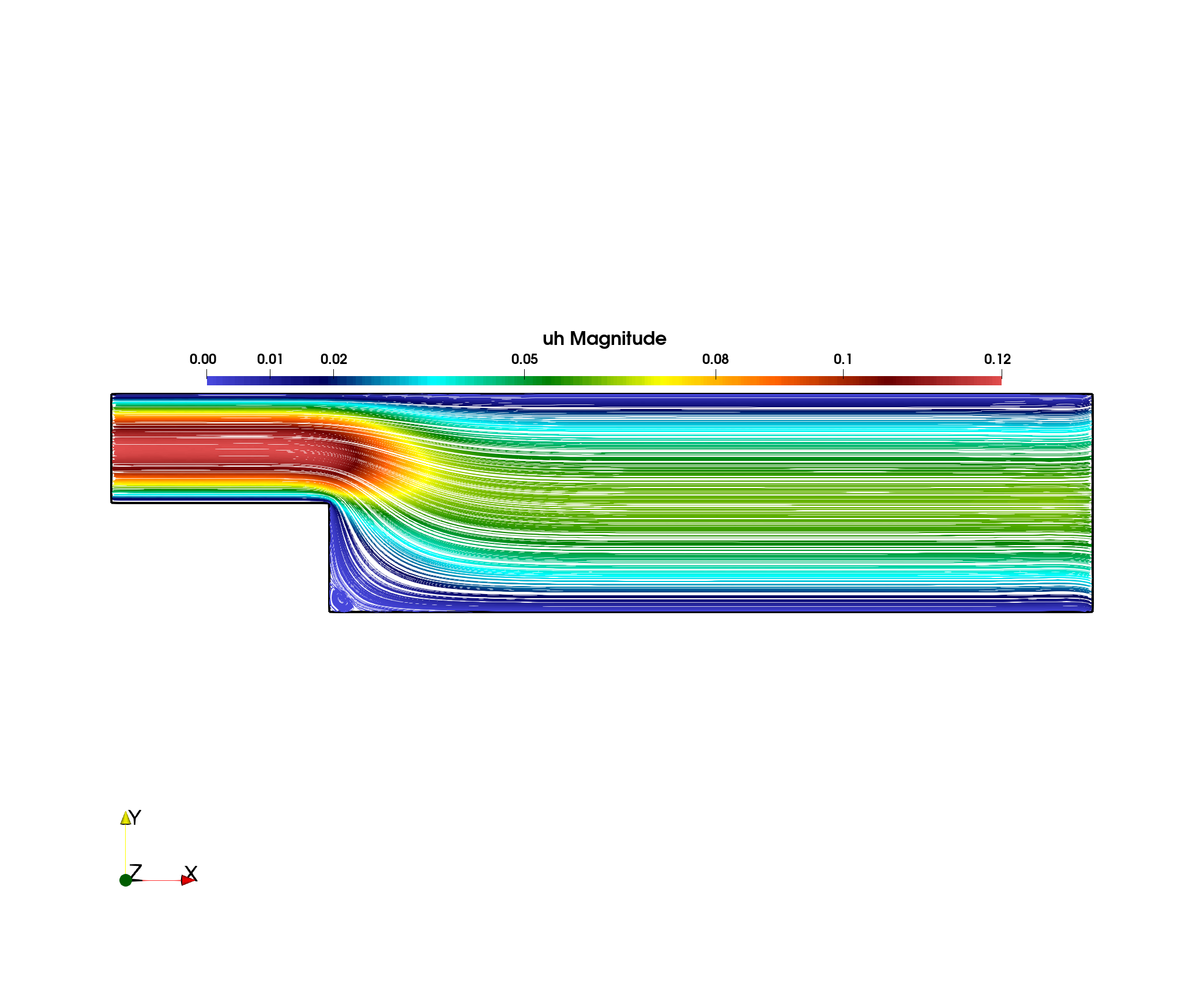}
			\end{minipage}\\
			\begin{minipage}{1\linewidth}\centering
				{\footnotesize \hspace{1.5cm}$p_h$}\\
				\includegraphics[scale=0.18, trim= 6cm 23cm 6cm 21.5cm, clip]{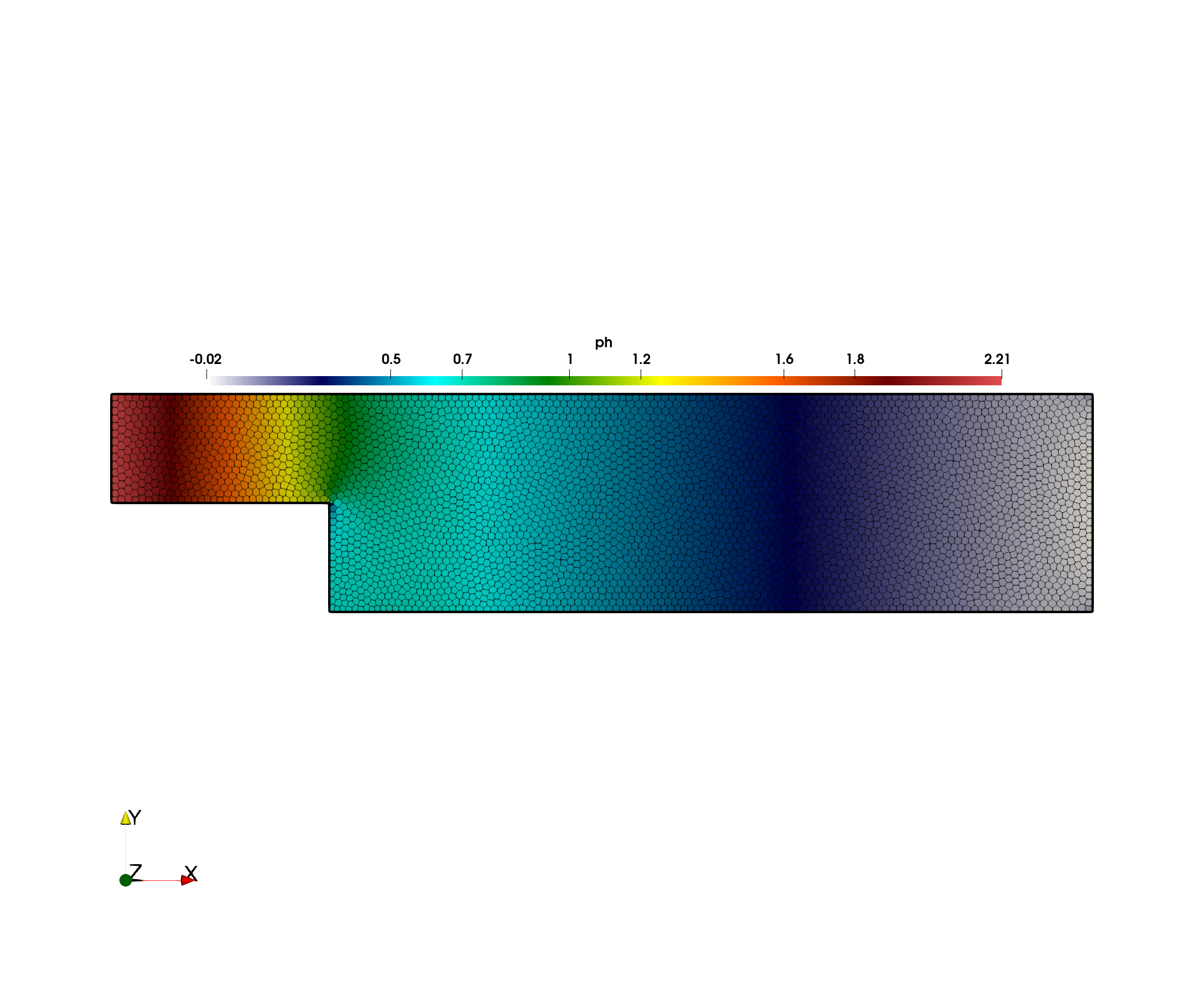}
			\end{minipage}\\
			\caption{Test \ref{subsec:backward_facing_step}. Second velocity compoment (top) Velocity field streamlines (two in the middle) with pressure distribution (bottom), computed in a Voronoi mesh with 4096 elements.}
			\label{fig:backward-facing-step}
		\end{figure}
		
			\begin{figure}[!hbt]\centering
			\begin{minipage}{0.49\linewidth}\centering
				\includegraphics[scale=0.4]{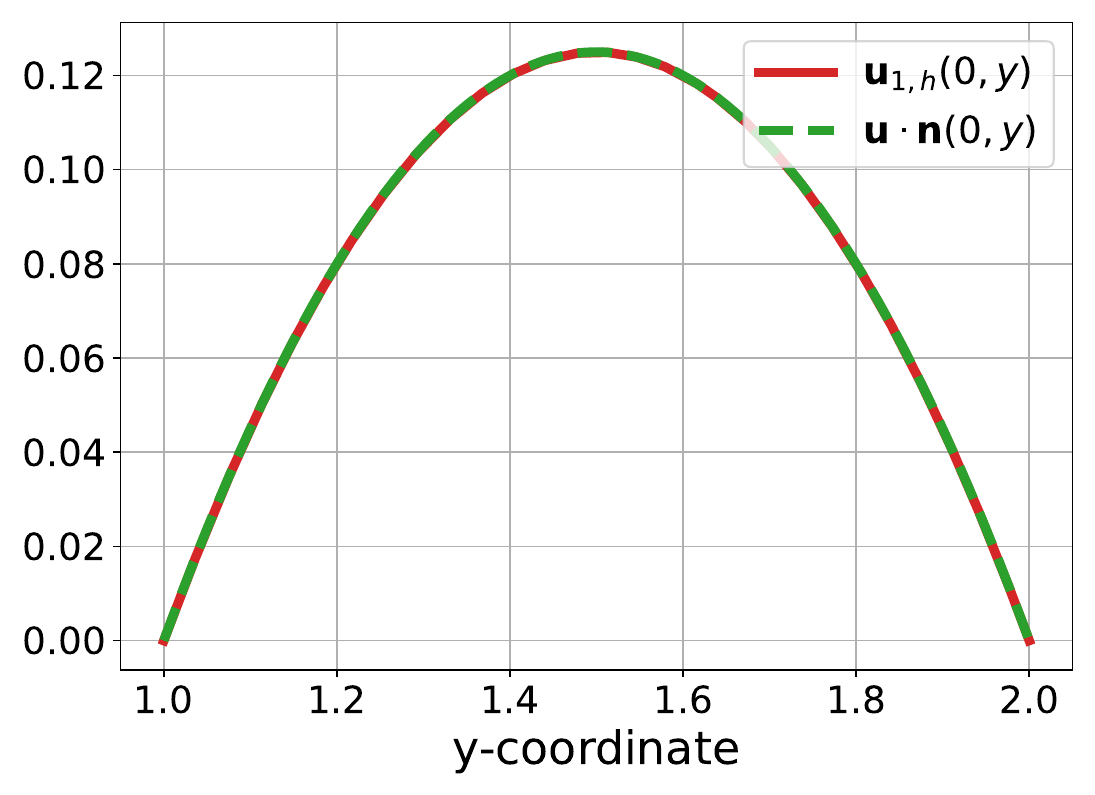}
			\end{minipage}
			\begin{minipage}{0.49\linewidth}\centering
				\includegraphics[scale=0.4]{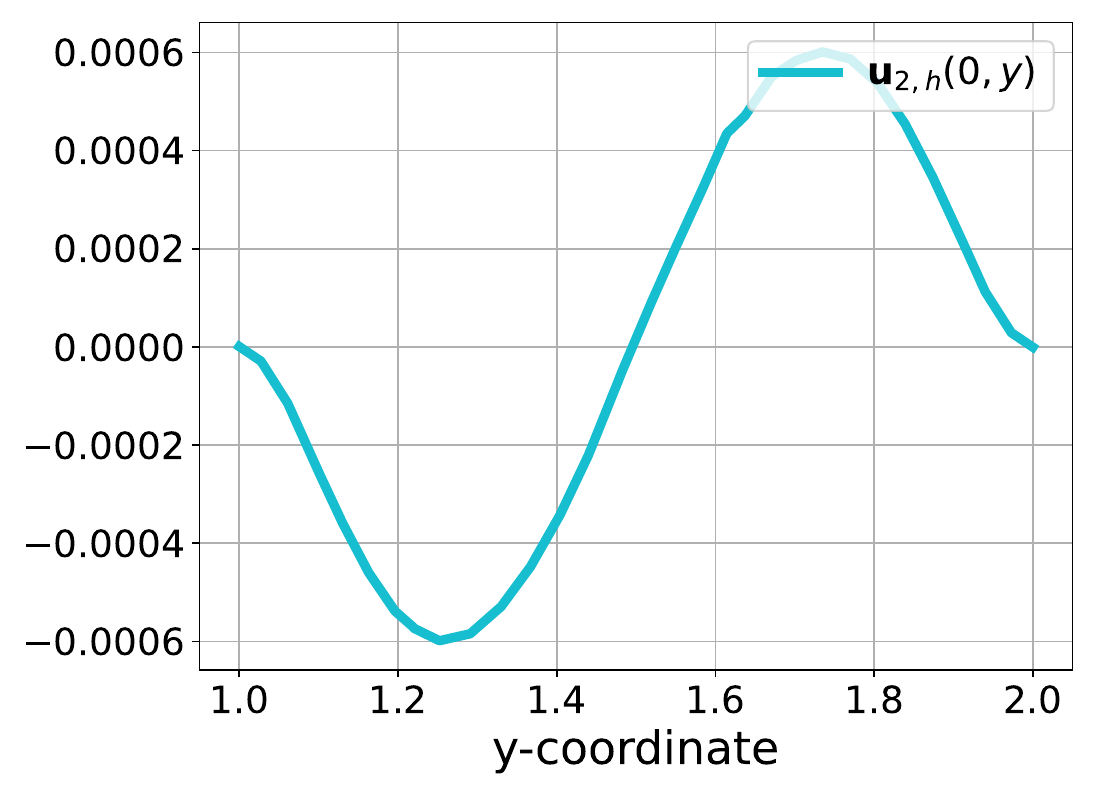}
			\end{minipage}
			\caption{Test \ref{subsec:backward_facing_step}. Plots of the resulting computed inlet velocity components weakly imposed at the inlet for the flow into a backwards facing step. The behavior of $\bu_{2,h}$ results from imposing $(\nu\beps(\bu))\nn\cdot\bt = \boldsymbol{g}_1\cdot\bt$.}
			\label{fig:backward-facing-step-inlet}
		\end{figure}

		\subsubsection{The lid-driven cavity flow}\label{subsec:lid_driven}
		We end this section by performing the classical lid-driven cavity test where we model the steady flow of an immiscible fluid in a box. We consider the unit square domain $\Omega=(0, 1)^2$. The viscosity is taken to be $\nu=10^{-3}$. 		The test is initially done  with $\mathbb{K}=\kappa\mathbb{I}$, with $\kappa=10^8$. We take $\Gamma_N=\emptyset$ and $\Gamma_D=\Gamma_{\text{wall}}\cup\Gamma_{\text{lid}}$, where $\Gamma_{\text{wall}}$ corresponds to the bottom, left and right boundarires, and $\Gamma_{\text{lid}}$ is the top boundary. We set $\bu\vert_{\Gamma_{\text{lid}}}=(1,0)^{\texttt{t}}$ and $\bu\vert_{\Gamma_{\text{wall}}}=\boldsymbol{0}$. We note that this type of boundary condition is also not covered by the theory. 
		
		The approximate velocities and pressure (displayed in Figure \ref{fig:lid_driven_1e8}) remain stable and corner singularities are clearly observed. Moreover, to assess the robustness of the method with respect to the choice of $\mathbb{K}$ , we run several tests varying $\mathbb{K}$ and keeping  $\nu=10^{-3}$. Plots of the streamlines of each case are presented on Figure \ref{fig:lid_driven_robustness}, where we confirm the stability of the approximation for each choice of $\mathbb{K}$.
		\begin{figure}[!hbt]
			\centering
			\begin{minipage}{0.49\linewidth}\centering
				{\footnotesize \hspace{0cm}$\bu_{1,h}$}\\
				\includegraphics[scale=0.14, trim= 12cm 8cm 12cm 6.3cm, clip]{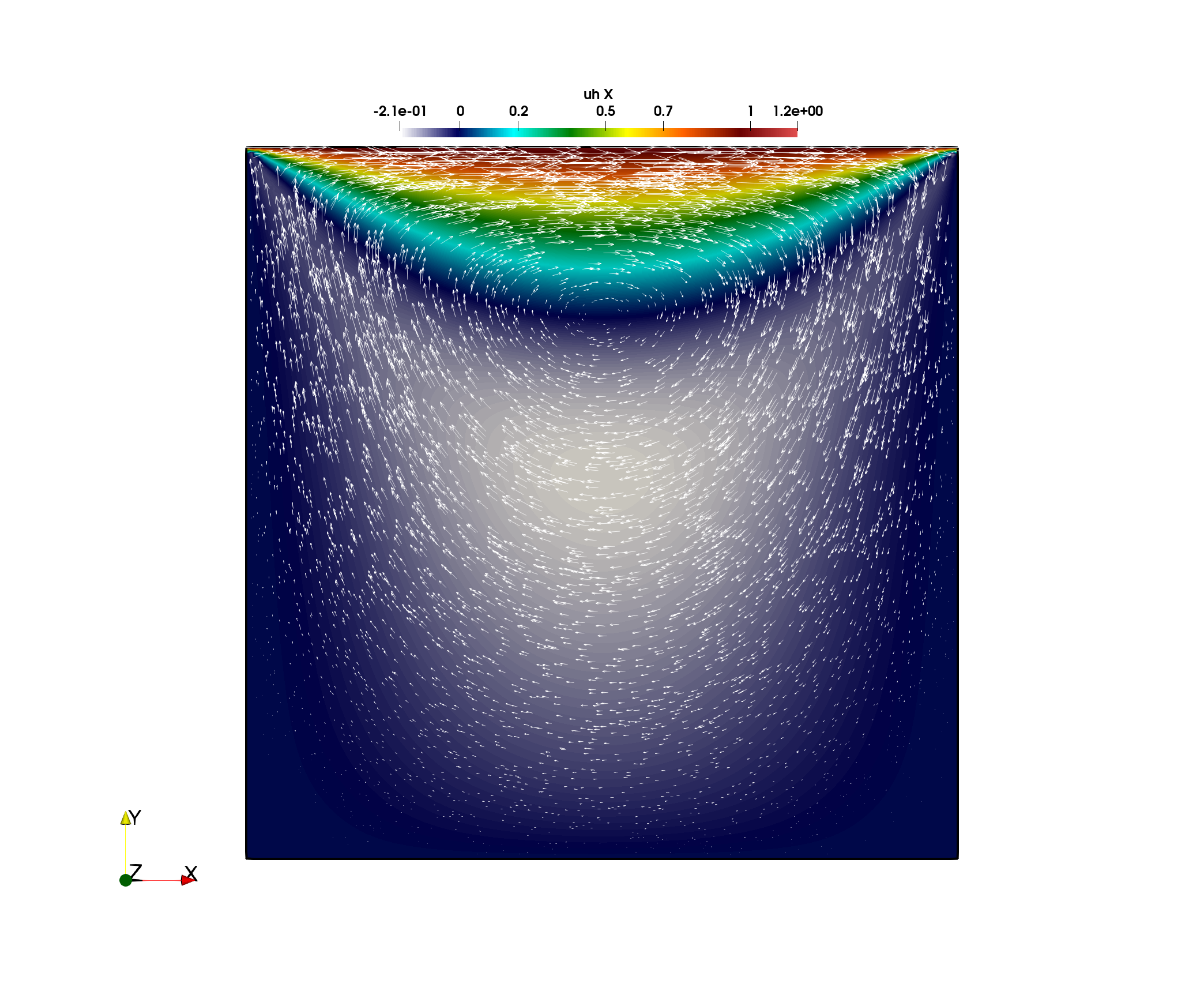}
			\end{minipage}
			\begin{minipage}{0.49\linewidth}\centering
				{\footnotesize \hspace{0cm}$\bu_{2,h}$}\\
				\includegraphics[scale=0.14, trim= 12cm 8cm 12cm 6.3cm, clip]{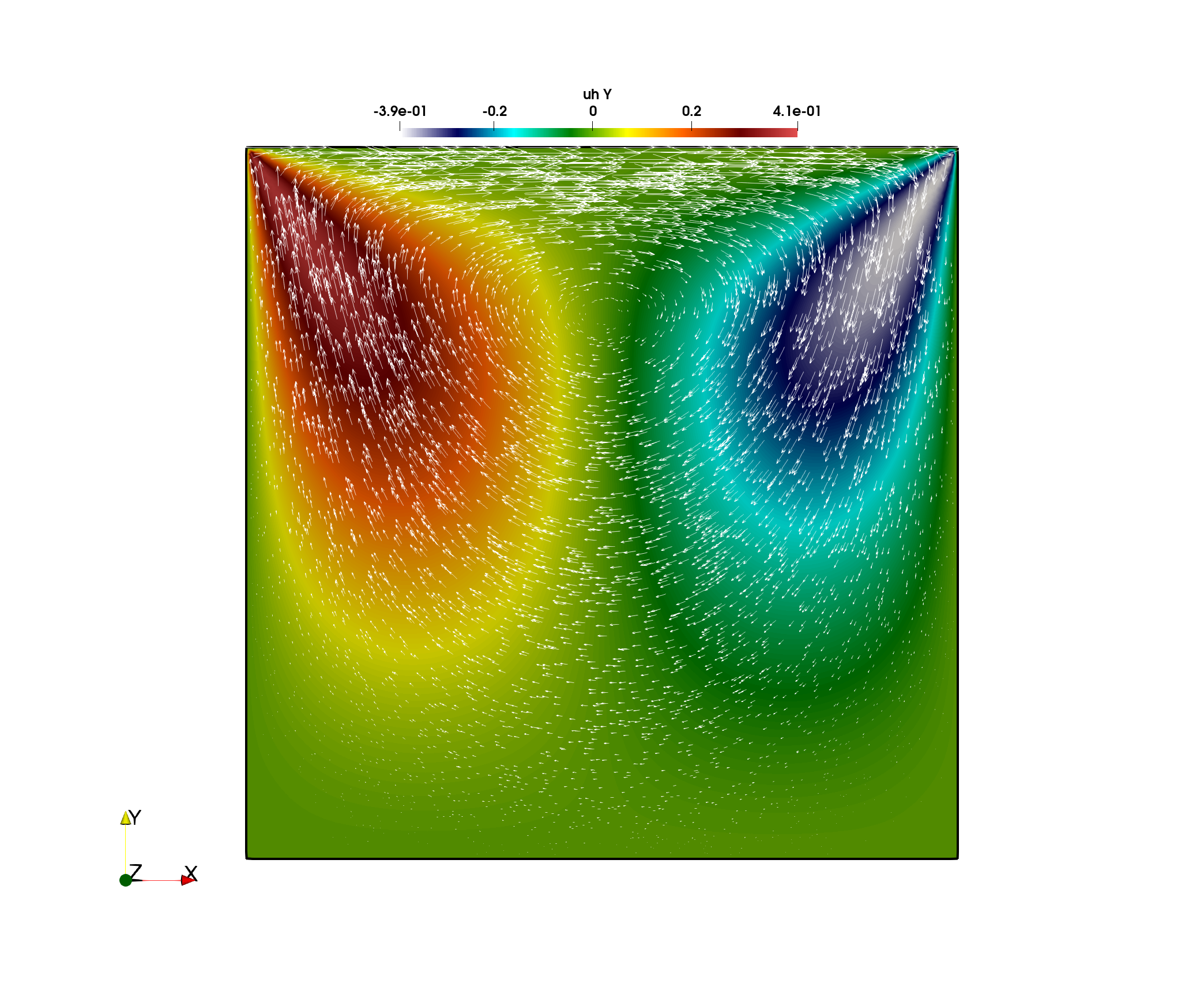}
			\end{minipage}\\
			\begin{minipage}{0.49\linewidth}\centering
				{\footnotesize \hspace{0cm}$|\bu_{h}|$}\\
				\includegraphics[scale=0.14, trim= 12cm 8cm 12cm 6.3cm, clip]{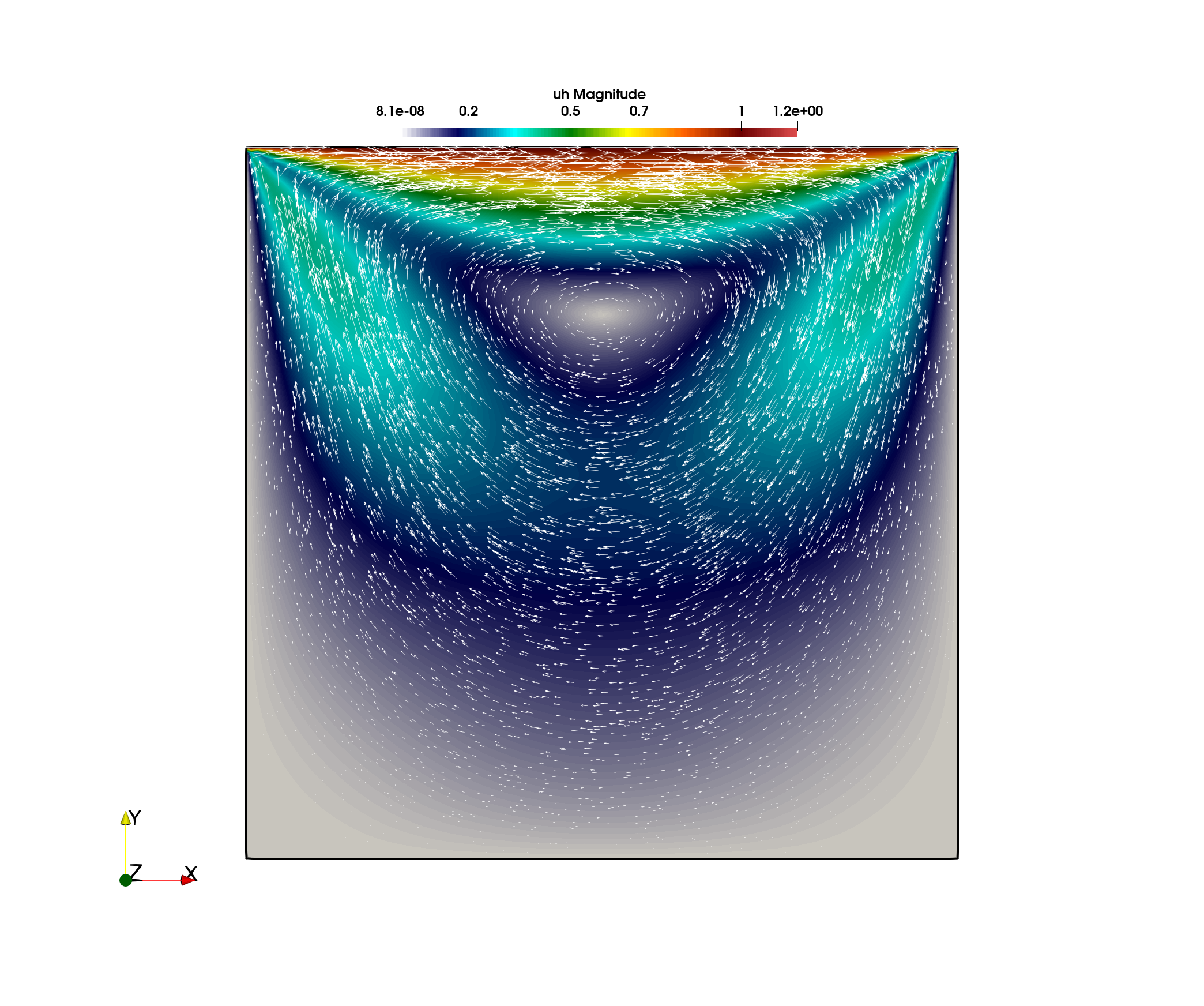}
			\end{minipage}
			\begin{minipage}{0.49\linewidth}\centering
				{\footnotesize \hspace{0cm}$p_h$}\\
				\includegraphics[scale=0.14, trim= 12cm 8cm 12cm 6.3cm, clip]{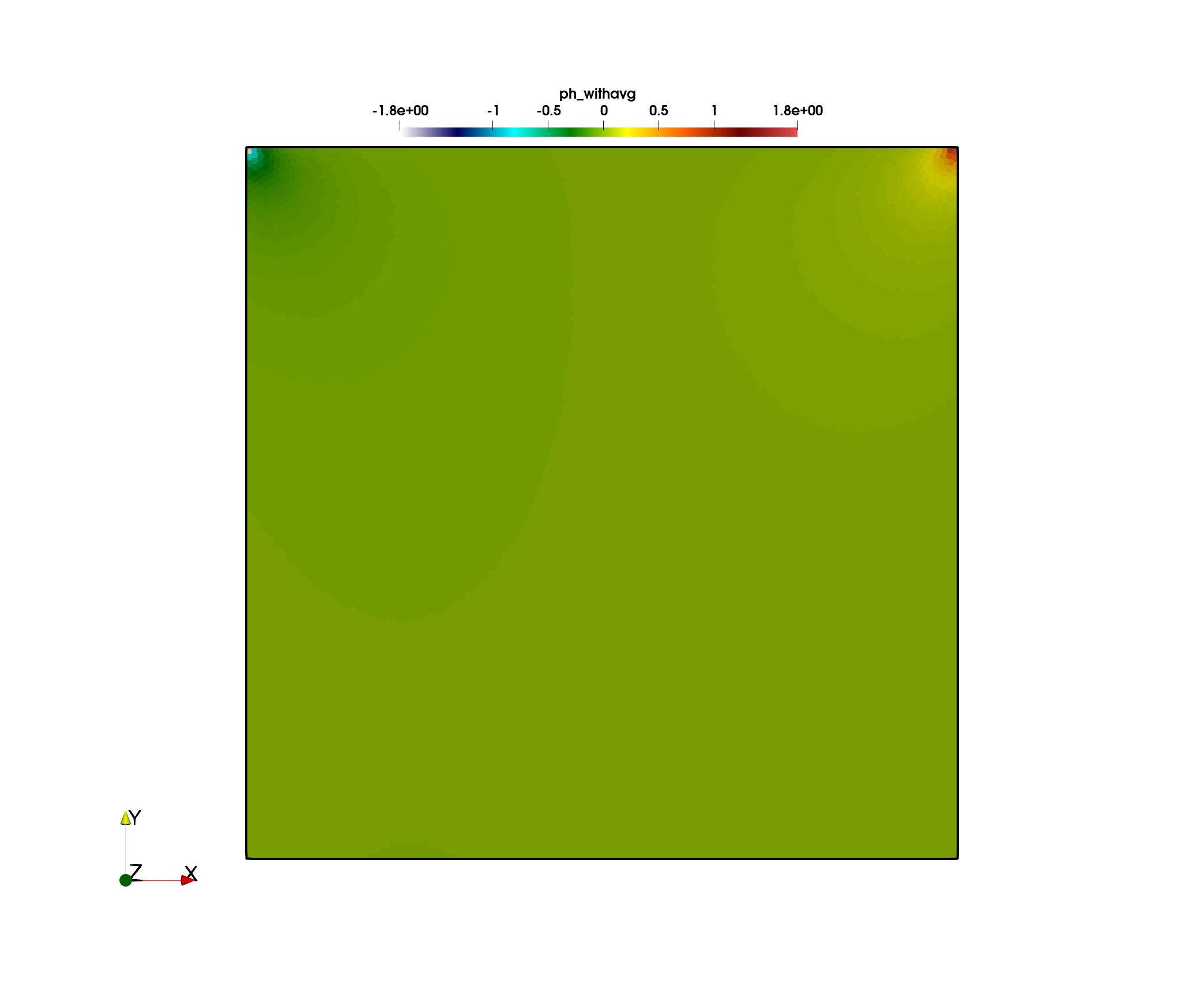}
			\end{minipage}\\
			\caption{Test \ref{subsec:lid_driven}. Approximated velocity components and computed pressure on a Voronoi mesh for the lid-driven cavity test.}
			\label{fig:lid_driven_1e8}
		\end{figure}

		\begin{figure}[!hbt]
			\centering
			\begin{minipage}{0.49\linewidth}\centering
				{\footnotesize \hspace{0cm}$\mathbb{K}^{-1}=10^{-8}\mathbb{I}$}\\
				\includegraphics[scale=0.14, trim= 12cm 8cm 12cm 6.3cm, clip]{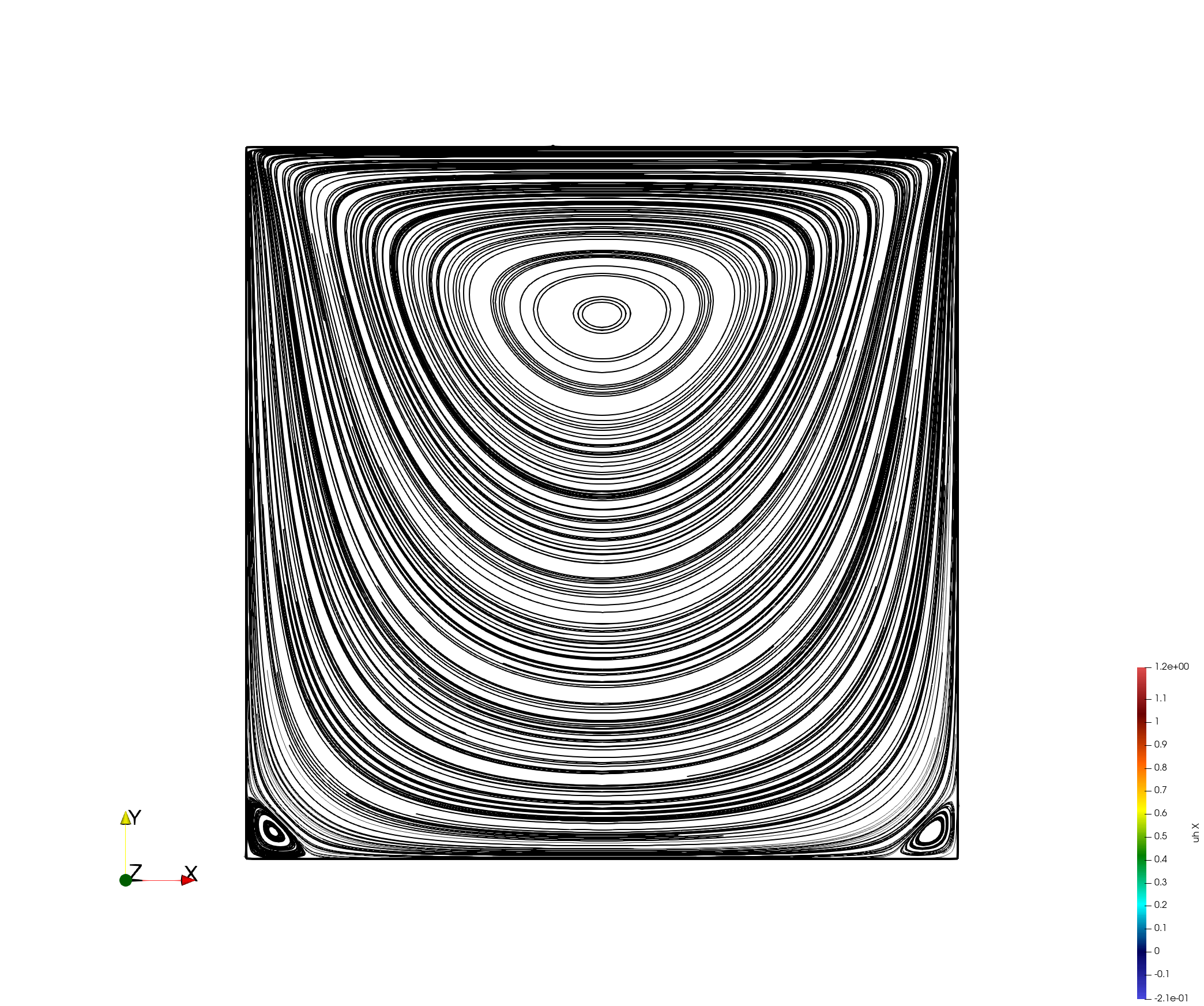}
			\end{minipage}
			\begin{minipage}{0.49\linewidth}\centering
				{\footnotesize \hspace{0cm}$\mathbb{K}^{-1}=10\mathbb{I}$}\\
				\includegraphics[scale=0.14, trim= 12cm 8cm 12cm 6.3cm, clip]{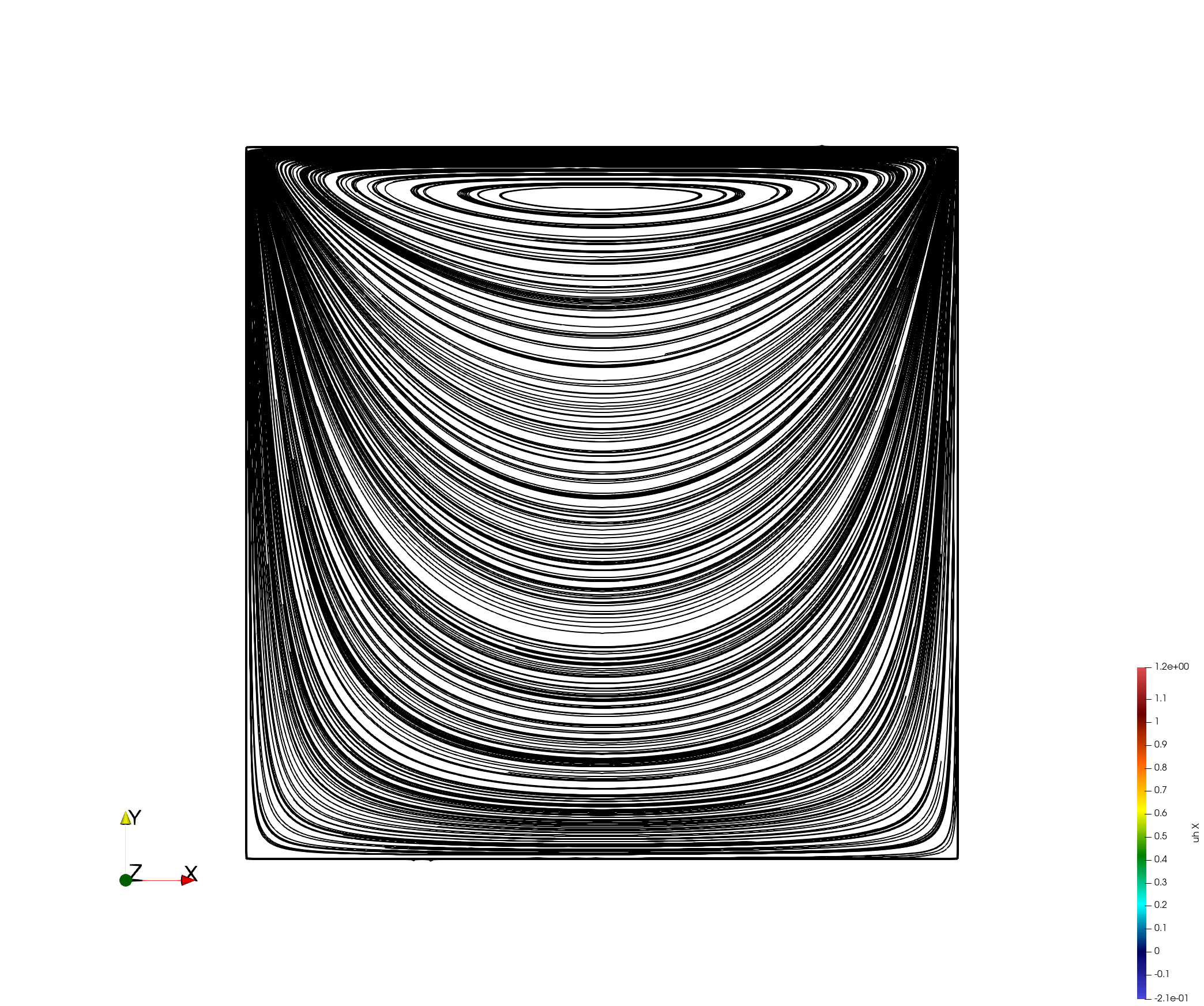}
			\end{minipage}\\
			\begin{minipage}{0.49\linewidth}\centering
				{\footnotesize \hspace{0cm}$\mathbb{K}^{-1}=10^{2}\mathbb{I}$}\\
				\includegraphics[scale=0.14, trim= 12cm 8cm 12cm 6.3cm, clip]{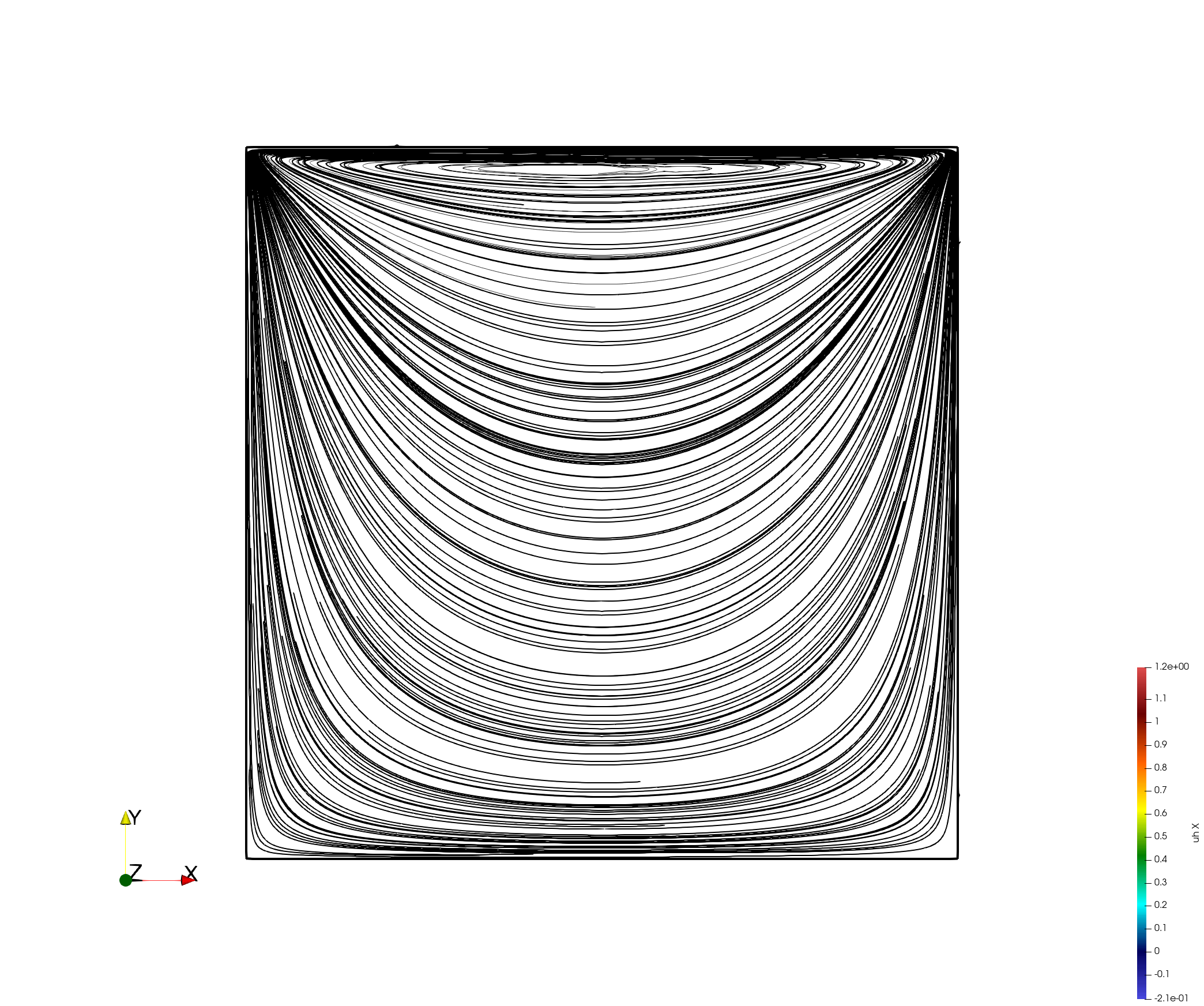}
			\end{minipage}
			\begin{minipage}{0.49\linewidth}\centering
				{\footnotesize \hspace{0cm}$\mathbb{K}^{-1}=10^{8}\mathbb{I}$}\\
				\includegraphics[scale=0.14, trim= 12cm 8cm 12cm 6.3cm, clip]{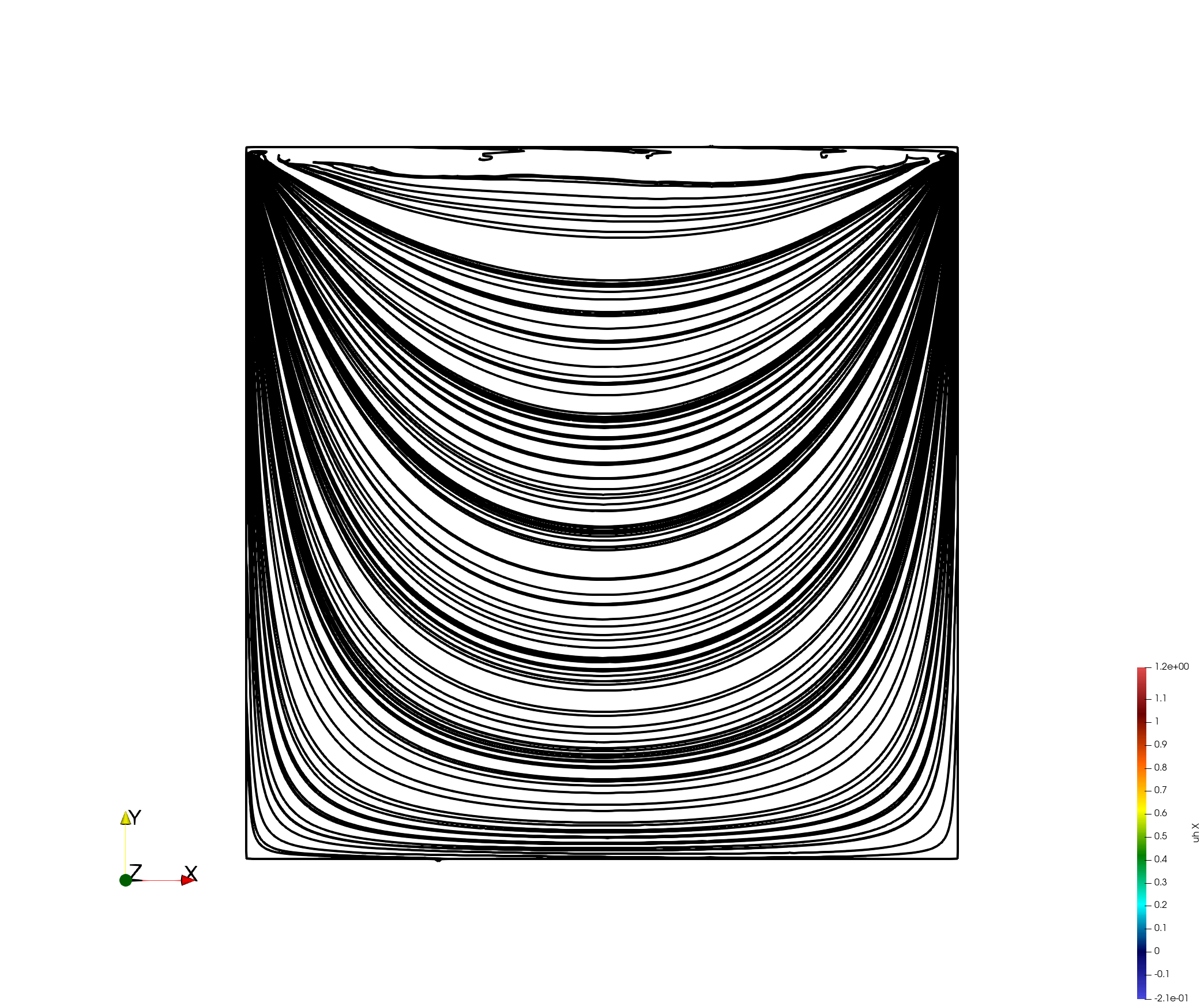}
			\end{minipage}\\
			\caption{Test \ref{subsec:lid_driven}. Comparison of the velocity streamlines obtained for several values of $\mathbb{K}$ for the lid-driven cavity test.}
			\label{fig:lid_driven_robustness}
		\end{figure}
		\begin{acknowledgements}
			The authors have been partially supported by project Centro de Modelamiento Matem\'atico (CMM), FB210005, BASAL funds for centers of excellence, and by the National Agency for Research and Development, ANID-Chile through project Anillo of Computational Mathematics for Desalination Processes ACT210087. The first author was partially supported by ANID-Chile through project FONDECYT 1220881. The second author was partially supported by the National Agency for Research and Development, ANID-Chile through FONDECYT Postdoctorado project 3230302. The third author was partially supported by the National Agency for Research and Development, ANID-Chile through FONDECYT Postdoctorado project 3240737. \\\\
			\bigskip
			\noindent\textbf{Data Availability.} Enquiries about data availability should be directed to the authors.\\
			\noindent\textbf{Conflicts of interest.} The authors have no conflicts of interest to declare.
		\end{acknowledgements}

		\bibliographystyle{siam}
		\bibliography{references}
		
	\end{document}